\documentclass[11pt]{article}

\usepackage{mathrsfs}
\usepackage{kbordermatrix}
\usepackage{enumerate}
\usepackage[section]{algorithm}
\usepackage{algorithmic}
\usepackage{multirow}
\usepackage{mathdots}
\usepackage{xcolor}

\usepackage{graphicx}
\usepackage{amsmath,amsfonts,amsthm,amsbsy,amssymb}
\usepackage{fixmath}
\usepackage{amssymb,latexsym}

\usepackage{hyperref}
\usepackage{cleveref}

\def\be{\pmb{e}}

\def\bh{\pmb{h}}

\def\bu{\pmb{u}}
\def\bv{\pmb{v}}

\def\bx{\pmb{x}}
\def\by{\pmb{y}}
\def\bz{\pmb{z}}

\def\bbC{\mathbb{C}}
\def\bbD{\mathbb{D}}

\def\bbR{\mathbb{R}}

\def\bbT{\mathbb{T}}

\def\scrA{\mathscr{A}}
\def\scrB{\mathscr{B}}

\def\scrG{\mathscr{G}}

\def\scrM{\mathscr{M}}
\def\scrN{\mathscr{N}}
\def\scrP{\mathscr{P}}
\def\scrQ{\mathscr{Q}}

\def\scrU{\mathscr{U}}
\def\scrV{\mathscr{V}}

\def\scrZ{\mathscr{Z}}

\def\cR{{\cal R}}

\def\cX{{\cal X}}

\def\cZ{{\cal Z}}

\def\wtd{\widetilde}
\def\what{\widehat}

\usepackage{accents}

\DeclareMathOperator{\eig}{eig}

\DeclareMathOperator{\diag}{diag}

\DeclareMathOperator{\NRes}{NRes}

\DeclareMathOperator{\rank}{rank}

\DeclareMathOperator{\stb}{s}

\DeclareMathOperator{\F}{F}
\DeclareMathOperator{\HH}{H}
\DeclareMathOperator{\T}{T}

\def\ARE{\mbox{\sc are}}

\def\CARE{\mbox{\sc care}}
\def\DARE{\mbox{\sc dare}}
\def\HARE{\mbox{\sc $H^*$\!are}}

\def\MARE{\mbox{\sc mare}}

\def\NME{\mbox{\sc nme}}

\def\hm{\hphantom{-}}

\newcommand{\toby}[1]{\stackrel {{\scriptstyle #1}}{\to}}
\def\PrP #1{\mbox{\rm P}_{#1}}
\def\rtol{\mbox{\tt rtol}}

\DeclareMathOperator{\rmd}{d}

\newtheorem{theorem}{Theorem}[section]
\newtheorem{lemma}{Lemma}[section]

\theoremstyle{definition}
\newtheorem{definition}{Definition}[section]
\newtheorem{remark}{Remark}[section]

\allowdisplaybreaks

\numberwithin{equation}{section}
\def\sss{\scriptstyle}

\title{A Unifying Framework for Doubling Algorithms}

\author{Changli Liu\thanks{%
     College of Mathematics, Sichuan University, Chengdu 610065, PR China.
     E-mail: {\tt  chliliu@scu.edu.cn}.
     Supported by the National Natural Science Foundation of China (NSFC) No.12571400.
     }
\and Tiexiang Li\thanks{%
     School of Mathematics and Shing-Tung Yau Center, Southeast University, Nanjing 211189, China.
     E-mail: {\tt  txli@seu.edu.cn}.
     Supported by the National Natural Science Foundation of China (NSFC) No.12371377, the Jiangsu Provincial Scientific Research Center of Applied Mathematics under Grant No. BK20233002,
     and by Shanghai Institute for Mathematics and Interdisciplinary Sciences (SIMIS)
     under grant number SIMIS-ID-2024-LG.
     }
\and
Jungong Xue\thanks{%
   School of Mathematical Science,
   Fudan University, Shanghai 200433, China. E-mail: {\tt xuej@fudan.edu.cn}.
   Supported by the National Natural Science Foundation of China (NSFC) No.12571400.
   }
\and
Ren-Cang Li\thanks{%
   Department of Mathematics, University of Texas at Arlington, P.O. Box 19408, Arlington, TX 76019, USA
   E-mail: {\tt rcli@uta.edu}.
   Supported in part by US NSF DMS-2407692.
   }
\and Wen-Wei Lin\thanks{%
   Shanghai Institute for Mathematics and Interdisciplinary Sciences, Shanghai 200433, China.
   E-mail: {\tt wwlin@outlook.com}.
   Supported by Shanghai Institute for Mathematics and Interdisciplinary Sciences (SIMIS)
   under grant number SIMIS-ID-2024-LG.
   }
}

\date{%May 21, 2025
January 31, 2026
}

\begin{document}

\maketitle

\begin{abstract}
The existing doubling algorithms have been proven efficient for several important nonlinear matrix equations arising from real-world engineering applications.
In a nutshell, the algorithms iteratively compute a basis matrix, in one of the two particular forms, for the eigenspace of some matrix pencil associated with its eigenvalues in certain complex region such as the left-half plane or the open unit disk, and their success  critically depends on
that the interested eigenspace do have a basis matrix taking one of the two particular forms.
However, that requirement in general cannot be guaranteed. In this paper, a new doubling algorithm,
called the $Q$-doubling algorithm, is proposed. It includes
the existing doubling algorithms as special cases
and does not require that the basis matrix takes one of the particular forms.
An application of the $Q$-doubling algorithm to
solve eigenvalue problems is investigated with numerical experiments that demonstrate
its superior robustness to the existing doubling algorithms.


\end{abstract}

\clearpage
\tableofcontents

\clearpage
\section{Introduction}\label{sec:intro}
In \cite{hull:2018}, a general framework and theory
for two types of structure-preserving doubling algorithms were developed
for solving nonlinear matrix equations in association
with the eigenspaces of certain regular matrix pencils. The development
in the book drew inspiration from early works
in \cite{biim:2012,chfl:2005,chfl:2004} for the continuous-time  algebraic Riccati equation (\CARE)  and discrete-time  algebraic Riccati equation (\DARE),
and  in \cite{bimp:2010,gulx:2006,wawl:2012}
for the $M$-Matrix algebraic Riccati equation (\MARE). In a nutshell
the kinds of nonlinear matrix equations for which the doubling algorithms \cite{hull:2018}
were created can be recasted equivalently into
\iffalse
 take one of the following forms:
\begin{subequations}\label{eq:Alg-Ricc-Eq'}
\begin{align}
BX(I+GX)^{-1}A+H-X&=0, \label{eq:D-Alg-Ricc-Eq'}\\
XDX-AX-XB+C&=0, \label{eq:C-Alg-Ricc-Eq'} \\
X+BX^{-1}A&=Q, \label{eq:type-II'} \\
AX^2+BX+C&=0. \label{eq:quad-eqs}
\end{align}
\end{subequations}
A key commonality among  them is that their solutions $X$ relate to
basis matrices, in the form of
\begin{equation}\label{eq:X2basis}
Z=\begin{bmatrix}
    I \\
    X
  \end{bmatrix}
%\kbordermatrix{ & \sss m  \cr
%                \sss m & I_m \cr
%                \sss n & X}
%\quad (\mbox{$X$ is $n\times m$}),
\end{equation}
of  eigenspaces  of certain regular matrix pencils $\scrA-\lambda\scrB$.
This fact certainly is well-known
for \ARE\ (\ref{eq:Alg-Ricc-Eq'}a,b)
\cite{laro:1995} and has been well-exploited in the past for solving \DARE s and \CARE s
\cite{laro:1995,laub:1979,laub:1991,mehr:1991,pals:1980} and \MARE s  \cite{gulx:2006,ngpo:2015,wawl:2012,xuli:2017}.

It turns out that all four equations in \eqref{eq:Alg-Ricc-Eq'} can be recasted as
\fi
some eigenvalue problems of certain matrix pencils. For example, \DARE\
\begin{equation}\label{eq:DARE}
BX(I+GX)^{-1}A+H-X=0
\end{equation}
is equivalent to
\begin{equation}\label{eq:DARE2eig-1}
\underbrace{\begin{bmatrix}
              A & 0 \\
              -H & I
            \end{bmatrix}}_{:=\scrA}\begin{bmatrix}
  I \\
  X
\end{bmatrix}
=\underbrace{\begin{bmatrix}
              I & G \\
              0 & B
          \end{bmatrix}}_{:=\scrB}\begin{bmatrix}
  I \\
  X
\end{bmatrix}M,
\end{equation}
where $M=(I+GX)^{-1}A$. Two things can be read off from \eqref{eq:DARE2eig-1}:
1) solving \DARE\ is the same as computing an eigenspace of the associated matrix pencil
$\scrA-\lambda \scrB$,
and 2) the eigenspace has a basis matrix taking the form
\begin{equation}\label{eq:basis-form1}
Z=\begin{bmatrix}
                                    I \\
                                    X
                                  \end{bmatrix}.
\end{equation}
That there is a basis matrix having this particular form, in general, cannot be guaranteed
and has to be justified on a case-by-case basis, e.g., for \DARE\ \eqref{eq:DARE}
from optimal control,
this is justifiable (see \cite[chapter~4]{hull:2018} and references therein).
On the other hand, even if there is a basis matrix of such a form, sometimes $X$ could have
huge magnitude, potentially creating numerical instability.
How to address such an issue effectively is one of our motivations in this paper.

%the eigenspace always has a basis matrix in the form $Q^{\T}\begin{bmatrix}
%                                    I \\
%                                    X
%                                  \end{bmatrix}$, where $Q$ is some permutation matrix.
%This observation lays the key theoretic foundation for what we will do in this paper.
%Most discussion will be forthcoming.

%It is the second observation that warrants extra attention, i.e., in general
%
% This implies that
%the column space of the matrix $Z=\begin{bmatrix}
%                                    I \\
%                                    X
%                                  \end{bmatrix}$, denoted by $\cR(Z)$ is an eigenspace
%of matrix pencil $\scrA-\lambda \scrB$, associated with its eigenvalues
%lying the unit disk in the application of optimal control.

It follows from \eqref{eq:DARE2eig-1} that the eigenvalues of $M$ are part of those of
$\scrA-\lambda \scrB$.
%
%Evidently these eigenvalues are the same as
%those of $M$ which becomes known once the \DARE\ is solved.
From the perspective of eigenvalue computation,
solving \DARE\ \eqref{eq:DARE} decouples the matrix pencil. In fact, we will have
\begin{equation}\label{eq:DARE2eig-2}
\begin{bmatrix}
              A & 0 \\
              -H & I
            \end{bmatrix}
\begin{bmatrix}
  I & 0\\
  X & I
\end{bmatrix}
=\begin{bmatrix}
              I & G \\
              0 & B
          \end{bmatrix}
\begin{bmatrix}
  I & 0\\
  X & I
\end{bmatrix}\begin{bmatrix}
               M & -(I+GX)^{-1}GM_2^{-1}\\
               0 & M_2^{-1}
             \end{bmatrix},
\end{equation}
where $M_2=B[I-X(I+GX)^{-1}G]$, and as a result, the spectrum of the matrix pencil
$\scrA-\lambda \scrB$
is the multiset union $\eig(M)\cup\{\mu^{-1}\,:\,\mu\in\eig(M_2)\}$, where $\eig(\cdot)$
denotes the set of the eigenvalues of a square matrix.
and suppose that we know an $m$-dimensional eigenspace $\cX$ of it in terms of
some $(m+n)\times m$ basis matrix $Z\in\bbC^{(m+n)\times m}$, which can always be written as
% The idea is to compute a basis matrix
\begin{equation}\label{eq:BaseM}
Z\equiv\kbordermatrix{ & \sss m  \cr
                \sss m & Z_1 \cr
                \sss n & Z_2}.
\end{equation}
If also $Z_1$ is invertible, then
$
ZZ_1^{-1}=\begin{bmatrix}
            I \\
            Z_2Z_1^{-1}
          \end{bmatrix}=:\begin{bmatrix}
  I \\
  X
\end{bmatrix}
$
is also a basis matrix which is in the form \eqref{eq:basis-form1},
and
$\scrA\begin{bmatrix}
  I \\
  X
\end{bmatrix}=\scrB\begin{bmatrix}
  I \\
  X
\end{bmatrix}M$
%\begin{equation}\label{eq:calX=ES}
%%\scrA \begin{bmatrix}
%%  I \\
%%  X
%%\end{bmatrix}
%%\equiv
%\begin{bmatrix}
%        A_{11} & A_{12} \\
%        A_{21} & A_{22}
%      \end{bmatrix}\begin{bmatrix}
%  I \\
%  X
%\end{bmatrix}
%=\begin{bmatrix}
%   B_{11} & B_{12} \\
%   B_{21} & B_{22}
% \end{bmatrix}\begin{bmatrix}
%  I \\
%  X
%\end{bmatrix}M
%%\equiv\scrB\begin{bmatrix}
%%  I \\
%%  X
%%\end{bmatrix} M
%\end{equation}
for some $m\times m$ matrix $M$. Detail can be worked out to obtain an equation like
\eqref{eq:DARE2eig-2} that decouples $\scrA-\lambda\scrB$ \cite[chapter~3]{hull:2018}.
Once again, it needs the nonsingularity of $Z_1$, which provably holds for
nonlinear matrix equations from practical applications such as optimal control, nano research,
and fast trains \cite{hull:2018}, but may fail for general matrix pencils.
In the latter case, the $m$-dimensional eigenspace does not have a basis matrix
in the form \eqref{eq:basis-form1} and thus the existing doubling algorithms \cite{hull:2018} are no longer applicable.
Even if $Z_1$ is invertible in theory, computational difficulties can and will arise
when $Z_1$ is nearly singular.

However, since $\rank(Z)=m$, $Z$ has a $m$-by-$m$ submatrix that is invertible. This means that
there is an all-weather representation for the basis matrix of $\cX$ as
\begin{equation}\label{eq:all-weather-basis}
Q^{\T}\begin{bmatrix}
  I \\
  X
\end{bmatrix}
\quad\mbox{with a permutation matrix $Q\in\bbR^{(m+n)\times (m+n)}$}.
\end{equation}
If $Q$ is chosen properly, the magnitude can be made modest as we will show later in this paper.
Unfortunately the doubling algorithms \cite{hull:2018}
are not designed to work with basis matrices in such a form.
On the other hand, $Q$ is in general not
known\footnote {Otherwise we can work with
  $\scrA Q-\lambda \scrB Q$ which will have an eigenspace with basis matrix of the form $\begin{bmatrix}
  I \\
  X
  \end{bmatrix}$.
  The doubling algorithms in \cite{hull:2018} can then be tried on $\scrA Q-\lambda \scrB Q$.}
{\em a prior}. Hence, any iterative procedure to compute a basis matrix
in the form \eqref{eq:all-weather-basis} must be able to  update $Q$ dynamically to
prevent the magnitude of $X$ from growing too big.
The goals of this paper are two-fold: 1) to develop a doubling algorithm framework
that unifies the ones in \cite{hull:2018}, and
2) to design an mechanism to update $Q$ dynamically so that $\|X\|$ is kept modest.
To that end, we will propose {\em the $Q$-standard form\/} (SFQ)
that unifies the two standard forms SF1 and SF2 in \cite{hull:2018}.

The rest of this paper is organized as follows. \Cref{sec:overview} provides an overview of the
framework of the existing doubling algorithms. \Cref{sec:new-basis} explains the need of changing basis
and the rationale of a new doubling algorithm. \Cref{sec:new-frame} presents our new structure-preserving
doubling algorithm, SDASFQ, with the new basis matrix representation such as \eqref{eq:all-weather-basis}. As for the existing doubling algorithms, the new doubling algorithm is also associated with nonlinear matrix equations of some sort,
which is done in \cref{sec:SF1SF2d}. Convergence analysis
is performed in \cref{sec:DA-conv-regu} for the so-called regular case
and in appendix~\ref{sec:DA-conv-crit} for the critical cases due to its mathematical complexity.
We present our complete QDA in \cref{sec:QDA} which including initialization and dynamically updating the $Q$-matrices in the basis matrix representation. An application of QDA to the eigenvalue problem
is investigated in \cref{sec:EigComp} and two numerical experiments for it are reported in
\cref{sec:egs}. Finally, we draw our conclusions in \cref{sec:conclu}.

{\bf Notation.}
$\bbC^{n\times m}$ is the set
of all $n\times m$ complex matrices, $\bbC^n=\bbC^{n\times 1}$,
and $\bbC=\bbC^1$. $I_n\in\bbC^{n\times n}$ (or simply $I$ if its dimension is
clear from the context) is the $n\times n$ identity matrix, and $\be_j$ is the $j$th column
of $I$ (of an apt size).
$A^{\T}$ and $A^{\HH}$  are the transpose and the complex conjugate transpose of a matrix $A$, respectively.
$\cR(Z)$ denotes the column subspace, spanned by the columns of $Z$.
$\|\cdot\|$  denotes a generic matrix/vector norm, and
$\|\cdot\|_2$ is the $\ell_2$-vector norm or the spectral norm of a matrix, and $\|\cdot\|_{\F}$ is the matrix Frobenius norm. For $r>0$, define
$$
\bbD_{r-}=\{z\in\bbC\,:\,|z|<r\}, \quad
\bbT_r=\{z\in\bbC\,:\,|z|=r\}, \quad
\bbD_{r+}=\{z\in\bbC\,:\,|z|>r\},
$$
and $\bar\bbD_{r-}$ and $\bar\bbD_{r+}$ are the closures of $\bbD_{r-}$ and $\bbD_{r+}$, respectively.
By default, $\bbD_-=\bbD_{1-}$, $\bbT=\bbT_1$, and $\bbD_+=\bbD_{1+}$. Finally,
$\bbC_-=\{z\in\bbC\,:\,\Re(z)<0\}\}$ where $\Re(z)$ takes the real part of $z\in\bbC$ and similarly
$\bbC_+=\{z\in\bbC\,:\,\Re(z)>0\}\}$.

%Another serious shortcoming is that
%the approach often has hard time in preserving any symmetry in $X$, especially when $Z_1$ is ill-conditioned.
%Recent researches
%have shown that structure-preserving doubling algorithms have their superiority  in computational efficiency and accuracy
%in the aforementioned applications.

\section{Overview of existing doubling algorithms}\label{sec:overview}
The general framework
for the two types of structure-preserving doubling algorithms -- SDASF1 \cite[p.22]{hull:2018} and SDASF2 \cite[p.24]{hull:2018} -- goes as follows. Given matrix pencil
$\scrA-\lambda\scrB$:
\begin{equation}\label{eq:scrA-scrB}
\scrA=\kbordermatrix{ & \sss m & \sss n \cr
                \sss m & A_{11} & A_{12} \cr
                \sss n & A_{21} & A_{22}}, \quad
\scrB=\kbordermatrix{ & \sss m & \sss n \cr
                \sss m & B_{11} & B_{12} \cr
                \sss n & B_{21} & B_{22}},
\end{equation}
we are interested in solving
%First,
%the underlying nonlinear matrix equation is converted into an eigenvalue problem
\begin{equation}\label{eq:calX=ES:c2}
\scrA \begin{bmatrix}
  I \\
  X
\end{bmatrix}
\equiv\begin{bmatrix}
        A_{11} & A_{12} \\
        A_{21} & A_{22}
      \end{bmatrix}\begin{bmatrix}
  I \\
  X
\end{bmatrix}
=\begin{bmatrix}
   B_{11} & B_{12} \\
   B_{21} & B_{22}
 \end{bmatrix}\begin{bmatrix}
  I \\
  X
\end{bmatrix}M
\equiv\scrB\begin{bmatrix}
  I \\
  X
\end{bmatrix} M
\end{equation}
for  $X\in\bbC^{n\times m}$, where $M\in\bbC^{m\times m}$ can be
post-determined once $X$ is found. This is an eigenvalue problem and the eigenvalues
are part of those of $\scrA-\lambda\scrB$, lying in some regions in the complex plane such as
the left-half plane $\bbC_-$ or the open unit disk $\bbD_-$.
%, and outside the unit circle,  etc.
To ensure
the convergence of a doubling algorithm later, making $\eig(M)$ lying in $\bbD_-$ is often needed, but when this is not the case,
$\scrA-\lambda\scrB$ has to undergo some transformation to yield another pencil
\begin{subequations}\label{eq:trans:scrAB2scrB'}
\begin{equation}\label{eq:trans:scrAB2scrB'-1}
\scrA-\lambda\scrB\,\to\,\scrA'-\lambda\scrB',
\end{equation}
e.g., the commonly used M{\" o}bius transformation for \CARE\ \cite{lixu:2006} and \MARE\ \cite{gulx:2006}
and the generalized two-parameter M{\" o}bius transformation for \MARE\ \cite{wawl:2012}
(see also \cite{hull:2018}).
The new pencil $\scrA'-\lambda\scrB'$ ought to satisfy
\begin{equation}\label{eq:trans:scrAB2scrB'-2}
\scrA' \begin{bmatrix}
  I \\
  X
\end{bmatrix}
=\scrB'\begin{bmatrix}
  I \\
  X
\end{bmatrix} \scrM,\,\,
\eig(\scrM)\subset\bbD_-\,\,
\mbox{or even}\,\,\eig(\scrM)\subset\bar\bbD_-,
\end{equation}
\end{subequations}
where %$\bbD_-$ denotes the open unit disk and
$\bar\bbD_-$ is the closure of $\bbD_-$, i.e., the closed unit disk.
%One such transformation is the {\em M{\" o}bius transformation\/}
%\cite[section~3.6]{hull:2018}.

Suppose now that such a transformation \eqref{eq:trans:scrAB2scrB'} has been performed.
What comes next is to pre-multiply $\scrA'-\lambda\scrB'$ by a nonsingular matrix
$P\in\bbC^{(m+n)\times (m+n)}$
\begin{equation}\label{eq:trans2init}
\scrA_0-\lambda\scrB_0=P(\scrA'-\lambda\scrB')
\end{equation}
to transform $\scrA'-\lambda\scrB'$ into $\scrA_0-\lambda\scrB_0$ with
$\scrA_0$ and $\scrB_0$ taking one of the two standard forms:
{\bf the First Standard Form} (SF1) as
\begin{equation}\tag{SF1}
\scrA_0=\kbordermatrix{ &\sss m & \sss n\cr
      \sss m & \hm E_0 & 0 \cr
      \sss n & -X_0 & I}, \quad
\scrB_0=\kbordermatrix{ &\sss m & \sss n\cr
      \sss m & I & -Y_0 \cr
      \sss n & 0 & \hm F_0},
\end{equation}
or {\bf the Second Standard Form} (SF2), which requires $m=n$, as
\begin{equation}\tag{SF2}
\scrA_0=\kbordermatrix{ &\sss n & \sss n\cr
      \sss n &  \hm E_0 &  0 \cr
      \sss n &  -X_0 & I}, \quad
\scrB_0=\kbordermatrix{ &\sss n & \sss n\cr
      \sss n & -Y_0 & I  \cr
      \sss n &  \hm F_0  & 0  }.
\end{equation}
It follows from \eqref{eq:trans:scrAB2scrB'} and \eqref{eq:trans2init} that also
\begin{equation}\label{eq:trans:scrAB2scrB0}
\scrA_0 \begin{bmatrix}
  I \\
  X
\end{bmatrix}
=\scrB_0\begin{bmatrix}
  I \\
  X
\end{bmatrix} \scrM,\,\,
\eig(\scrM)\subset\bbD_-\,\,
\mbox{or even}\,\,\eig(\scrM)\subset\bar\bbD_-.
\end{equation}
Two respective doubling algorithms, SDASF1 \cite[p.22]{hull:2018} and SDASF2 \cite[p.24]{hull:2018},
are then created and each will generate a sequence of matrix pencils
$\scrA_i-\lambda \scrB_i$ while preserving the block structure in their corresponding $\scrA_0-\lambda \scrB_0$, i.e.,
\begin{subequations}\label{eq:preservation-SF1SF2}
\begin{align}
\scrA_i=\kbordermatrix{ &\sss m & \sss n\cr
      \sss m & \hm E_i & 0 \cr
      \sss n & -X_i & I}, \quad
\scrB_i=\kbordermatrix{ &\sss m & \sss n\cr
      \sss m & I & -Y_i \cr
      \sss n & 0 & \hm F_i} \quad&\mbox{for SF1}; \label{eq:preservation-SF1}\\
\scrA_i=\kbordermatrix{ &\sss n & \sss n\cr
      \sss n &  \hm E_i &  0 \cr
      \sss n &  -X_i & I}, \quad
\scrB_i=\kbordermatrix{ &\sss n & \sss n\cr
      \sss n & -Y_i & I  \cr
      \sss n &  \hm F_i  & 0  }\quad&\mbox{for SF2}. \label{eq:preservation-SF2}
\end{align}
\end{subequations}

In general, for one of the two doubling algorithms to work, the spectrum of the associated matrix pencil
$\scrA-\lambda\scrB$
is split into two parts, one of which corresponds to $X$ in the sense of \eqref{eq:calX=ES:c2}
and the other corresponds to $Y\in\bbC^{m\times n}$ in the sense of
\begin{equation}\label{eq:calY=ES}
\mbox{either}\quad
\scrA \begin{bmatrix}
  Y \\
  I
\end{bmatrix} N
=\scrB\begin{bmatrix}
  Y \\
  I
\end{bmatrix}
\,\,\mbox{for SF1, or}\,\,
\scrA \begin{bmatrix}
  I \\
  Y
\end{bmatrix} N
=\scrB\begin{bmatrix}
  I \\
  Y
\end{bmatrix}\,\,\mbox{for SF2}.
\end{equation}
The spectrum of $\scrA-\lambda\scrB$ is the multiset union $\eig(M)\cup\eig(N)$.
Upon performing aforementioned transformations from $\scrA-\lambda\scrB$ to $\scrA_0-\lambda B_0$,
%similarly to \eqref{eq:trans:scrAB2scrB0},
\eqref{eq:calY=ES} leads to,
as a counterpart,
\begin{equation}\label{eq:calY=ES'}
\mbox{either}\quad
\scrA_0 \begin{bmatrix}
  Y \\
  I
\end{bmatrix} \scrN
=\scrB_0\begin{bmatrix}
  Y \\
  I
\end{bmatrix}
\,\,\mbox{for SF1, or}\,\,
\scrA_0 \begin{bmatrix}
  I \\
  Y
\end{bmatrix} \scrN
=\scrB_0\begin{bmatrix}
  I \\
  Y
\end{bmatrix}\,\,\mbox{for SF2},
\end{equation}
where
$\eig(\scrN)\subset\bbD_-$
or even $\eig(\scrN)\subset\bar\bbD_-$.
The spectrum of $\scrA_0-\lambda B_0$ is the multiset union
$\eig(\scrM)\cup\{\mu^{-1}\,:\,\mu\in\eig(\scrN)\}$.

For both SF1 and SF2, the sequence of matrix pencils
$\scrA_i-\lambda \scrB_i$ generated by their respective doubling algorithms satisfies
\begin{gather*}
\scrA_i \begin{bmatrix}
  I \\
  X
\end{bmatrix}
=\scrB_i\begin{bmatrix}
  I \\
  X
\end{bmatrix} \scrM^{2^i}, \\
\mbox{either}\quad
\scrA_i \begin{bmatrix}
  Y \\
  I
\end{bmatrix} \scrN^{2^i}
=\scrB_i\begin{bmatrix}
  Y \\
  I
\end{bmatrix}
\,\,\mbox{for SF1, or}\,\,
\scrA_i \begin{bmatrix}
  I \\
  Y
\end{bmatrix} \scrN^{2^i}
=\scrB_i\begin{bmatrix}
  I \\
  Y
\end{bmatrix}\,\,\mbox{for SF2}.
\end{gather*}
Convergence of $X_i$ in \eqref{eq:preservation-SF1SF2} to the solution $X$ in \eqref{eq:calX=ES:c2} and, as a by-product, $Y_i$ to the solution $Y$ in \eqref{eq:calY=ES}
is then evident when both $\rho(\scrM)<1$ and $\rho(\scrN)<1$. But more deeply,
the convergence is assured if $\rho(\scrM)\rho(\scrN)<1$ or even $\rho(\scrM)\rho(\scrN)=1$
with an additional condition on the Jordan structure of $\scrA_0-\lambda\scrB_0$
(the so-called critical case) \cite[section~3.8]{hull:2018}. To get an intuitive understanding of the convergence,
let us say $\rho(\scrM)<1$ and $\rho(\scrN)<1$. Then both $\scrM^{2^i}\to 0$ and $\scrN^{2^i}\to 0$ and hence
\begin{equation}\label{eq:cvg-SDASF1SF2}
\scrA_i \begin{bmatrix}
  I \\
  X
\end{bmatrix}\to 0,\,\,
\mbox{and either}\,\,
\scrB_i\begin{bmatrix}
  Y \\
  I
\end{bmatrix}\to 0
\,\,\mbox{for SF1 or}\,\,
\scrB_i\begin{bmatrix}
  I \\
  Y
\end{bmatrix}\to 0\,\,\mbox{for SF2},
\end{equation}
yielding $X-X_i\to 0$ and $Y-Y_i\to 0$ because of the special structures in $\scrA_i$ and $\scrB_i$ in \eqref{eq:preservation-SF1SF2}.

%We observe in passing that \eqref{eq:trans:scrAB2scrB'-2} remains true with
%$\scrA'-\lambda\scrB'$ replaced by $\scrA_0-\lambda\scrB_0$. This will become important later when we discuss our
%primal-dual view towards the doubling algorithms.

Theoretically, how doubling algorithms work is explained by \Cref{thm:DBTrans}, the doubling transformation theorem, below, and
computationally, structure-preserving doubling algorithms consist of two phases:
the initial setup phase for obtaining $\scrA_0-\lambda\scrB_0$ and
the structure-preserving doubling iteration phase for generating $\scrA_i-\lambda\scrB_i$ for
$i\ge 1$.

\begin{theorem}[doubling transformation theorem]\label{thm:DBTrans}
Let $\scrA-\lambda\scrB\in\bbC^{N\times N}$ be a regular matrix pencil, i.e., $\det(\scrA-\lambda\scrB)\not\equiv 0$, and let
$\scrA_{\bot},\,\scrB_{\bot}\in\bbC^{N\times N}$ such that
\begin{equation}\label{eq:DTT-1}
\rank([\scrA_{\bot},\scrB_{\bot}])=N, \quad
[\scrA_{\bot},\scrB_{\bot}]\begin{bmatrix}
                             \hm\scrB \\
                             -\scrA
                           \end{bmatrix}=0.
\end{equation}
Define
$\wtd\scrA=\scrA_{\bot}\scrA$ and $\wtd\scrB=\scrB_{\bot}\scrB$.
The following statements hold.
\begin{enumerate}
  \item[{\rm (a)}] $\wtd\scrA-\lambda\wtd\scrB$ is regular;
  \item[{\rm (b)}] If $\cZ=\cR(Z)$ is a $n$-dimensional eigenspace
        of $\scrA-\lambda\scrB$, where $Z\in\bbC^{N\times n}$, i.e.,
        there is $M\in\bbC^{n\times n}$ such that
        $\scrA Z=\scrB Z M$, then
        \begin{equation}\label{eq:DTT-2}
        \wtd\scrA\, Z=\wtd\scrB Z M^2.
        \end{equation}
        More generally, if there exist $M_A,\,M_B\in\bbC^{n\times n}$, with
        the property  $M_AM_B=M_BM_A$, such that
        $\scrA ZM_B=\scrB Z M_A$, then $\wtd\scrA\, ZM_B^2=\wtd\scrB Z M_A^2$.
\end{enumerate}
\end{theorem}

This theorem can be traced back to
1980s in \cite{bugo:1988,godu:1986,maly:1989} and more recently in \cite{badg:1997,benn:1997,suqu:2004}.
%Our proof for item (a) appears to be new.
The reader is referred to \cite[p.13]{hull:2018} for a proof.
%A different proof is through the Weierstrass canonical form of a regular matrix pencil
%\cite[Remark 2.1]{lixu:2006}.

\section{Versatile Basis Representation}\label{sec:new-basis}
Both doubling algorithms, SDASF1 \cite[p.22]{hull:2018} and SDASF2 \cite[p.24]{hull:2018}, have found
tremendous successes in solving various nonlinear matrix equations (\NME),
most notably,  \DARE s and \CARE s
\cite{laro:1995,laub:1979,laub:1991,mehr:1991,pals:1980} and \MARE s  \cite{gulx:2006,ngpo:2015,wawl:2012,xuli:2017}. They were done through the connections between
an \NME\ and an eigenvalue problem of some matrix pencil $\scrA-\lambda\scrB$.

Roughly speaking, there are two prerequisites before
one of the doubling algorithms can even be considered:
\begin{enumerate}[(1)]
  \item the spectrum of matrix pencil $\scrA-\lambda\scrB$ is nicely split into two non-overlapping parts, one of which can be mapped to, generally, $\bbD_-$ while the other to $\bbD_+$ through a judicious transformation on the matrix pencil $\scrA-\lambda\scrB$;
  \item the basis matrices correspond to the eigenspaces associated with the two non-overlapping parts
        of the spectrum can provably take the form
        $\begin{bmatrix}
  I \\
  X
\end{bmatrix}$, and $\begin{bmatrix}
  Y \\
  I
\end{bmatrix}$ for SF1 or $\begin{bmatrix}
  I \\
  Y
\end{bmatrix}$ for SF2, respectively.
\end{enumerate}
In the case of \DARE s, \CARE s, and \MARE s, both prerequisites have been established in the literature, and, moreover, computed approximations  by the doubling iterations are monotonically convergent in certain matrix  partial order such as entrywise comparison or matrix positive semidefiniteness.

The assumptions on the spectrum being split into two overlapping parts can be somewhat weakened to allow
some multiple eigenvalues on the unit circle $\bbT$ after transformation. The reader is referred to \cite[section~3.8]{hull:2018}
for more detail.

Now let us  consider
a general matrix pencil
$\scrA-\lambda\scrB$ other than those from \DARE s, \CARE s, or \MARE s and suppose the spectrum of matrix pencil $\scrA-\lambda\scrB$ is nicely split into two non-overlapping parts. Then the eigenvalue problem
\eqref{eq:calX=ES:c2} which restricts the basis matrix of the eigenspace
associated with one of the two non-overlapping parts to the form
\eqref{eq:basis-form1}
%$\begin{bmatrix}
%  I \\
%  X
%\end{bmatrix}$
may not result in a solution because, as explained in section~\ref{sec:intro},
$Z_1$ in \eqref{eq:BaseM} is singular.
%This is for one of the two non-overlapping parts of the spectrum.
Likewise, \eqref{eq:calY=ES} for the other part of the spectrum may not result in a solution either!
%In general, any basis matrix of the eigenspace can always be written in the form of \eqref{eq:BaseM}.
%In order for it to be turned into the form $\begin{bmatrix}
%  I \\
%  X
%\end{bmatrix}$, a prerequisite is that $Z_1$ must be nonsingular.
Fortunately, this is the case for
\DARE s and \CARE s from the control theory and \MARE s from applied probability and transportation theory, and Markov modulated fluid queue theory \cite{hull:2018}. Conceivably, this is unlikely the case for a general regular matrix pencil $\scrA-\lambda\scrB$, even though the spectral distribution meets the prerequisite. In any case one thing is for sure for any basis matrix in \eqref{eq:BaseM}: there
is always a submatrix of the basis matrix, consisting of $m$ of its rows and all columns, are nonsingular and well-conditioned if the
basis matrix is well-conditioned, namely, the basis matrix can always be written in the
form of\footnote {The development in the rest of this paper does not really require that $Q_1$  and $Q_2$
   in \eqref{eq:QBases} are permutation matrices. The only property we will use on them is that they are orthogonal matrices. For that reason, our development below works equally well, without any changes, for orthogonal $Q_1$ and $Q_2$. For unitary $Q_1$ and $Q_2$, all it takes is to
   replace transpose $Q_i^{\T}$ with conjugate transpose $Q_i^{\HH}$.
   While we are at it, we point out that, with mostly straightforward modification,
   the development can be made to work with nonsingular $Q_1$ and $Q_2$, but it is not clear what we might gain from such a generality, and so we omit the detail.}
\begin{subequations}\label{eq:QBases}
\begin{equation}\label{eq:QBaseX}
Q_1^{\T}\begin{bmatrix}
  I \\
  X
\end{bmatrix}
\quad\mbox{with a permutation matrix $Q_1\in\bbR^{(m+n)\times (m+n)}$}
\end{equation}
so that $X$ is of modest magnitude. In fact, it can be nicely bounded as in \Cref{thm:BdDABasis} below shows, i.e., we can make $\|X\|_2\le\sqrt{mn+1}$ with a proper $Q_1$.

\begin{theorem}\label{thm:BdDABasis}
Let $U\in\bbR^{N\times m}$ have orthonormal columns, i.e., $U^{\T}U=I_m$. Then there exists a permutation matrix $Q\in\bbR^{N\times N}$ such that, with the partitioning
$
Q^{\T}U=\begin{bmatrix}
       U_1 \\
       U_2
     \end{bmatrix}
$
where $U_1\in\bbR^{m\times m}$,
$$
%\|U_1^{-1}\|_2\le\sqrt{(N-m)(m+1)}, \quad
\|U_2U_1^{-1}\|_2\le\|U_1^{-1}\|_2\le\sqrt{(N-m)m+1}.
%\quad\|U_2U_1^{-1}\|_{\F}\le\sqrt{(N-m)(m+1)m},
$$
\end{theorem}

\begin{proof}
It is well-known that $U$ has an $m\times m$ submatrix $U_1$ such that  \cite{dhma:2007,gotz:1997,hopa:1992,xuzl:2015}
$$
\|U_1^{-1}\|_2\le\sqrt{(N-m)m+1}.
$$
Let $Q\in\bbR^{N\times N}$ be the permutation matrix such that $Q^{\T}U$ is as partitioned in the lemma. Then we have $\|U_2\|_2\le\|U\|_2=1$ and
$$
\|U_2U_1^{-1}\|_2\le\|U_1^{-1}\|_2\|U_2\|_2\le\|U_1^{-1}\|_2\le\sqrt{(N-m)m+1},
$$
as expected.
\end{proof}

The same can be said for the basis matrix
$\begin{bmatrix}
  Y \\
  I
\end{bmatrix}$ in \eqref{eq:calY=ES}, and, for that reason, an alternative and numerically preferable form
is
\begin{equation}\label{eq:QBaseY}
Q_2^{\T}\begin{bmatrix}
  Y \\
  I
\end{bmatrix}
\quad\mbox{with a permutation matrix $Q_2\in\bbR^{(m+n)\times (m+n)}$}
\end{equation}
\end{subequations}
so that $Y$ is of modest magnitude.
Now there are two permutation matrices $Q_1$ and $Q_2$ involved. Unlikely they are known {\em apriori\/}, except in handful cases such as \DARE s, \CARE s, or \MARE s for which both are the identity matrices.
Suppose somehow we do know $Q_1$ and $Q_2$ from other
means\footnote {We will propose an adaptive way to construct them in subsection~\ref{ssec:adaptiveQ1Q2}.}
and they are not the identity matrices. Can we
create a doubling algorithm correspondingly? What we will do below draws inspiration from
a recent work \cite{lixu:2025}, where $Q_1=I$
while $Q_2$ is known and essentially a permutation matrix.

Upon replacing $\begin{bmatrix}
                 I \\
                 X
               \end{bmatrix}$ with the one in \eqref{eq:QBaseX}, equation \eqref{eq:calX=ES:c2}
becomes
\begin{equation}\label{eq:calX=ES-Q}
\scrA Q_1^{\T}\begin{bmatrix}
  I \\
  X
\end{bmatrix}
%\equiv\begin{bmatrix}
%        A_{11} & A_{12} \\
%        A_{21} & A_{22}
%      \end{bmatrix}\begin{bmatrix}
%  I \\
%  X
%\end{bmatrix}
%=\begin{bmatrix}
%   B_{11} & B_{12} \\
%   B_{21} & B_{22}
% \end{bmatrix}\begin{bmatrix}
%  I \\
%  X
%\end{bmatrix}M
%\equiv
=\scrB Q_1^{\T}\begin{bmatrix}
  I \\
  X
\end{bmatrix} M.
\end{equation}
Accordingly, with new basis matrix \eqref{eq:QBaseY}, equations in \eqref{eq:calY=ES} can be combined into
one as
\begin{equation}\label{eq:calY=ES-Q}
\scrA Q_2^{\T}\begin{bmatrix}
  Y \\
  I
\end{bmatrix} N
=\scrB Q_2^{\T}\begin{bmatrix}
  Y \\
  I
\end{bmatrix}.
\end{equation}
%We may not be able to recover
%the \NME s in \eqref{eq:Alg-Ricc-Eq'} when $Q_1\ne I$ and/or $Q_2\ne I$.
A natural question to ask is:
what good will \eqref{eq:QBases} do? Firstly,
the versatility coming with $Q_1$ and $Q_2$ paves the way to
solve \NME s other than the known ones, for example, the one in
\cite{lixu:2025}, and, more importantly, potential ones yet to appear, and secondly, the new formulation of basis matrices
may expand the new application frontier of efficient doubling algorithms to applied eigenvalue problems.
Thirdly, it allows us to unify the two previous standard forms, SF1 and SF2, and the two associated doubling algorithms, SDASF1  and SDASF2,
abstracted from years of studies prior to 2018, under one big umbrella.

%In the same way as we did in \cref{sec:intro} where we expanded \eqref{eq:calX=ES} and removed the matrix $M$ to yield the four special \NME s of enormously practical interests in \eqref{eq:Alg-Ricc-Eq'}, we could do the same to \eqref{eq:calX=ES-Q}. The resulting equation will necessarily be more complicated as it
%will involves $Q_1$ partitioned. It is not clear what we might gain from that at this point, and so we will omit the detail for now.

\section{A New Unifying Framework of Doubling Algorithms}\label{sec:new-frame}
Recall our brief discussion in \cref{sec:overview} on the convergence of the doubling algorithms, SDASF1 and SDASF2 \cite{hull:2018}. For the  representations of basis matrices in
\eqref{eq:QBases}, instead of \eqref{eq:cvg-SDASF1SF2} to yield $X-X_i\to 0$ and $Y-Y_i\to 0$, naturally
we will need
\begin{equation}\label{eq:cvg-QSDASF1SF2}
\scrA_i Q_1^{\T}\begin{bmatrix}
  I \\
  X
\end{bmatrix}\to 0\,\,
\mbox{and}\,\,
\scrB_iQ_2^{\T}\begin{bmatrix}
  Y \\
  I
\end{bmatrix}\to 0
\end{equation}
to yield $X-X_i\to 0$ and $Y-Y_i\to 0$ as well. For this to happen, a viable choice of $\scrA_i-\lambda\scrB_i$
should be, for $i=0,1,2,\ldots$,
\begin{equation}\label{eq:preservation-SFQ}
\scrA_i=\kbordermatrix{ &\sss m & \sss n\cr
      \sss m & \hm E_i & 0 \cr
      \sss n & -X_i & I} Q_1, \quad
\scrB_i=\kbordermatrix{ &\sss m & \sss n\cr
      \sss m & I & -Y_i \cr
      \sss n & 0 & \hm F_i} Q_2.
\end{equation}
This is precisely what we will follow for an ultimate unifying framework of the doubling algorithms.

Note that \eqref{eq:preservation-SFQ} becomes SF1 for $Q_1=Q_2=I_{m+n}$ and SF2 for $m=n$, $Q_1=I_{2n}$ and
$Q_2=\begin{bmatrix}
       0 & I_n \\
       I_n & 0
     \end{bmatrix}$. In a sense, $Q_1$ and $Q_2$ bridge a transition from SF1 to SF2.
For the ease of future references, we will call the form in \eqref{eq:preservation-SFQ}
{\em the $Q$-standard form\/} (SFQ).

\subsection{Initialization}\label{ssec:initial-DA}
Suppose now that we already have done \eqref{eq:trans:scrAB2scrB'}.
What comes next for us to do is to pre-multiply $\scrA'-\lambda\scrB'$ by a nonsingular matrix $P$
\begin{equation}\label{eq:trans2init:Q}
\scrA_0-\lambda\scrB_0=P(\scrA'-\lambda\scrB')
\end{equation}
to transform $\scrA'-\lambda\scrB'$ into $\scrA_0-\lambda\scrB_0$ in the form of \eqref{eq:preservation-SFQ} for $i=0$, where $Q_1$ and $Q_2$ are two  permutation matrices assumed known for now.
To this end,
we partition
\begin{equation}\label{eq:scrAscrB'}
\scrA'Q_1^{\T}=\kbordermatrix{ & \sss m & \sss n \cr
                \sss m & A_{11}' & A_{12}' \cr
                \sss n & A_{21}' & A_{22}'}, \quad
\scrB'Q_2^{\T}=\kbordermatrix{ & \sss m & \sss n \cr
                \sss m & B_{11}' & B_{12}' \cr
                \sss n & B_{21}' & B_{22}'},
%[\scrA',\scrB']\begin{bmatrix}
%                 Q_1^{\T} &  \\
%                  & Q_2^{\T}
%               \end{bmatrix}
%=\kbordermatrix{ & \sss m & \sss n & \sss m & \sss n \cr
%                \sss m & A_{11}' & A_{12}' & B_{11}' & B_{12}' \cr
%                \sss n & A_{21}' & A_{22}' & B_{21}' & B_{22}'}.
\end{equation}
and write \eqref{eq:trans2init:Q} equivalently as
$P[\scrA',\scrB']=[\scrA_0,\scrB_0]$, which together with \eqref{eq:scrAscrB'} lead to
\begin{equation}\label{eq:P-dfn:init}
P\begin{bmatrix}
                           A_{11}' & A_{12}' & B_{11}' & B_{12}' \\
                           A_{21}' & A_{22}' & B_{21}' & B_{22}'
                         \end{bmatrix}
                 =\begin{bmatrix}
                     \hm E_0 & 0 & I_m & -Y_0 \\
                     -X_0    & I_n & 0 & \hm F_0
                  \end{bmatrix}.
\end{equation}
Equivalently,
\begin{equation}\label{eq:init:SF1:pf-1}
P\begin{bmatrix}
  B_{11}' & A_{12}' \\
  B_{21}' & A_{22}'
 \end{bmatrix}=\begin{bmatrix}
                 I & 0 \\
                 0 & I
               \end{bmatrix}, \,\,
P\begin{bmatrix}
   A_{11}' & B_{12}' \\
   A_{21}' & B_{22}'
 \end{bmatrix}=\begin{bmatrix}
                     \hm E_0  & -Y_0 \\
                     -X_0     & \hm F_0
                  \end{bmatrix},
\end{equation}
%The first part of the theorem is an immediate consequence of \eqref{eq:init:SF1:pf-1}. Let
%$$
%J=\begin{bmatrix}
%      I_m & \hm 0 \\
%      0 & -I_n
%    \end{bmatrix}.
%$$
%We have by the second equation in \eqref{eq:init:SF1:pf-1} that
yielding
$$
\begin{bmatrix}
E_0 & Y_0 \\
X_0 & F_0
\end{bmatrix}=J\begin{bmatrix}
                     \hm E_0  & -Y_0 \\
                     -X_0     & \hm F_0
                  \end{bmatrix}J
=J\begin{bmatrix}
  B_{11}' & A_{12}' \\
  B_{21}' & A_{22}'
 \end{bmatrix}^{-1}\begin{bmatrix}
   A_{11}' & B_{12}' \\
   A_{21}' & B_{22}'
 \end{bmatrix}J,
$$
where
\begin{equation}\label{eq:mtx-J}
J=\begin{bmatrix}
	     -I_m &  0 \\
	      0 & I_n
	    \end{bmatrix}.
\end{equation}
We summarize what we have so far into the theorem below. The idea of determining
the matrices $E_0$, $X_0$, $Y_0$, and $F_0$ this way is inspired by Poloni~\cite{polo:2010}.

\begin{theorem}\label{thm:init:SF1}
The matrix pencil $\scrA'-\lambda\scrB'$  can be reduced to
$\scrA_0-\lambda\scrB_0$ in the form of \eqref{eq:preservation-SFQ} for $i=0$:
\begin{equation}\label{eq:SFQ:i=0}
\scrA_0=\kbordermatrix{ &\sss m & \sss n\cr
      \sss m & \hm E_0 & 0 \cr
      \sss n & -X_0 & I} Q_1, \quad
\scrB_0=\kbordermatrix{ &\sss m & \sss n\cr
      \sss m & I & -Y_0 \cr
      \sss n & 0 & \hm F_0} Q_2,
\end{equation}
by some nonsingular matrix $P$
as in \eqref{eq:trans2init:Q} if and only if
$
\begin{bmatrix}
  B_{11}' & A_{12}' \\
  B_{21}' & A_{22}'
 \end{bmatrix}
$
is invertible. Moreover, when this matrix is invertible, the matrices $E_0$, $X_0$, $Y_0$, and $F_0$ are given by
\begin{equation}\label{eq:init:SF1}
\begin{bmatrix}
E_0 & Y_0 \\
X_0 & F_0
\end{bmatrix}=-\begin{bmatrix}
                B_{11}' & -A_{12}' \\
                B_{21}' & -A_{22}'
              \end{bmatrix}^{-1}\begin{bmatrix}
                                   - A_{11}' &B_{12}' \\
                                   - A_{21}' &B_{22}'
                                \end{bmatrix}.
\end{equation}
\end{theorem}

The initialization procedure from $\scrA-\lambda\scrB$ to $\scrA'-\lambda\scrB'$ and then to $\scrA_0-\lambda\scrB_0$ guarantees that
\begin{subequations}\label{eq:EigEqs}
\begin{align}
\scrA_0 Q_1^{\T}\begin{bmatrix}
  I \\
  X
\end{bmatrix}\hphantom{\scrN}
&=\scrB_0Q_1^{\T}\begin{bmatrix}
  I \\
  X
\end{bmatrix} \scrM,  \label{eq:EigEqs-X}\\
\scrA_0 Q_2^{\T}\begin{bmatrix}
  Y \\
  I
\end{bmatrix}\scrN
&=\scrB_0Q_2^{\T}\begin{bmatrix}
  Y \\
  I
\end{bmatrix},  \label{eq:EigEqs-Y}
\end{align}
\end{subequations}
provided that the original matrix pencil $\scrA-\lambda\scrB$ has eigenspaces with basis matrices as
in \eqref{eq:QBases}.

\subsection{$Q$-Structure-Preserving Doubling Algorithm}\label{ssec:QSDA}
We will create a corresponding structure-preserving doubling algorithm
%starting from $\scrA_0-\lambda\scrB_0$
which takes in
$\scrA_0-\lambda\scrB_0$ in \eqref{eq:SFQ:i=0} and executes its structure-preserving doubling iteration kernel
to be derived in a moment.
The basic idea is, with guidance of
the doubling transformation theorem, Theorem~\ref{thm:DBTrans}, to construct a sequence of
matrix pencils $\scrA_i-\lambda\scrB_i$  in the form of \eqref{eq:preservation-SFQ}. Hence
$\scrA_i-\lambda\scrB_i$ for $i\ge 1$ inherit
the same block-structure form as $\scrA_0-\lambda\scrB_0$,
thus said {\em structure-preserving},
and at the same time satisfy for $i=0,1,\ldots$
\begin{equation}\label{eq:DI:i-X}
\scrA_iQ_1^{\T}\begin{bmatrix}
         I \\
         X
       \end{bmatrix}
=\scrB_iQ_1^{\T}\begin{bmatrix}
         I \\
         X
       \end{bmatrix}\scrM^{2^i}.
\end{equation}
The reason behind the idea is as follows: when the spectral radius $\rho(\scrM)<1$, $\{\scrM^{2^i}\}$ goes to $0$ quadratically,
and when that happens, we have, as $i\to\infty$,
$$
\begin{bmatrix}
         E_i & 0 \\
         -X_i & I
       \end{bmatrix}
\begin{bmatrix}
         I \\
         X
       \end{bmatrix}\to 0
\quad\Rightarrow\quad
X-X_i\to 0,
$$
provided $\{\|\scrB_i\|\}_{i=0}^{\infty}$ is uniformly bounded, where $X$ is a solution to \eqref{eq:EigEqs-X}.

According to Theorem~\ref{thm:DBTrans}, we need to find
$\scrA_{i\bot},\,\scrB_{i\bot}\in\bbC^{(m+n)\times(m+n)}$ such that
\begin{equation}\label{eq:SFQ:i-req}
\rank([\scrA_{i\bot},\scrB_{i\bot}])=m+n, \quad
[\scrA_{i\bot},\scrB_{i\bot}]\begin{bmatrix}
                             \hm\scrB_i \\
                             -\scrA_i
                           \end{bmatrix}=0,
\end{equation}
and then let $\scrA_{i+1}=\scrA_{i\bot}\scrA_i$ and $\scrB_{i+1}=\scrB_{i\bot}\scrB_i$
while preserving the form of \eqref{eq:preservation-SFQ}.
This last requirement entails that
$\scrA_{i\bot}$ and $\scrB_{i\bot}$ take the forms:
\begin{equation}\label{eq:SF-1:i-bot-form}
\scrA_{i\bot}=\kbordermatrix{ &\sss m & \sss n\cr
      \sss m & \hm \what E_i & 0 \cr
      \sss n & -\what X_i & I}, \quad
\scrB_{i\bot}=\kbordermatrix{ &\sss m & \sss n\cr
      \sss m & I & -\what Y_i \cr
      \sss n & 0 & \hm \what F_i},
\end{equation}
where $\what E_i$, $\what F_i$, $\what X_i$, and $\what F_i$ are to be determined to satisfy the second equation
in \eqref{eq:SFQ:i-req}.
This is because, without worrying about invertibility, we will have $\scrA_{i\bot}=\scrA_{i+1}\scrA_i^{-1}$
and $\scrB_{i\bot}=\scrB_{i+1}\scrB_i^{-1}$, which force the forms in
\eqref{eq:SF-1:i-bot-form} upon $\scrA_{i\bot}$ and $\scrB_{i\bot}$.
Note the first equation in \eqref{eq:SFQ:i-req} evidently holds for $\scrA_{i\bot}$ and $\scrB_{i\bot}$ in \eqref{eq:SF-1:i-bot-form}.

To find $\scrA_{i\bot}$ and $\scrB_{i\bot}$ of \eqref{eq:SF-1:i-bot-form}, we adopt an
approach of elimination as follows. Let
\begin{equation}\label{eq:Q1Q2t}
Q_1Q_2^{\T}=\kbordermatrix{ &\sss m &\sss n \\
           \sss m & Q_{11} & Q_{12} \\
           \sss n & Q_{21} & Q_{22}    },
\end{equation}
and notice\footnote {This is where we use that $Q_2$ is an orthogonal matrix, i.e., $Q_2^{\T}Q_2=I_{m+n}$.
    If $Q_2$ is allowed to be unitary, we will have to replace $Q_2^{\T}$ in \eqref{eq:Q1Q2t} with
    $Q_2^{\HH}$.}
\begin{equation}\label{eq:Bi-Ai-form-one}
	\begin{bmatrix}
		\hm\scrB_i \\
		-\scrA_i
	\end{bmatrix}
	=\begin{bmatrix}
		I_m & -Y_i \\
		0 &  F_i \\
		-E_iQ_{11} & -E_iQ_{12} \\
		X_iQ_{11}-Q_{21} & X_iQ_{12}-Q_{22}
	\end{bmatrix} Q_2 .
\end{equation}
First,
perform block-Gaussian eliminations
$$
\begin{bmatrix}
   I_m & -Y_i \\
   0 &  F_i \\
   -E_iQ_{11} & -E_iQ_{12} \\
   X_iQ_{11}-Q_{21} & X_iQ_{12}-Q_{22}
 \end{bmatrix}
\,\toby{L_1}\,
\begin{bmatrix}
   I_m & -Y_i \\
   0 &  F_i \\
   0 & -E_iQ_{12}-E_iQ_{11}Y_i \\
   0 & X_iQ_{12}-Q_{22}+(X_iQ_{11}-Q_{21})Y_i
 \end{bmatrix}
 =:\begin{bmatrix}
   I_m & -Y_i \\
   0 &  F_i \\
   0 & -Z_i \\
   0 & -W_i
   \end{bmatrix},
$$
where $L_1\in\bbC^{2(m+n)\times 2(m+n)}$ is given by
$$
L_1=\begin{bmatrix}
      I_m & 0 & 0 & 0 \\
      0 & I_n & 0 & 0 \\
      E_iQ_{11} & 0 & I_m & 0 \\
      -(X_iQ_{11}-Q_{21}) & 0 & 0 & I_n
    \end{bmatrix},
$$
and $\toby{L_1}$ means that the matrix before it multiplied
by $L_1$ from the left gives the matrix after it, and
\begin{align}
Z_i&:=-\big(-E_iQ_{12}-E_iQ_{11}Y_i\big)=E_i(Q_{11}Y_i+Q_{12}), \nonumber\\
W_i&:=Q_{22}-X_iQ_{12}+(Q_{21}-X_iQ_{11})Y_i=[-X_i,I]Q_1Q_2^{\T}\begin{bmatrix}
                             Y_i \\
                             I
                           \end{bmatrix}. \label{eq:Wi-dfn}
\end{align}
Suppose that $W_i$ is nonsingular. We continue to perform block-Gaussian eliminations:
$$
\begin{bmatrix}
   I_m & -Y_i \\
   0 &  F_i \\
   0 & -Z_i \\
   0 & -W_i
   \end{bmatrix}
\,\toby{L_2}\,
\begin{bmatrix}
   I_m & -Y_i \\
   0 &  F_i \\
   0 & -Z_i \\
   0 & -I_n
 \end{bmatrix}
\,\toby{L_3}\,
\begin{bmatrix}
   I_m & -Y_i \\
   0 &  0 \\
   0 & 0 \\
   0 & -I_n
 \end{bmatrix}
\,\toby{L_4}\,
\begin{bmatrix}
   I_m & -Y_i \\
   0 & -I_n \\
   0 &  0 \\
   0 & 0
 \end{bmatrix},
$$
where
$$
L_2=\begin{bmatrix}
      I_{2m+n} &  \\
       & W_i^{-1}
    \end{bmatrix},\,\,
L_3=\begin{bmatrix}
      I_m & 0 & 0 & 0 \\
       & I_n & 0 & F_i \\
       &  & I_m & -Z_i \\
       &  &  & I_n
    \end{bmatrix},\,\,
L_4=\begin{bmatrix}
      I_m & 0 & 0 & 0 \\
      0 & 0 & 0 & I_n \\
      0 & 0 & I_m & 0 \\
      0 & I_n & 0 & 0
    \end{bmatrix}.
$$
Each $L_j$ performs one of the following three actions:
(1) pre-multiply a block-row by some matrix and add the result to another block-row, (2) pre-multiply a block-row by some nonsingular matrix,
and (3) swap two block-rows. Finally the last 2 block-rows of $L_4L_3L_2L_1$ give
$[\scrA_{i\bot},\scrB_{i\bot}]$:
\begin{equation}\label{eq:SFQ:i-bot}
\scrA_{i\bot}=\begin{bmatrix}
          E_iQ_{11}+Z_iW_i^{-1}(X_iQ_{11}-Q_{21}) & 0 \\
          -F_iW_i^{-1}(X_iQ_{11}-Q_{21}) & I_n
        \end{bmatrix}, \quad
\scrB_{i\bot}=\begin{bmatrix}
          I_m & -Z_iW_i^{-1} \\
          0 & \hm F_iW_i^{-1}
        \end{bmatrix}
\end{equation}
%$$
%\scrA_{i\bot}=\begin{bmatrix}
%          E_i(I_m-Y_iX_i)^{-1} & 0 \\
%          -F_i(I_n-X_iY_i)^{-1}X_i & I_n
%        \end{bmatrix}, \quad
%\scrB_{i\bot}=\begin{bmatrix}
%          I_m & -E_i(I_m-Y_iX_i)^{-1}Y_i \\
%          0 & \hm F_i(I_n-X_iY_i)^{-1}\hphantom{Y_i}
%        \end{bmatrix}
%$$
in the form of \eqref{eq:SF-1:i-bot-form}.
The second approach directly works with the second equation in \eqref{eq:SFQ:i-req}:
%$$
%\begin{bmatrix}
%  \hm \what E_i & 0 & I & -\what Y_i \\
%  -\what X_i    & I & 0 & \hm \what F_i \end{bmatrix}
%          \begin{bmatrix}
%               \hm I & -Y_i \\
%               \hm 0 & \hm F_i\\
%               -E_i & \hm 0 \\
%               \hm X_i & -I
%           \end{bmatrix}=0.
%$$
$$
\begin{bmatrix}
  \hm \what E_i & 0 & I & -\what Y_i \\
  -\what X_i    & I & 0 & \hm \what F_i \end{bmatrix}
          \begin{bmatrix}
               \hm I & -Y_i \\
               \hm 0 & \hm F_i\\
               -E_iQ_{11} & -E_iQ_{12} \\
   X_iQ_{11}-Q_{21} & X_iQ_{12}-Q_{22}
           \end{bmatrix}=0.
$$
Swap the second and last block-column of the first matrix and, accordingly, the second and
last block-row of the second matrix to get
%$$
%\begin{bmatrix}
%  \hm \what E_i & -\what Y_i    & I & 0\\
%  -\what X_i    & \hm \what F_i & 0 & I
%\end{bmatrix}
%  \begin{bmatrix}
%               \hm I & -Y_i \\
%               \hm X_i & -I \\
%               -E_i & \hm 0 \\
%               \hm 0 & \hm F_i
%  \end{bmatrix}=0,
%$$
$$
\begin{bmatrix}
  \hm \what E_i & -\what Y_i    & I & 0\\
  -\what X_i    & \hm \what F_i & 0 & I
\end{bmatrix}
\begin{bmatrix}
               \hm I & -Y_i \\
   X_iQ_{11}-Q_{21} & X_iQ_{12}-Q_{22} \\
               -E_iQ_{11} & -E_iQ_{12} \\
               \hm 0 & \hm F_i
  \end{bmatrix}=0,
$$
or equivalently
\begin{equation}\label{eq:hatEFXYform-one}
\begin{bmatrix}
	\hm \what E_i & -\what Y_i    \\
	-\what X_i    & \hm \what F_i
\end{bmatrix}
\begin{bmatrix}
	\hm I & -Y_i \\
	X_iQ_{11}-Q_{21} & X_iQ_{12}-Q_{22}
\end{bmatrix}+
\begin{bmatrix}
	-E_iQ_{11} & -E_iQ_{12} \\
	\hm 0 & \hm F_i
\end{bmatrix}=0,	
\end{equation}
giving
\begin{equation}\label{eq:SFQ:i-bot-det}
\begin{bmatrix}
  \hm \what E_i & -\what Y_i    \\
  -\what X_i    & \hm \what F_i
\end{bmatrix}
 \begin{bmatrix}
 	\hm I & -Y_i \\
 	X_iQ_{11}-Q_{21} & X_iQ_{12}-Q_{22}
 \end{bmatrix}
 =\begin{bmatrix}
              E_iQ_{11} & E_iQ_{12} \\
               \hm 0 &  -F_i
  \end{bmatrix}.
\end{equation}
It can be seen that $W_i$ is the Schur complement of the $(1,1)$st block, $I$,
in the second matrix on the left-hand side of \eqref{eq:SFQ:i-bot-det}, and hence
$W_i$  is invertible if and only if the second matrix is invertible.
Assuming that $W_i$  is invertible,
%the second matrix on the left-hand side of \eqref{eq:SFQ:i-bot-det} is invertible. In fact,
we have
\begin{equation}\label{eq:anInv}
  \begin{bmatrix}
  	\hm I & -Y_i \\
  	X_iQ_{11}-Q_{21} & X_iQ_{12}-Q_{22}
  \end{bmatrix}^{-1}
=\begin{bmatrix}
   I_m+Y_iW_i^{-1}(X_iQ_{11}-Q_{21}) & -Y_iW_i^{-1} \\
   W_i^{-1}(X_iQ_{11}-Q_{21}) & -W_i^{-1}
 \end{bmatrix}
\end{equation}
%$$
%\begin{bmatrix}
%                I & Y_i \\
%                X_i & I
%\end{bmatrix}^{-1}
%  =\begin{bmatrix}
%     \hm (I-Y_iX_i)^{-1}\hphantom{X_i} & -(I-Y_iX_i)^{-1}Y_i \\
%     -(I-X_iY_i)^{-1}X_i & \hm (I-X_iY_i)^{-1}\hphantom{Y_i}
%   \end{bmatrix}
%$$
which, together with \eqref{eq:SFQ:i-bot-det}, leads to
\begin{subequations}\label{eq:hatEFXY-iter-SFQ}
\begin{align}
\what E_i&=E_i\big[Q_{11}+(Q_{11}Y_i+Q_{12})W_i^{-1}(X_iQ_{11}-Q_{21})\big], \label{eq:hatE-iter-SFQ}\\
\what F_i&=F_iW_i^{-1}, \label{eq:hatF-iter-SFQ}\\
\what X_i&=F_iW_i^{-1}(X_iQ_{11}-Q_{21}), \label{eq:hatX-iter-SFQ}\\
\what Y_i&=E_i(Q_{11}Y_i+Q_{12})W_i^{-1} \label{eq:hatY-iter-SFQ}
\end{align}
\end{subequations}
that can be verified to be the same as implied by \eqref{eq:SFQ:i-bot}.

Finally, we have $\scrA_{i+1}=\scrA_{i\bot}\scrA_i$ and $\scrB_{i+1}=\scrB_{i\bot}\scrB_i$ in the form of \eqref{eq:preservation-SFQ} with
\begin{subequations}\label{eq:EFXY-iter-SFQ}
\begin{align}
E_{i+1}&=\what E_iE_i \nonumber \\
       &=E_i\big[Q_{11}+(Q_{11}Y_i+Q_{12})W_i^{-1}(X_iQ_{11}-Q_{21})\big]E_i,
            \label{eq:E-iter-SFQ}\\
F_{i+1}&=\what F_iF_i \nonumber \\
       &= F_iW_i^{-1}F_i, \label{eq:F-iter-SFQ}\\
X_{i+1}&=X_i+\what X_iE_i \nonumber \\
       &= X_i+F_iW_i^{-1}(X_iQ_{11}-Q_{21})E_i, \label{eq:X-iter-SFQ}\\
Y_{i+1}&=Y_i+\what Y_iF_i \nonumber \\
       &= Y_i+E_i(Q_{11}Y_i+Q_{12})W_i^{-1}F_i, \label{eq:Y-iter-SFQ}
\end{align}
\end{subequations}
where $W_i$ is defined by \eqref{eq:Wi-dfn}.

The doubling transformation theorem, Theorem~\ref{thm:DBTrans}, ensures
\eqref{eq:DI:i-X} for $i\ge 1$ if it holds for $i=0$, and also
\begin{equation}\label{eq:DI:i-Y}
\scrA_iQ_2^{\T}\begin{bmatrix}
         I \\
         Y
       \end{bmatrix}\scrN^{2^i}
=\scrB_iQ_2^{\T}\begin{bmatrix}
         I \\
         Y
       \end{bmatrix}
\end{equation}
for $i\ge 1$ if it holds for $i=0$.

The third approach to  find $\scrA_{i\bot}$ and $\scrB_{i\bot}$ is based on the fact that
also\footnote{Here we need that $Q_1$ is an orthogonal matrix, i.e., $Q_1^{\T}Q_1=I_{m+n}$.
    If $Q_1$ is allowed to be unitary, we will have to replace $Q_1^{\T}$ with $Q_1^{\HH}$ and each $Q_{ij}^{\T}$ in \eqref{eq:Bi-Ai-form-two} with $Q_{ij}^{\HH}$.}
\begin{equation}\label{eq:Bi-Ai-form-two}
	\begin{bmatrix}
		\hm\scrB_i \\
		-\scrA_i
	\end{bmatrix}
	=\begin{bmatrix}
		Q_{11}^{\T}-Y_{i}Q_{12}^{\T} & Q_{21}^{\T}-Y_iQ_{22}^{\T} \\
		F_iQ_{12}^{\T} &  F_iQ_{22}^{\T} \\
		-E_i& 0 \\
		X_i & -I_n
	\end{bmatrix} Q_1.
\end{equation}
Substituting \eqref{eq:Bi-Ai-form-two} into the second equation in \eqref{eq:SFQ:i-req} and upon
replacing $\scrA_{i\bot}$ and $\scrB_{i\bot}$ with their expressions in
\eqref{eq:SF-1:i-bot-form},
we obtain
 \begin{equation}\label{eq:hatEFXYform-two}
 	\begin{bmatrix}
 		\hm \what E_i & -\what Y_i    \\
 		-\what X_i    & \hm \what F_i
 	\end{bmatrix}
 	\begin{bmatrix}
 	Q_{11}^{\T}-Y_{i}Q_{12}^{\T} & Q_{21}^{\T}-Y_iQ_{22}^{\T} \\
 		X_i & -I_n
 	\end{bmatrix}+
 	\begin{bmatrix}
 		-E_i& 0 \\
 			F_iQ_{12}^{\T} &  F_iQ_{22}^{\T}
 	\end{bmatrix}=0,
 \end{equation}
 similar to \eqref{eq:hatEFXYform-one}.
Compare \eqref{eq:Bi-Ai-form-one} with \eqref{eq:Bi-Ai-form-two} to get
\begin{equation}
\underbrace{\begin{bmatrix}
		\hm I & -Y_i \\
		X_iQ_{11}-Q_{21} & X_iQ_{12}-Q_{22}
	\end{bmatrix}}_{=:G_1}Q_2
=\underbrace{\begin{bmatrix}
	Q_{11}^{\T}-Y_{i}Q_{12}^{\T} & Q_{21}^{\T}-Y_iQ_{22}^{\T} \\
	X_i & -I_n
	\end{bmatrix}}_{=:G_2}Q_1.
	\end{equation}
Thus $G_1$ is invertible if and only if $G_2$ is invertible.
Previously, we assumed that $W_i$ is invertible and also showed that $W_i$
is invertible if and only if $G_1$ is invertible. On the other hand, we also see that
$G_2$ is invertible if and only if
\begin{equation}\label{eq:tildeWi}
\widetilde{W}_i=Q_{11}^{\T}-Y_{i}Q_{12}^{\T}+(Q_{21}^{\T}-Y_iQ_{22}^{\T})X_i	
\end{equation}
is invertible because $\widetilde{W}_i$ is the Schur complement of
the $(2,2)$nd block, $-I_n$, in $G_2$. In summary, if any one of
$G_1$, $G_2$, $W_i$ and $\wtd W_i$ is nonsingular, then all the other three are
nonsingular, too.

Assume $W_i$ is invertible.
%
%The latter is invertible if
%$W_i$ is invertible, and therefore
%$\begin{bmatrix}
%Q_{11}^{\T}-Y_{i}Q_{12}^{\T} & Q_{21}^{\T}-Y_iQ_{22}^{\T} \\
%X_i & -I_n
%\end{bmatrix}$
%is invertible and so is
%
It follows from \eqref{eq:hatEFXYform-two} that
\begin{equation}\label{eq:hatEFXYnewform-two}
	\begin{bmatrix}
		\widehat E_{i} & \widehat Y_{i}\\
		\widehat X_{i}& \widehat F_{i}
	\end{bmatrix}=J\begin{bmatrix}
	E_i& 0 \\
	-F_iQ_{12}^{\T} & -F_iQ_{22}^{\T}
	\end{bmatrix} \begin{bmatrix}
	Q_{11}^{\T}-Y_{i}Q_{12}^{\T} & Q_{21}^{\T}-Y_iQ_{22}^{\T} \\
	X_i & -I_n
	\end{bmatrix}^{-1}J,
	\end{equation}
where $J$ is as defined in \eqref{eq:mtx-J}.
Note
	\begin{equation*}
		\begin{bmatrix}
			Q_{11}^{\T}-Y_{i}Q_{12}^{\T} & Q_{21}^{\T}-Y_iQ_{22}^{\T} \\
			X_i & -I_n
		\end{bmatrix}^{-1}=\begin{bmatrix}
		\wtd W_i^{-1} & \wtd W_i^{-1}(Q_{21}^{\T}-Y_iQ_{22}^{\T})\\
		X_i\wtd W_i^{-1}&-I_n+X_i\wtd W_i^{-1}(Q_{21}^{\T}-Y_iQ_{22}^{\T})
		\end{bmatrix}
	\end{equation*}
which, together with  \eqref{eq:hatEFXYnewform-two}, lead to \begin{subequations}\label{eq:hatEFXY-iter-SFQ-new}
	\begin{align}
		\what E_i&=E_i\wtd W_i^{-1}, \label{eq:hatE-iter-SFQ-new}\\
		\what F_i&=F_i\big[Q_{22}^{\T}+(Q_{22}^{\T}X_i+Q_{12}^{\T})\wtd W_i^{-1}(Y_iQ_{22}^{\T}-Q_{21}^{\T})\big], \label{eq:hatF-iter-SFQ-new}\\
		\what X_i&=F_i(Q_{22}^{\T}X_i+Q_{12}^{\T})\wtd W_i^{-1}, \label{eq:hatX-iter-SFQ-new}\\
		\what Y_i&=E_i\wtd W_i^{-1}(Y_iQ_{22}^{\T}-Q_{21}^{\T}). \label{eq:hatY-iter-SFQ-new}
	\end{align}
\end{subequations}
Finally, we have $\scrA_{i+1}=\scrA_{i\bot}\scrA_i$ and $\scrB_{i+1}=\scrB_{i\bot}\scrB_i$ in the form of \eqref{eq:preservation-SFQ} with
\begin{subequations}\label{eq:EFXY-iter-SFQ-new}
	\begin{align}
		E_{i+1}&=\what E_iE_i \nonumber \\
		&=E_i\wtd W_i^{-1}E_i,
		\label{eq:E-iter-SFQ-new}\\
		F_{i+1}&=\what F_iF_i \nonumber \\
		&= F_i\big[Q_{22}^{\T}+(Q_{22}^{\T}X_i+Q_{12}^{\T})\wtd W_i^{-1}(Y_iQ_{22}^{\T}-Q_{21}^{\T})\big]F_i, \label{eq:F-iter-SFQ-new}\\
		X_{i+1}&=X_i+\what X_iE_i \nonumber \\
		&= X_i+F_i(Q_{22}^{\T}X_i+Q_{12}^{\T})\wtd W_i^{-1}E_i, \label{eq:X-iter-SFQ-new}\\
		Y_{i+1}&=Y_i+\what Y_iF_i \nonumber \\
		&= Y_i+E_i\wtd W_i^{-1}(Y_iQ_{22}^{\T}-Q_{21}^{\T})F_i, \label{eq:Y-iter-SFQ-new}
	\end{align}
\end{subequations}
where $\wtd W_i$ is defined by \eqref{eq:tildeWi}.

In \eqref{eq:EFXY-iter-SFQ} and \eqref{eq:EFXY-iter-SFQ-new}, we have established two sets
of doubling iteration formulas for the same purpose, but each involves either $W_i^{-1}$ or $\wtd W_i^{-1}$,
exclusively. That offers flexibility in implementation, dependent on $m$ and $n$.
Noting that $W_i$ is $n\times n$ while $\widetilde{W}_i$ is $m\times m$,
we should use \eqref{eq:EFXY-iter-SFQ} if $n<m$ or \eqref{eq:EFXY-iter-SFQ-new} otherwise.
Algorithm~\ref{alg:DA-SFQ} outlines the structure-preserving doubling algorithm
for SFQ.
We will discuss what stopping criteria should be used at line 1 to terminate
the doubling iteration later in subsection~\ref{ssec:QSDA}.
Because the algorithm involves matrix inverses
at each iterative step, breakdowns, i.e., when some of the inverses fail to exist, are possibilities in general.
In its various important applications such as \CARE, \DARE,\index{Discrete-time algebraic Riccati equation (\DARE)} \MARE,\index{$M$-matrix algebraic Riccati equation (\MARE)} and \HARE, provably no breakdown can occur.

\begin{algorithm}[t]
\caption{SDASFQ: Structured Doubling Algorithm for SFQ}
\label{alg:DA-SFQ}
\begin{algorithmic}[1]
    \hrule\vspace{1ex}
    \REQUIRE $X_0\in\bbC^{n\times m},\, Y_0\in\bbC^{m\times n},\, E_0\in\bbC^{m\times m},\,
             F_0\in\bbC^{n\times n}$, and $Q_1Q_2^{\T}$ partitioned as in \eqref{eq:Q1Q2t};
    \ENSURE  the limit of $X_i$ if it converges.
    \hrule\vspace{1ex}
    \FOR{$i=0, 1, \ldots,$ until convergence}
         \STATE $W_i=Q_{22}-X_iQ_{12}-(X_iQ_{11}-Q_{21})Y_i$ \\
                (or, alternatively, $\widetilde{W}_i=Q_{11}^{\T}-Y_{i}Q_{12}^{\T}+(Q_{21}^{\T}-Y_iQ_{22}^{\T})X_i$ if $n>m$);
         \STATE compute $E_{i+1},\, F_{i+1},\, X_{i+1},\, Y_{i+1}$ according to \eqref{eq:EFXY-iter-SFQ} \\ (or, alternatively, \eqref{eq:EFXY-iter-SFQ-new} if $n>m$);
    \ENDFOR
    \RETURN $X_i$ at convergence as the computed solution.
\end{algorithmic}
\end{algorithm}

\subsection*{Special Cases}
The iterative formulas in \eqref{eq:EFXY-iter-SFQ}/\eqref{eq:EFXY-iter-SFQ-new} can be greatly simplified when $Q_1=Q_2$ in which case
$Q_1Q_2^{\T}=I$, implying $Q_{11}=I$, $Q_{22}=I$, and $Q_{ij}=0$ for $i\ne j$ and therefore
$W_i=I-X_iY_i$ and
\begin{subequations}\label{eq:EFXY-iter-SFQ:Q1=Q2}
\begin{align}
E_{i+1}&=E_i\big[I+Y_iW_i^{-1}X_i\big]E_i \nonumber \\
       &=E_i\big(I-Y_iX_i\big)^{-1}E_i     \label{eq:E-iter-SFQ:Q1=Q2}\\
F_{i+1}&= F_iW_i^{-1}F_i \nonumber \\
       &= F_i\big(I-X_iY_i\big)^{-1}F_i, \label{eq:F-iter-SFQ:Q1=Q2}\\
X_{i+1}&= X_i+F_iW_i^{-1}(X_iQ_{11}-Q_{21})E_i\nonumber \\
       &= X_i+F_i\big(I-X_iY_i\big)^{-1}X_iE_i, \label{eq:X-iter-SFQ:Q1=Q2}\\
Y_{i+1}&= Y_i+E_i(Q_{11}Y_i+Q_{12})W_i^{-1}F_i \nonumber \\
       &= Y_i+E_iY_i\big(I-X_iY_i\big)^{-1}F_i, \label{eq:Y-iter-SFQ:Q1=Q2}
\end{align}
\end{subequations}
which are exactly the same as SDASF1 \cite[p.21]{hull:2018}, the doubling algorithm for SF1.

%\marginpar{\tiny need to check?checked.}
Another case that can also lead to much simplified iterative formulas than \eqref{eq:EFXY-iter-SFQ} is
\begin{equation}\label{eq:SFQ-DA:special2}
m=n, \quad Q_1Q_2^{\T}=\begin{bmatrix}
                 &  I_n \\
                I_n &
              \end{bmatrix},
\end{equation}
implying $Q_{11}=0$, $Q_{22}=0$, and $Q_{ij}=I_n$ for $i\ne j$ and therefore
$W_i=Y_i-X_i$ and
\begin{subequations}\label{eq:EFXY-iter-SFQ:Q1=Q2-a}
\begin{align}
E_{i+1}&=E_i\big(X_i-Y_i\big)^{-1}E_i     \label{eq:E-iter-SFQ:Q1=Q2-a}\\
F_{i+1}&= F_i\big(Y_i-X_i\big)^{-1}F_i, \label{eq:F-iter-SFQ:Q1=Q2-a}\\
X_{i+1}&= X_i+F_i\big(X_i-Y_i\big)^{-1}E_i, \label{eq:X-iter-SFQ:Q1=Q2-a}\\
Y_{i+1}&= Y_i+E_i\big(Y_i-X_i\big)^{-1}F_i, \label{eq:Y-iter-SFQ:Q1=Q2-a}
\end{align}
\end{subequations}
which are exactly the same as SDASF2 \cite[p.24]{hull:2018}, the doubling algorithm for SF2.

\section{Dual Equations}\label{sec:SF1SF2d}
Recall \eqref{eq:SFQ:i=0} and \eqref{eq:EigEqs-X}.
%Equation \eqref{eq:DI:i-X} for $i=0$ gives
%\begin{subequations}\label{eq:scrA0B0-eigEQ:ALL}
%\begin{equation}\label{eq:scrA0B0-eigEQ:p}
%\scrA_0 Q_1^{\T}\begin{bmatrix}
%  I \\
%  X
%\end{bmatrix}
%=\scrB_0Q_1^{\T}\begin{bmatrix}
%  I \\
%  X
%\end{bmatrix} \scrM.
%\end{equation}
Blockwise expanding this equation into two and then eliminating $\scrM$ will lead to an equation in $X$ without
any explicit reference to $\scrM$. We call the resulting equation
the {\em primal equation\/} which is the one that the structure-preserving doubling algorithm -- SDASFQ (Algorithm~\ref{alg:DA-SFQ}) --
is created to solve
with the $X$-sequence $\{X_i\}_{i=0}^{\infty}$. It is natural to ask what
the $Y$-sequence $\{Y_i\}_{i=0}^{\infty}$  converges to, if anything.
As a matter of fact,  a doubling algorithm cannot be  understood fully, especially for its convergence behavior, without
knowing what
the $Y$-sequence $\{Y_i\}_{i=0}^{\infty}$ will do. At the first glance, the $Y$-sequence
is auxiliary and its very existence is to make computing $X$-sequence possible.
But this is only appearance.
The matter of fact is that $Y$-sequence also approaches to something
that turns out to be a solution to another nonlinear matrix equation which we will call the
{\em dual equation}.

Substituting
$\scrA_0$ and $\scrB_0$ in \eqref{eq:SFQ:i=0} into \eqref{eq:EigEqs-X}, we get, noticing \eqref{eq:Q1Q2t},
\begin{align}
\begin{bmatrix}
  \hm E_0 & 0 \\
  -X_0 & I_n
\end{bmatrix}\begin{bmatrix}
  I_m \\
  X
\end{bmatrix}
&=\begin{bmatrix}
   I_m & -Y_0 \\
   0 & \hm F_0
 \end{bmatrix} Q_2Q_1^{\T}\begin{bmatrix}
  I_m \\
  X
\end{bmatrix}\scrM \nonumber \\
&=\begin{bmatrix}
   I_m & -Y_0 \\
   0 & \hm F_0
 \end{bmatrix} \begin{bmatrix}
                 Q_{11}^{\T} & Q_{21}^{\T} \\
                 Q_{12}^{\T} & Q_{22}^{\T}
               \end{bmatrix}
 \begin{bmatrix}
  I \\
  X
\end{bmatrix}\scrM \label{eq:4-primal}\\
&=\begin{bmatrix}
  Q_{11}^{\T}-Y_0Q_{12}^{\T}+(Q_{21}^{\T}-Y_0Q_{22}^{\T})X \\
  F_0(Q_{12}^{\T}+Q_{22}^{\T}X)
\end{bmatrix}\scrM, \nonumber
\end{align}
yielding the primal equation
\begin{equation}\label{eq:NME-dual:SFQp}
X=X_0+F_0(Q_{12}^{\T}+Q_{22}^{\T}X)\big[Q_{11}^{\T}-Y_0Q_{12}^{\T}+(Q_{21}^{\T}-Y_0Q_{22}^{\T})X\big]^{-1}E_0.
\end{equation}
Next, substituting $\scrA_0$ and $\scrB_0$ in \eqref{eq:SFQ:i=0} into \eqref{eq:EigEqs-Y}, we get
\begin{align}
\begin{bmatrix}
  I_m & -Y_0 \\
  0 & \hm F_0
\end{bmatrix}\begin{bmatrix}
  Y \\
  I_n
\end{bmatrix}
&=\begin{bmatrix}
  \hm E_0 & 0\\
    -X_0 & I
  \end{bmatrix} Q_1 Q_2^{\T}\begin{bmatrix}
  Y \\
  I
\end{bmatrix}\scrN \nonumber\\
&=\begin{bmatrix}
  \hm E_0 & 0\\
    -X_0 & I
  \end{bmatrix} \begin{bmatrix}
                 Q_{11} & Q_{12} \\
                 Q_{21} & Q_{22}
               \end{bmatrix}
 \begin{bmatrix}
  Y \\
  I
\end{bmatrix}\scrN \label{eq:4-dual}\\
&=\begin{bmatrix}
  E_0(Q_{12}+Q_{11}Y) \\
  Q_{22}-X_0Q_{12}+(Q_{21}-X_0Q_{11})Y
\end{bmatrix}\scrN, \nonumber
\end{align}
yielding the dual equation
\begin{equation}\label{eq:NME-dual:SFQd}
Y=Y_0+E_0(Q_{12}+Q_{11}Y)\big[Q_{22}-X_0Q_{12}+(Q_{21}-X_0Q_{11})Y\big]^{-1}F_0.
\end{equation}

Let
\begin{equation}\label{eq:PImn}
\Pi_{m,n}=\begin{bmatrix}
            0 & I_m \\
            I_n & 0
          \end{bmatrix}.
\end{equation}
It can be seen that
$$
\Pi_{m,n}^{\T}\begin{bmatrix}
  I_m & -Y_0 \\
  0 & \hm F_0
\end{bmatrix}\begin{bmatrix}
  Y \\
  I_n
\end{bmatrix}=\Pi_{m,n}^{\T}\begin{bmatrix}
  I_m & -Y_0 \\
  0 & \hm F_0
\end{bmatrix}\Pi_{m,n}\Pi_{m,n}^{\T}\begin{bmatrix}
  Y \\
  I_n
\end{bmatrix}
=\begin{bmatrix}
  \hm F_0 & 0 \\
  -Y_0 & I_m
\end{bmatrix}\begin{bmatrix}
  I_n \\
  Y
\end{bmatrix}
$$
and, in the same way,
$$
\Pi_{m,n}^{\T}\begin{bmatrix}
  \hm E_0 & 0\\
    -X_0 & I
  \end{bmatrix} \begin{bmatrix}
                 Q_{11} & Q_{12} \\
                 Q_{21} & Q_{22}
               \end{bmatrix}
 \begin{bmatrix}
  Y \\
  I
\end{bmatrix}
=\begin{bmatrix}
  I_n & -X_0 \\
  0 & \hm E_0
\end{bmatrix}\begin{bmatrix}
Q_{22} & Q_{21}\\
Q_{12} & Q_{11}
\end{bmatrix}\begin{bmatrix}
  I_n \\
  Y
\end{bmatrix}.
$$
We can rewrite \eqref{eq:4-dual} equivalently as
\begin{equation}\label{eq:4-dual'}
\begin{bmatrix}
  \hm F_0 & 0 \\
  -Y_0 & I_m
\end{bmatrix}\begin{bmatrix}
  I_n \\
  Y
\end{bmatrix}=\begin{bmatrix}
  I_n & -X_0 \\
  0 & \hm E_0
\end{bmatrix}\begin{bmatrix}
Q_{22} & Q_{21}\\
Q_{12} & Q_{11}
\end{bmatrix}\begin{bmatrix}
  I_n \\
  Y
\end{bmatrix}\scrN,
\end{equation}
which structurally takes the same form as \eqref{eq:4-primal}. This will be the basis for us to define
the dual matrix pencil $\scrA_0^{(\rmd)}-\lambda\scrB_0^{(\rmd)}$ to $\scrA_0-\lambda\scrB_0$
such that
%
%\textcolor{blue}{Set
%$E_0^{(\rmd)}=F_0,\ F_0^{(\rmd)}=E_0,\ X_0^{(\rmd)}=Y_0,\ Y_0^{(\rmd)}=X_0,\ X^{(\rmd)}=Y,\ Y^{(\rmd)}=X$ and
%$$(Q_{12}^{(\rmd)})^{\T}=Q_{12},\ (Q_{22}^{(\rmd)})^{\T}=Q_{11},\ (Q_{11}^{(\rmd)})^{\T}=Q_{22},\ (Q_{21}^{(\rmd)})^{\T}=Q_{21},$$
%then  \eqref{eq:NME-dual:SFQd} can be written as
%$$\begin{bmatrix}
%	E^{(\rmd)}_0 & 0\\
%	-X_0^{(\rmd)} & I
%\end{bmatrix}\begin{bmatrix}
%I\\
%X^{(\rmd)}
%\end{bmatrix}=\begin{bmatrix}
%I & -Y_0^{(\rmd)}\\
% & F_0^{(\rmd)}
%\end{bmatrix}\begin{bmatrix}
%Q_{22} & Q_{21}\\
%Q_{12} & Q_{11}
%\end{bmatrix}\begin{bmatrix}
%I\\
%X^{(\rmd)}
%\end{bmatrix}\scrM^{(\rmd)}.$$
%We can show that $X_i^{(\rmd)}=Y_i$, $Y_i^{(\rmd)}=X_i$, $E_i^{(\rmd)}=F_i$, and $F_i^{(\rmd)}=E_i$ for all $i\ge 0$. }
%
%To better understand what $Y$-sequence is going to approach to,  we will have to invent the associated dual equation.
%For this purpose, we will define a new matrix pencil $\scrA_0^{(\rmd)}-\lambda\scrB_0^{(\rmd)}$
%which we will call the {\em dual matrix pencil\/} of $\scrA_0-\lambda\scrB_0$.
%In doing so, we have keep the following in mind:
\begin{equation}\label{eq:scrA0B0-eigEQ:d}
\scrA_0^{(\rmd)} (Q_1^{(\rmd)})^{\T}\begin{bmatrix}
  I \\
  Y
\end{bmatrix}
=\scrB_0^{(\rmd)}(Q_1^{(\rmd)})^{\T}\begin{bmatrix}
  I \\
  Y
\end{bmatrix}\scrN ,
\end{equation}
that will recover \eqref{eq:EigEqs-Y}.
%Notice the striking similarity between \eqref{eq:scrA0B0-eigEQ:d} and \eqref{eq:EigEqs-X}.
%It remains to define $\scrA_0^{(\rmd)}-\lambda\scrB_0^{(\rmd)}$.
%%, where
%%$\scrN$ is an $n\times n$ matrix and often $\rho(\scrN)<1$ (later we will see what is important is  $\rho(\scrM)\cdot\rho(\scrN)<1$ or at least $\rho(\scrM)\cdot\rho(\scrN)\le 1$),
%%and at the same time $Y$-sequence produces approximations to some  $Y$ that satisfies \eqref{eq:scrA0B0-eigEQ:d}.
%
%??????????????
%and recall
%$\scrA_0$ and $\scrB_0$ as in \eqref{eq:SFQ:i=0}.
%%\begin{equation}\label{eq:SFQ:i=0}
%%\scrA_0=\kbordermatrix{ &\sss m & \sss n\cr
%%      \sss m & \hm E_0 & 0 \cr
%%      \sss n & -X_0 & I} Q_1, \quad
%%\scrB_0=\kbordermatrix{ &\sss m & \sss n\cr
%%      \sss m & I & -Y_0 \cr
%%      \sss n & 0 & \hm F_0} Q_2,
%%\end{equation}
%%and
Guided by \eqref{eq:4-dual'}, we define the dual matrix pencil $\scrA_0^{(\rmd)}-\lambda\scrB_0^{(\rmd)}$
by
\begin{subequations}\label{eq:scrA0B0:SFQd}
\begin{equation}\label{eq:scrA0B0:SFQd-1}
\scrA_0^{(\rmd)}:=\Pi_{m,n}^{\T}\scrB_0\Pi_{m,n}, \,\,
\scrB_0^{(\rmd)}:=\Pi_{m,n}^{\T}\scrA_0\Pi_{m,n}
\end{equation}
to give
\begin{equation}\label{eq:scrA0B0:SFQd-2}
\scrA_0^{(\rmd)}=\kbordermatrix{ & \sss n & \sss m \cr
                \sss n & \hm F_0 & 0 \cr
                \sss m & -Y_0 & I} \underbrace{\Pi_{m,n}^{\T}Q_2\Pi_{m,n}}_{=:Q_1^{(\rmd)}}, \,\,
\scrB_0^{(\rmd)}=\kbordermatrix{ & \sss n & \sss m \cr
                \sss n & I & -X_0 \cr
                \sss m & 0 & \hm E_0} \underbrace{\Pi_{m,n}^{\T}Q_1\Pi_{m,n}}_{=:Q_2^{(\rmd)}},
\end{equation}
where $Q_1^{(\rmd)}$ and $Q_2^{(\rmd)}$ are also defined.
\end{subequations}
It is interesting to observe that they are simply derivable from swapping $X_0$ and $Y_0$,
and swapping $E_0$ and $F_0$
in $\scrA_0$ and $\scrB_0$ and with proper $Q$-matrices.
Note that
$\scrA_0^{(\rmd)}-\lambda\scrB_0^{(\rmd)}$ is still in the form of SFQ, and
recover \eqref{eq:EigEqs-Y} via \eqref{eq:scrA0B0-eigEQ:d}.
It can be verified that
\begin{equation}\label{eq:SFQ22SFQ}
\Pi_{m,n}\scrA_0^{(\rmd)}=\scrB_0\Pi_{m,n}, \quad
\Pi_{m,n}\scrB_0^{(\rmd)}=\scrA_0\Pi_{m,n}.
\end{equation}

\begin{theorem}\label{thm:SFQ=SFQd-eig}
For  $\scrA_0-\lambda\scrB_0$ in \eqref{eq:SFQ:i=0} and $\scrA_0^{(\rmd)}-\lambda\scrB_0^{(\rmd)}$ defined by
 \eqref{eq:scrA0B0:SFQd}, if $(\lambda,\bz)$
is an eigenpair of $\scrA_0-\lambda\scrB_0$, then $(1/\lambda,\Pi_{m,n}^{\T}\bz)$
is an eigenpair of $\scrA_0^{(\rmd)}-\lambda\scrB_0^{(\rmd)}$ and vice versa.
\end{theorem}

\begin{proof}
Let $(\lambda,\bz)$
is an eigenpair of $\scrA_0-\lambda\scrB_0$, i.e.,
$\scrA_0\bz=\lambda\scrB_0\bz$. By \eqref{eq:scrA0B0:SFQd-1}, we get
$$
\Pi_{m,n}\scrB_0^{(\rmd)}\Pi_{m,n}^{\T}\bz=\lambda\Pi_{m,n}\scrA_0^{(\rmd)}\Pi_{m,n}^{\T}\bz,
$$
implying $(1/\lambda,\Pi_{m,n}^{\T}\bz)$
is an eigenpair of $\scrA_0^{(\rmd)}-\lambda\scrB_0^{(\rmd)}$. Similarly, we can prove the converse
statement.
\end{proof}

\begin{remark}\label{rk:SF1=SF1d-eig}
{\rm
This theorem sheds very helpful light on the convergence of SDASFQ (Algorithm~\ref{alg:DA-SFQ}), namely
in the event that $\rho(\scrM)<1$ and $\rho(\scrN)<1$, the limits of the $X$-sequence and $Y$-sequence
relate to the eigenspaces of $\scrA_0-\lambda\scrB_0$ associated with its eigenvalues in $\bbD_-$ and $\bbD_+$,  respectively.
}
\end{remark}

Next we compare the results of SDASFQ (Algorithm~\ref{alg:DA-SFQ}) with input $\scrA_0-\lambda\scrB_0$
and  with input  $\scrA_0^{(\rmd)}-\lambda\scrB_0^{(\rmd)}$, respectively.
Ignoring any possible breakdown, we denote the matrices generated by the algorithm
by $X_i,\,Y_i,\,E_i,\,F_i$ for the input $\scrA_0-\lambda\scrB_0$ and by
$X_i^{(\rmd)},\,Y_i^{(\rmd)},\,E_i^{(\rmd)},\,F_i^{(\rmd)}$ for the input $\scrA_0^{(\rmd)}-\lambda\scrB_0^{(\rmd)}$.
We have the following theorem.

\begin{theorem}\label{thm:dualBYSFQ}
We have
$X_i^{(\rmd)}=Y_i$, $Y_i^{(\rmd)}=X_i$, $E_i^{(\rmd)}=F_i$, and $F_i^{(\rmd)}=E_i$ for all $i\ge 0$. In particular
\begin{align*}
\scrA_i^{(\rmd)}&\equiv\kbordermatrix{ &\sss n & \sss m\cr
                          \sss n & \hm E_i^{(\rmd)} & 0 \cr
                          \sss m & -X_i^{(\rmd)} & I} Q_1^{(\rmd)}
                =\Pi_{m,n}^{\T}\scrB_i\Pi_{m,n}
                \equiv\kbordermatrix{ & \sss n & \sss m \cr
                \sss n & \hm F_i & 0 \cr
                \sss m & -Y_i & I} (\Pi_{m,n}^{\T}Q_2\Pi_{m,n}), \\
\scrB_i^{(\rmd)}&\equiv\kbordermatrix{ &\sss n & \sss m\cr
                                \sss n & I & -Y_i^{(\rmd)} \cr
                                \sss m & 0 & \hm F_i^{(\rmd)}} Q_2^{(\rmd)}
      =\Pi_{m,n}^{\T}\scrA_i\Pi_{m,n}
                \equiv\kbordermatrix{ & \sss n & \sss m \cr
                \sss n & I & -X_i \cr
                \sss m & 0 & \hm E_i} (\Pi_{m,n}^{\T}Q_1\Pi_{m,n}).
\end{align*}
\end{theorem}

\begin{proof}
It is noted that  $(E_{i+1}, F_{i+1}, X_{i+1}, Y_{i+1})$
is defined from $(E_i, F_i, X_i, Y_i)$ via $(\what E_i,\what F_i,\what X_i,\what Y_i)$,
which is determined by
\begin{equation}\label{eq:dualBYSFQ:pf-1}
\begin{bmatrix}
  \hm \what E_i & 0 & I & -\what Y_i \\
  -\what X_i    & I & 0 & \hm \what F_i \end{bmatrix}
\begin{bmatrix}
  \begin{bmatrix}
      I_m & -Y_i \\
                0 & \hm F_i
  \end{bmatrix}Q_2   \\
  \begin{bmatrix}
    \hm E_i & 0 \\
    -X_i & I_n
  \end{bmatrix}  Q_1
\end{bmatrix}
=0.
\end{equation}
This equation leads to \eqref{eq:SFQ:i-bot-det} that uniquely determines
$(\what E_i,\what F_i,\what X_i,\what Y_i)$
because of the special structured form that the left most matrix in \eqref{eq:dualBYSFQ:pf-1} has.
In the other words, $(\what E_i,\what F_i,\what X_i,\what Y_i)$ is uiquely determined by
\eqref{eq:dualBYSFQ:pf-1}.
For the same reason, $(E_{i+1}^{(\rmd)}, F_{i+1}^{(\rmd)}, X_{i+1}^{(\rmd)}, Y_{i+1}^{(\rmd)})$
is defined from $(E_i^{(\rmd)}, F_i^{(\rmd)}, X_i^{(\rmd)}, Y_i^{(\rmd)})$
via $(\what E_i^{(\rmd)},\what F_i^{(\rmd)},\what X_i^{(\rmd)},\what Y_i^{(\rmd)})$,
which is determined uniquely by
\begin{equation}\label{eq:dualBYSFQ:pf-2}
\begin{bmatrix}
  \hm \what E_i^{(\rmd)} & 0 & I & -\what Y_i^{(\rmd)} \\
  -\what X_i^{(\rmd)}    & I & 0 & \hm \what F_i^{(\rmd)} \end{bmatrix}
\begin{bmatrix}
  \begin{bmatrix}
      I_n & -Y_i^{(\rmd)} \\
      0 & \hm F_i^{(\rmd)}
  \end{bmatrix}Q_2^{(\rmd)}   \\
  \begin{bmatrix}
    \hm E_i^{(\rmd)} & 0 \\
    -X_i^{(\rmd)} & I_m
  \end{bmatrix}  Q_1^{(\rmd)}
\end{bmatrix}
=0
\end{equation}
because the left most matrix is exactly restricted to the special structured form.

Recalling \eqref{eq:scrA0B0:SFQd-2}, we find
\begin{align*}
\begin{bmatrix}
  \begin{bmatrix}
      I_n & -Y_i^{(\rmd)} \\
                0 & \hm F_i^{(\rmd)}
  \end{bmatrix}Q_2^{(\rmd)}   \\
  \begin{bmatrix}
    \hm E_i^{(\rmd)} & 0 \\
    -X_i^{(\rmd)} & I_m
  \end{bmatrix}  Q_1^{(\rmd)}
\end{bmatrix}
&=\begin{bmatrix}
  \begin{bmatrix}
      I_n & -Y_i^{(\rmd)} \\
      0 & \hm F_i^{(\rmd)}
  \end{bmatrix}\Pi_{m,n}^{\T}Q_1\Pi_{m,n}   \\
  \begin{bmatrix}
    \hm E_i^{(\rmd)} & 0 \\
    -X_i^{(\rmd)} & I_m
  \end{bmatrix}  \Pi_{m,n}^{\T}Q_2\Pi_{m,n}
\end{bmatrix} \\
&=\begin{bmatrix}
  \begin{bmatrix}
      -Y_i^{(\rmd)} & I_n \\
      \hm F_i^{(\rmd)} & 0
  \end{bmatrix}Q_1   \\
  \begin{bmatrix}
    0 & \hm E_i^{(\rmd)} \\
    I_m & -X_i^{(\rmd)}
  \end{bmatrix}  Q_2
\end{bmatrix}\Pi_{m,n} \\
&=\begin{bmatrix}
        &     &     & I_m \\
        &     & I_n &  \\
        & I_m &  &  \\
    I_n &     &  &
  \end{bmatrix}\begin{bmatrix}
  \begin{bmatrix}
    I_m & -X_i^{(\rmd)} \\
    0 & \hm E_i^{(\rmd)}
  \end{bmatrix}  Q_2 \\
  \begin{bmatrix}
      \hm F_i^{(\rmd)} & 0 \\
      -Y_i^{(\rmd)} & I_n
  \end{bmatrix}Q_1
\end{bmatrix}\Pi_{m,n}.
\end{align*}
Hence, by \eqref{eq:dualBYSFQ:pf-2}, we conclude that $(\what E_i^{(\rmd)},\what F_i^{(\rmd)},\what X_i^{(\rmd)},\what Y_i^{(\rmd)})$
satisfies
$$
\begin{bmatrix}
  \hm \what E_i^{(\rmd)} & 0 & I & -\what Y_i^{(\rmd)} \\
  -\what X_i^{(\rmd)}    & I & 0 & \hm \what F_i^{(\rmd)} \end{bmatrix}
\begin{bmatrix}
        &     &     & I_m \\
        &     & I_n &  \\
        & I_m &  &  \\
    I_n &     &  &
  \end{bmatrix}\begin{bmatrix}
  \begin{bmatrix}
    I_m & -X_i^{(\rmd)} \\
    0 & \hm E_i^{(\rmd)}
  \end{bmatrix}  Q_2 \\
  \begin{bmatrix}
      \hm F_i^{(\rmd)} & 0 \\
      -Y_i^{(\rmd)} & I_n
  \end{bmatrix}Q_1
\end{bmatrix}
=0,
$$
or equivalently,
$$
\begin{bmatrix}
   -\what Y_i^{(\rmd)} & I & 0 & \hm \what E_i^{(\rmd)}\\
    \hm \what F_i^{(\rmd)} & 0 & I & -\what X_i^{(\rmd)} \end{bmatrix}\begin{bmatrix}
  \begin{bmatrix}
    I_m & -X_i^{(\rmd)} \\
    0 & \hm E_i^{(\rmd)}
  \end{bmatrix}  Q_2 \\
  \begin{bmatrix}
      \hm F_i^{(\rmd)} & 0 \\
      -Y_i^{(\rmd)} & I_n
  \end{bmatrix}Q_1
\end{bmatrix}
=0,
$$
or equivalently,
\begin{equation}\label{eq:dualBYSFQ:pf-3}
\begin{bmatrix}
    \hm \what F_i^{(\rmd)} & 0 & I & -\what X_i^{(\rmd)} \\
   -\what Y_i^{(\rmd)} & I & 0 & \hm \what E_i^{(\rmd)}
    \end{bmatrix}\begin{bmatrix}
  \begin{bmatrix}
    I_m & -X_i^{(\rmd)} \\
    0 & \hm E_i^{(\rmd)}
  \end{bmatrix}  Q_2 \\
  \begin{bmatrix}
      \hm F_i^{(\rmd)} & 0 \\
      -Y_i^{(\rmd)} & I_n
  \end{bmatrix}Q_1
\end{bmatrix}
=0.
\end{equation}
Now we are ready to prove the conclusion of the theorem by induction.
Comparing \eqref{eq:SFQ:i=0} with \eqref{eq:scrA0B0:SFQd-2}, we find that
$X_i^{(\rmd)}=Y_i$, $Y_i^{(\rmd)}=X_i$, $E_i^{(\rmd)}=F_i$, and $F_i^{(\rmd)}=E_i$ hold for $i=0$.
Suppose that they hold for $i$ and we now consider $i+1$. Since $X_i^{(\rmd)}=Y_i$, $Y_i^{(\rmd)}=X_i$, $E_i^{(\rmd)}=F_i$, and $F_i^{(\rmd)}=E_i$,
the second large matrices in \eqref{eq:dualBYSFQ:pf-1} and \eqref{eq:dualBYSFQ:pf-3}
are the same. By the uniqueness that determine the first matrices in the special structured form
in each of the equations, we conclude that
$\what X_i^{(\rmd)}=\what Y_i$, $\what Y_i^{(\rmd)}=\what X_i$, $\what E_i^{(\rmd)}=\what F_i$,
and $\what F_i^{(\rmd)}=\what E_i$. Finally, by \eqref{eq:EFXY-iter-SFQ}, we have
$E_{i+1}^{(\rmd)}=\what E_i^{(\rmd)}E_i^{(\rmd)}=\what F_i\what F_i=F_{i+1}$,
$F_{i+1}^{(\rmd)}=\what F_i^{(\rmd)}F_i^{(\rmd)}=E_{i+1}$,
$X_{i+1}^{(\rmd)}=X_i^{(\rmd)}+\what X_i^{(\rmd)}E_i^{(\rmd)}=Y_{i+1}$, and
$Y_{i+1}^{(\rmd)}=Y_i^{(\rmd)}+\what Y_i^{(\rmd)}F_i^{(\rmd)}=X_{i+1}$,
completing the induction proof.
\end{proof}

\section{Convergence Analysis}\label{sec:DA-conv-regu}
When the structure-preserving doubling algorithm, SDASFQ, is applied to both
$\scrA_0-\lambda\scrB_0$ and $\scrA_0^{(\rmd)}-\lambda\scrB_0^{(\rmd)}$, two sequences of
matrix pencils $\{\scrA_i-\lambda\scrB_i\}_{i=0}^{\infty}$ and
$\{\scrA_i^{(\rmd)}-\lambda\scrB_i^{(\rmd)}\}_{i=0}^{\infty}$ are produced, assuming no breakdown occurs.
The two sequences are ``identical'' in the sense of Theorem~\ref{thm:dualBYSFQ}. Moreover,
\begin{subequations}\label{eq:DI:i:ALL}
\begin{align}
\scrA_iQ_1^{\T}\begin{bmatrix}
         I \\
         X
       \end{bmatrix}\hphantom{\scrN^{2^i}}&=\scrB_iQ_1^{\T}\begin{bmatrix}
         I \\
         X
       \end{bmatrix}\scrM^{2^i}\quad \mbox{for $i=0,1,\ldots$}, \label{eq:DI:i:ALL-1}\\
\scrA_i Q_2^{\T}\begin{bmatrix}
  Y \\
  I
\end{bmatrix}\scrN^{2^i}
&=\scrB_i Q_2^{\T}\begin{bmatrix}
  Y \\
  I
\end{bmatrix}\hphantom{\scrM^{2^i}}\quad\mbox{for $i=0,1,\ldots$}.\label{eq:DI:i:ALL-2}
\end{align}
\end{subequations}
For practical considerations, we should keep in mind, in the preset-up of transforming
$\scrA-\lambda\scrB$ to $\scrA'-\lambda\scrB'$, that  certain conditions on $\rho(\scrM)$ and $\rho(\scrN)$ should be
imposed to ensure convergence.
%Now is the time for that.
It also can be verified that \eqref{eq:DI:i:ALL-2} is the same as
$$
\scrA_i^{(\rmd)}(Q_1^{(\rmd)})^{\T}\begin{bmatrix}
         I \\
         Y
       \end{bmatrix}
=\scrB_i^{(\rmd)}(Q_1^{(\rmd)})^{\T}\begin{bmatrix}
         I \\
         Y
       \end{bmatrix}\scrN^{2^i}\quad \mbox{for $i=0,1,\ldots$}
$$
by \Cref{thm:dualBYSFQ}.

Suppose that we are seeking a special
solution $X=\Phi$ to the primal equation \eqref{eq:NME-dual:SFQp}
and possibly its associated solution $Y=\Psi$ to the dual equation \eqref{eq:NME-dual:SFQd}, or equivalently the ones to \eqref{eq:EigEqs}.
The equations in \eqref{eq:DI:i:ALL} hold  with $X=\Phi$ and $Y=\Psi$, respectively, too.
Plug in the expressions for $\scrA_i$ and $\scrB_i$ in \eqref{eq:preservation-SFQ} into
the equations in \eqref{eq:DI:i:ALL} to get
\begin{subequations}\label{eq:ErrEq:SFQ}
\begin{alignat}{2}
E_i&=\big[Q_{11}^{\T}+Q_{21}^{\T}\Phi-Y_i(Q_{12}^{\T}+Q_{22}^{\T}\Phi)\big]\scrM^{2^i}, &\quad \Phi-X_i&=F_i(Q_{12}^{\T}+Q_{22}^{\T}\Phi)\scrM^{2^i}, \label{eq:ErrEq:SFQ-X}\\
F_i&=\big[Q_{22}+Q_{21}\Psi-X_i(Q_{12}+Q_{11}\Psi)\big]\scrN^{2^i}, &\quad \Psi-Y_i&=E_i(Q_{12}+Q_{11}\Psi)\scrN^{2^i}. \label{eq:ErrEq:SFQ-Y}
\end{alignat}
\end{subequations}
Consequently,
\begin{subequations}\label{eq:ErrEqSF1'}
\begin{alignat}{2}
\Phi-X_i&=\big[Q_{22}+Q_{21}\Psi-X_i(Q_{12}+Q_{11}\Psi)\big]\scrN^{2^i}(Q_{12}^{\T}+Q_{22}^{\T}\Phi)\scrM^{2^i}, \label{eq:ErrEq:SFQ-X'}\\
\Psi-Y_i&=\big[Q_{11}^{\T}+Q_{21}^{\T}\Phi-Y_i(Q_{12}^{\T}+Q_{22}^{\T}\Phi)\big]\scrM^{2^i}(Q_{12}+Q_{11}\Psi)\scrN^{2^i}. \label{eq:ErrEq:SFQ-Y'}
\end{alignat}
\end{subequations}

\begin{theorem}\label{thm:DA-cvg-SFQ:regu}
Suppose that there are solutions $X=\Phi$ and $Y=\Psi$ to the equations in \eqref{eq:NME-dual:SFQp}
and \eqref{eq:NME-dual:SFQd}, respectively, such that
\begin{equation}\label{eq:DA-cvg:regu-dfn}
\rho(\scrM)\cdot\rho(\scrN)<1,
\end{equation}
and suppose that {\rm SDASFQ (Algorithm~\ref{alg:DA-SFQ})} executes without any breakdown
when applied to $\scrA_0-\lambda\scrB_0$ in {\em SFQ}~\eqref{eq:SFQ:i=0},
i.e., all the inverses  exist during the doubling iteration process.
Then $X_i$ and $Y_i$ converge to $\Phi$ and $\Psi$ quadratically, respectively, and moreover,
\begin{subequations}\label{eq:cov-SF1:regu}
\begin{gather}
\limsup_{i\to\infty}\|X_i-\Phi\|^{1/2^i},\quad
        \limsup_{i\to\infty}\|Y_i-\Psi\|^{1/2^i}\le\rho(\scrM)\cdot\rho(\scrN), \label{eq:cov-SF1:regu-1}\\
\limsup_{i\to\infty}\|E_i\|^{1/2^i}\le\rho(\scrM), \quad
        \limsup_{i\to\infty}\|F_i\|^{1/2^i}\le\rho(\scrN). \label{eq:cov-SF1:regu-2}
\end{gather}
\end{subequations}
%In the case when $m=n$, $F_0^{\trans}=\pm E_0$, and $X_0$ and $Y_0$ are $\trans$-symmetric, i.e.,
%$X_0^{\trans}=X_0$ and $Y_0^{\trans}=Y_0$, $\rho(\scrM)\cdot\rho(\scrN)<1$ implies
%$\rho(\scrM)=\rho(\scrN)<1$ and vice versa.
\end{theorem}

What this theorem says is that under suitable conditions, both $X_i$ and $Y_i$ converge
to $\Phi$ and $\Psi$ at least as fast as $[\rho(\scrM)\cdot\rho(\scrN)]^{2^i}$.
%In the later applications of this theorem to \CARE\   and \DARE\   from
%the optimal control theory, and \MARE\ from the transport theory and others, these conditions will be verified to hold.

\begin{proof}
We will prove the first inequality in \eqref{eq:cov-SF1:regu-1} and the second inequality in \eqref{eq:cov-SF1:regu-2} only, and
the other two inequalities there can be proved in a similar way.
Let $Z_i=\Phi-X_i$. It follows from  \eqref{eq:ErrEq:SFQ-X'} that
\begin{alignat*}{2}
&&Z_i&=\big[Q_{22}+Q_{21}\Psi-(\Phi-Z_i)(Q_{12}+Q_{11}\Psi)\big]\scrN^{2^i}(Q_{12}^{\T}+Q_{22}^{\T}\Phi)\scrM^{2^i} \\
&\Rightarrow&\,\,\|Z_i\|&\leq \Big[\|Q_{22}+Q_{21}\Psi-\Phi(Q_{12}+Q_{11}\Psi)\|+\|Z_i\|\,\|Q_{12}+Q_{11}\Psi\|\Big]\\
 &&&\qquad\times\|\scrN^{2^i}\|\,\|Q_{12}^{\T}+Q_{22}^{\T}\Phi\|\,\|\scrM^{2^i}\|.
\end{alignat*}
Let $\epsilon_i=\|\scrN^{2^i}\|\,\|\scrM^{2^i}\|$. It is known that
$$
\lim_{i\to\infty}(\epsilon_i)^{1/2^i}=\lim_{i\to\infty}\|\scrN^{2^i}\|^{1/2^i}\|\scrM^{2^i}\|^{1/2^i}
   =\rho(\scrM)\cdot\rho(\scrN).
$$
In particular, $\epsilon_i\to 0$ as $i\to\infty$ if $\rho(\scrM)\cdot\rho(\scrN)<1$ which is assumed.
Therefore, for $i$ sufficiently large, $1-\|Q_{12}+Q_{11}\Psi\|\,\|Q_{12}^{\T}+Q_{22}^{\T}\Phi\|\,\epsilon_i>0$ and thus
$$
\|Z_i\|\le\frac {\|Q_{22}+Q_{21}\Psi-\Phi(Q_{12}+Q_{11}\Psi)\|\,\|Q_{12}^{\T}+Q_{22}^{\T}\Phi\|}
                {1-\|Q_{12}+Q_{11}\Psi\|\,\|Q_{12}^{\T}+Q_{22}^{\T}\Phi\|\,\epsilon_i}
                \epsilon_i\to 0
\quad\mbox{as $i\to\infty$}.
$$
Moreover, as $i\to\infty$
$$
\|Z_i\|^{1/2^i}\le\left[\frac {\|Q_{22}+Q_{21}\Psi-\Phi(Q_{12}+Q_{11}\Psi)\|\,\|Q_{12}^{\T}+Q_{22}^{\T}\Phi\|}
                {1-\|Q_{12}+Q_{11}\Psi\|\,\|Q_{12}^{\T}+Q_{22}^{\T}\Phi\|\,\epsilon_i}\right]^{1/2^i} (\epsilon_i)^{1/2^i}
   \to\rho(\scrM)\cdot\rho(\scrN).
$$
This proves the first inequality in \eqref{eq:cov-SF1:regu-1}, which also implies that $\|X_i\|$ is uniformly bounded, i.e.,
there exists a constant $\omega$ such that $\|X_i\|\le\omega$ for all $i$. We have from the first equation in
\eqref{eq:ErrEq:SFQ-Y} that
\begin{align*}
   &\|F_i\|\le \big(\|Q_{22}+Q_{21}\Psi\|+\omega\|Q_{12}+Q_{11}\Psi\|\big)\|\scrN^{2^i}\| \\
\Rightarrow\quad & \|F_i\|^{1/2^i}\le\big(\|Q_{22}+Q_{21}\Psi\|+\omega\|Q_{12}+Q_{11}\Psi\|\big)^{1/2^i}\|\scrN^{2^i}\|^{1/2^i}\to\rho(\scrN)
\end{align*}
as $i\to\infty$. This proves the second inequality in \eqref{eq:cov-SF1:regu-2}.
\end{proof}

\begin{remark}\label{rk:eig-distr-SF1}
{\rm
In Theorem~\ref{thm:DA-cvg-SFQ:regu},
the assumption that there are solutions $X=\Phi$ and $Y=\Psi$ to the equations in   \eqref{eq:NME-dual:SFQp}
and \eqref{eq:NME-dual:SFQd}, respectively, such that
$\rho(\scrM)\cdot\rho(\scrN)<1$ has an important implication on the eigenvalue distribution of $\scrA_0-\lambda\scrB_0$, namely, it has $m$
eigenvalues in $\bbD_{\varrho-}$ and the other $n$ eigenvalues in $\bbD_{\varrho+}$, where $\varrho$ is any number such that
$\rho(\scrM)<\varrho<1/\rho(\scrN)$.
Moreover, $\cR(Q_1^{\T}\begin{bmatrix}
             I_m \\
             \Phi
           \end{bmatrix})$ is the eigenspace associated with the $m$ eigenvalues in $\bbD_{\varrho-}$ and
$\cR(Q_2^{\T}\begin{bmatrix}
       \Psi \\
       I_n
     \end{bmatrix})$ is the eigenspace associated with the $n$ eigenvalues in $\bbD_{\varrho+}$. To see all these,
we note
$\rho(\scrN)<1/\varrho$ and $\eig(\scrN)\subset\eig(\scrA_0^{(\rmd)},\scrB_0^{(\rmd)})$, and thus
by Theorem~\ref{thm:SFQ=SFQd-eig} we know $\scrA_0-\lambda\scrB_0$ has $n$ eigenvalues in $\bbD_{\varrho+}$
which are the reciprocals of those in $\eig(\scrN)$.
Before this section, our discussion and motivation have been focusing on the situation
that the spectrum of $\scrA_0-\lambda B_0$ can be split into two parts: in $\bbD_-$ or $\bbD_+$.
It is for the ease of the presentation and understanding.
}
\end{remark}

A key condition of Theorem~\ref{thm:DA-cvg-SFQ:regu}
is \eqref{eq:DA-cvg:regu-dfn}, which we  call the {\em regular case}. As we have seen,
the analysis is rather simple. It is interesting to note that in the regular case
it is possible that one of $\rho(\scrM)$ and $\rho(\scrN)$ is bigger than $1$
so long as their product is less than $1$.
%are less than $1$, as well as when one of them .
%
The error equations in \eqref{eq:ErrEqSF1'} for SDASFQ  strongly suggest that
the  doubling algorithms may diverge if $\rho(\scrM)\cdot\rho(\scrN)>1$.
%Specifically, we believe that the following statement
%is true: if $\rho(\scrM)\cdot\rho(\scrN)>1$ for all solution pairs $(X,\scrM)$ to \eqref{eq:DI:i:ALL-1}
%and all solution pairs $(Y,\scrN)$ to \eqref{eq:DI:i:ALL-2}, then the doubling algorithms may diverge.
%
However, the structure-preserving doubling algorithms may still converge under suitable conditions, even if
\begin{equation}\label{eq:DA-cvg:crit-dfn-0}
\rho(\scrM)\cdot\rho(\scrN)=1,
\end{equation}
which we will call the {\em critical case\/}.
% when \eqref{eq:DA-cvg:crit-dfn-1} holds.
The convergence analysis for the critical case is rather complicated. In order not to disrupt
the main flow of the article, we defer the analysis to Appendix~\ref{sec:DA-conv-crit}.
%The goal of this section is to analyse the convergence of SDASFQ for the critical case.
%We largely employ the techniques used in \cite[section~3.8]{hull:2018}.

%but it may not hold. For example, in the case when $\scrA_0-\lambda\scrB_0$ is $\trans$-symplectic, it may have unimodular eigenvalues. In such a case,
%the structure-preserving doubling algorithm for SF1 or SF2 may not converge, or converges just linearly if the
%partial multiplicities of all unimodular eigenvalues are favorable. But the proofs are much more involved
%than those of Theorems~\ref{thm:DA-cvg-SFQ:regu}  and \ref{thm:DA-cvg-SF2:regu}.
%They are dealt with in section~\ref{sec:DA-conv-crit} next.

\section{Q-Doubling Algorithm and Implementation}\label{sec:QDA}
In \cref{sec:new-basis}, we argued the necessity of introducing permutation matrices $Q_1$ and $Q_2$
in order to represent robustly basis matrices for the eigenspaces of interest as in \eqref{eq:QBases}, and we also commented that,
unlikely they are known {\em apriori\/}, except in handful cases such as \DARE s, \CARE s, or \MARE s for which both are the identity matrices. In this section, we will describe ways of
constructing $Q_1$ and $Q_2$ initially and adaptively updating them during the doubling iteration. We will assume that we are at the stage \eqref{eq:trans:scrAB2scrB'-2}.

For ease of presentation, we will adopt MATLAB-like convention to
access the entries of vectors and matrices.
$i:j$ is the set of integers from $i$ to $j$ inclusive.
For a vector $u$ and an matrix $X$, $u_{(j)}$
is $u$'s $j$th entry,
$X_{(i,j)}$ is $X$'s
$(i,j)$th entry;
$X$'s submatrices $X_{(k:\ell,i:j)}$, $X_{(k:\ell,:)}$, and $X_{(:,i:j)}$
consist of intersections of row $k$ to row $\ell$ and  column $i$ to column $j$,
row $k$ to row $\ell$, and column $i$ to column $j$, respectively.

\subsection{Initialization}\label{ssec:init}
Recall that at the end of the initialization we need to get
\begin{equation}\label{eq:P-dfn:init4Qi}
P\left[\scrA'Q_1^{\T},\scrB'Q_2^{\T}\right]
         =\begin{bmatrix}
             \hm E_0 & 0 & I_m & -Y_0 \\
             -X_0    & I_n & 0 & \hm F_0
          \end{bmatrix},
\end{equation}
where $Q_1$ and $Q_2$ are two permutation matrices and $P$ is some nonsingular matrix that will not actually be needed
at all  rather than their existence.
Since we do not know  $Q_1$ and $Q_2$, we cannot use \Cref{thm:init:SF1} for initialization.
In this subsection, we explain three ideas that may produce \eqref{eq:P-dfn:init4Qi}
with proper permutations $Q_1$ and $Q_2$ so that $\|X_0\|$ and $\|Y_0\|$ are hopefully not big.
Each idea leads to two versions of reduction from starting with $\scrA'$ or with $\scrB'$ first. Mathematically,
the two versions are the same, but  may behave differently numerically for given $\scrA'-\lambda B'$.

\subsubsection*{Idea 1}
The first version out of this idea consists of three steps:
\begin{enumerate}[(1)]
  \item We perform the reverse Gaussian elimination on $\scrA'$ with complete pivoting
        within the last $n$ rows,   starting from position
        $(m+n,m+n)$ backwards towards position $(m+1,m+1)$, for $n$ steps. If successful, we collect all column permutations due to row pivoting into $Q_1^{\T}$ and, by this time, the $n$-by-$n$ bottom-right submatrix of the updated matrix $\scrA'$ is a lower triangular matrix, which will have to be assumed nonsingular in order to move
        forward\footnote {Otherwise we may try the second version out of this idea or
                          any of the other reductions from {\bf Idea 2} and {\bf Idea 3} below.}.
        Now we collect all matrices
        from the $n$ Gaussian elimination steps and the one from multiplying the last $n$ rows by
        the inverse of the lower triangular matrix into $P_1$. At this point,
        we have
        \begin{equation}\label{eq:P-dfn:init4Qi-2}
        P_1\scrA'Q_1^{\T}=\begin{bmatrix}
                     \hm \wtd E_0 & 0  \\
                     -\wtd X_0    & I_n
                  \end{bmatrix}.
        \end{equation}
        During the reverse Gaussian process to obtain \eqref{eq:P-dfn:init4Qi-2}, $P_1$ is composed of
        a sequence of simple matrix multiplications from the left and row permutations. Those actions are
        applied simultaneously to $\scrB'$ as well, i.e., when \eqref{eq:P-dfn:init4Qi-2} is obtained, so is $P_1\scrB'$.
  \item Next, we perform the usual Gaussian elimination on $P_1\scrB'$ with complete pivoting
        within the first $m$ rows,  starting from position $(1,1)$ towards position $(m,m)$, for $m$ steps. If successful, we collect all column permutations due to row pivoting into $Q_2^{\T}$ and, by this time, the $m$-by-$m$ top-left submatrix of
        the updated matrix $P_1\scrB'$  is an  upper triangular matrix, , which will be assumed nonsingular to move forward. Now we collect all matrices
        from the $m$ Gaussian elimination steps and the one from multiplying the first $m$ rows by
        the inverse of the upper triangular matrix into $P_2$. At this point,
        $$
        P_2P_1\scrB'Q_2^{\T}=\begin{bmatrix}
                      I_m & -Y_0 \\
                      0 & \hm F_0
                  \end{bmatrix}.
        $$
  \item Finally, let $P=P_2P_1$. It is noted that $P\scrA'Q_1^{\T}=P_2(P_1\scrA'Q_1^{\T})$ preserves the block structured
        in \eqref{eq:P-dfn:init4Qi-2} but with $\wtd E_0$ and $\wtd X_0 $ updated to $E_0$ and $X_0 $ in \eqref{eq:P-dfn:init4Qi}.
         Hence, we will have \eqref{eq:P-dfn:init4Qi}.
\end{enumerate}
At two separate places above, a lower triangular matrix and an upper triangular matrix have to be assumed
nonsingular. If, however, either one of them is singular, then {\bf Idea 1} will fail to  produce \eqref{eq:P-dfn:init4Qi}.

Similarly to what we just described, the second version is the same except switching the roles
of $\scrA'$ or with $\scrB'$, i.e., to first work on $\scrB'$  by
performing the usual Gaussian elimination with complete pivoting within the first $m$ rows of $\scrB'$
and at the same time updating $\scrA'$, and then work on the updated $\scrA'$ by
performing the reverse Gaussian elimination with complete pivoting within the last $n$ rows. Detail is omitted.

\subsubsection*{Idea 2}
The first version out of this idea also consists of three steps:
\begin{enumerate}[(1)]
  \item We perform the reverse Gaussian elimination on $\scrA'$ with complete pivoting within {\em the entire matrix},   starting from position
        $(m+n,m+n)$ backwards towards position $(m+1,m+1)$, for $n$ steps. Again, if successful,
        we have \eqref{eq:P-dfn:init4Qi-2} and $P_1\scrB'$.

        This step differs from the first step in {\bf Idea 1} in selecting pivots within {\em the entire matrix}.
%  \item Next, we perform the usual Gaussian elimination on $P_1\scrB'$ with complete pivoting
%        within the first $m$ rows,  starting from position $(1,1)$ towards position $(m,m)$, for $m$ steps. If successful, we collect all column permutations due to row pivoting into $Q_2^{\T}$ and, by this time, the $m$-by-$m$ top-left submatrix of
%        the updated matrix $P_1\scrB'$  is a nonsingular upper triangular matrix. Now we collect all matrices
%        from the $m$ Gaussian elimination steps and the one from multiplying the first $m$ rows by
%        the inverse of the nonsingular upper triangular matrix into $P_2$. At this point,
%        $$
%        P_2P_1\scrB'Q_2^{\T}=\begin{bmatrix}
%                      I_m & -Y_0 \\
%                      0 & \hm F_0
%                  \end{bmatrix}.
%        $$
%  \item Finally, let $P=P_2P_1$. It is noted that $P\scrA'Q_1^{\T}=P_2(P_1\scrA'Q_1^{\T})$ preserves the block structured
%        in \eqref{eq:P-dfn:init4Qi-2} but with $\wtd E_0$ updated to $E_0$ in \eqref{eq:P-dfn:init4Qi}.
%         Hence, we will have \eqref{eq:P-dfn:init4Qi}.
\end{enumerate}
Steps (2) and (3) are the same as in {\bf Idea 1}. This is because once \eqref{eq:P-dfn:init4Qi-2} is in place,
we do not want to destroy the structure in its right-hand side while working on $P_1\scrB'$.
As in {\bf Idea~1}, {\bf Idea 2} will also come to the points when a lower triangular matrix
needs to be inverted and  when an upper triangular matrix needs to be inverted.
If either one of them is singular, then {\bf Idea 2} will fail to  produces \eqref{eq:P-dfn:init4Qi}.
But it is expected that {\bf Idea 2} should have a greater chance to succeed because Step (1) does complete pivoting
within {\em the entire matrix}, and for the same reason it likely will
produce an initialization with smaller  $\|X_0\|$ and $\|Y_0\|$.

An alternative to what we just described is to first work on $\scrB'$ and then on updated $\scrA'$.
This is the second version.
Detail is omitted.
%by
%performing the usual Gaussian elimination with complete pivoting within the first $m$ rows of $\scrB'$
%and at the same time updating $\scrA'$, and then work on the updated $\scrA'$ by
%performing the reverse Gaussian elimination with complete pivoting within the last $n$ rows. Detail is omitted.

\subsubsection*{Idea 3}
In both {\bf Idea 1} and {\bf Idea 2}, we proposed to work on exclusive either $\scrA'$ or $\scrB'$ first while
updating the other with the same actions
%when they are for left matrix multiplications or row permutations
and then switch to work
on the other matrix while preserving the structure of the first matrix after Gaussian elimination.

In this {\bf Idea 3}, we propose to work on $\scrA'$ and $\scrB'$ alternatingly. The first version
goes as follows:
\begin{enumerate}[(1)]
  \item perform the first step of the reverse Gaussian elimination on $\scrA'$ with complete pivoting
within the entire $\scrA'$ followed by the first step of the Gaussian elimination on updated $\scrB'$ with complete pivoting within the first $m+n-1$ rows;
  \item perform the second step of the reverse Gaussian elimination on updated $\scrA'$ with complete pivoting
within updated $\scrA'_{(2:m+n-1,1:m+n-1)}$ followed by one step of the Gaussian elimination on updated $\scrB'$ with complete pivoting within updated $\scrB'_{(2:m+n-2,2:m+n)}$;
  \item repeat the process.
\end{enumerate}
Working on $\scrB'$ first and then on updated $\scrA'$ in the same alternating way yields the second version.

\subsection{During Doubling Iteration}\label{ssec:adaptiveQ1Q2}
We have to monitor $\|X_i\|$ and $\|Y_i\|$ during the doubling iteration. If one or both of them are deemed too huge, we will have to update $Q_1$ and/or $Q_2$ to bring $\|X_i\|$ and $\|Y_i\|$ down. But how big is too big? Earlier \Cref{thm:BdDABasis} provides us a guidance as we will explain.
\Cref{lm:simple-inv} is well-known and will become useful later.

\begin{lemma}\label{lm:simple-inv}
Let $\bu,\,\bv\in\bbC^k$. If $1-\bv^{\T}\bu\ne 0$, then $I_k-\bu\bv^{\T}$ is invertible and
$$
(I_k-\bu\bv^{\T})^{-1}=I_k+\frac {\bu\bv^{\T}}{1-\bv^{\T}\bu}.
$$
\end{lemma}

Return to the doubling iteration. At convergence, the limits of $X_i$ and $Y_i$ will be placed into \eqref{eq:QBases}
to form two basis matrices  of the corresponding
eigenspaces, which can also be derived from some orthonormal basis matrices. Thus, ideally both $\|X\|_2$ and $\|Y\|_2$ can be controlled to be no more than
$\sqrt{nm+1}$ with appropriate permutation matrices $Q_1$ and $Q_2$. This suggests that it is reasonable to control the entries
of $X_i$ and $Y_i$ during the doubling iteration to no more than some multiple of
$\sqrt{nm+1}$. Currently, in our numerical examples, we check if there is any entry of either $X_i$ or $Y_i$ that is bigger than
\begin{equation}\label{eq:maxXY}
\tau=\max\{10^3, 10\sqrt{nm+1}\}.
\end{equation}
If there is one, then we will have bring it down. The simplest solution is to call upon one of the initialization procedures in subsection~\ref{ssec:init}
to yield
\begin{equation}\label{eq:AiBi-update}
P[(\scrA_iQ_1^{\T})S_1^{\T},(\scrB_iQ_2^{\T})S_2^{\T}]
         =\begin{bmatrix}
             \hm \wtd E_i & 0 & I_m & -\wtd Y_i \\
             -\wtd X_i    & I_n & 0 & \hm \wtd F_i
          \end{bmatrix}
\end{equation}
and hopefully the 2-norms of $\wtd X_i$ and $\wtd Y_i$ are (much) smaller than
the preselected $\tau$. Now replace $(X_i,Y_i,E_i,F_i)$ with
$(\wtd X_i,\wtd Y_i,\wtd E_i,\wtd F_i)$, update $Q_1$ to $S_1Q_1$ and $Q_2$ to $S_2Q_2$,
and then resume the doubling process from corresponding updated $\scrA_i$ and $\scrB_i$.

%What happens if the 2-norms of some columns of $\wtd X_i$ and $\wtd Y_i$ are still bigger than
%$\tau$? This could be an indication that the chosen $\tau$ is unrealistically small and should be made bigger, such as twice as
%much.

What we have just discussed is to reuse the initialization procedures in subsection~\ref{ssec:init}. In doing so, we completely ignore the structures already in the current $\scrA_iQ_1^{\T}$ and $\scrB_iQ_2^{\T}$. Conceivably, the structures
can be taken advantage of for better numerical efficiency. Recall that
$$
\left[\scrA_iQ_1^{\T},\scrB_iQ_2^{\T}\right]
         =\begin{bmatrix}
             \hm E_i & 0   & I_m & -Y_i \\
             -X_i    & I_n & 0 & \hm F_i
          \end{bmatrix}.
$$
For ease of presentation, let's drop the subscript $i$ to $X_i,Y_i,E_i,F_i$. Our idea is to repeatedly apply
the following two actions to bring $X$ and $Y$ under control.
% where $X_{(j,\ell)}$
%denotes the $(j,\ell)$th entry of $X$, and $X_{(:,\ell)}$ the $\ell$th column of $X$
%and similarly for $Y_{(j,\ell)}$ and $Y_{(:,\ell)}$.
%, focusing one column of large magnitude at a time.
\begin{description}%[(i)]
  \item[Action (i)] Suppose $X_{(j,\ell)}$ is the largest entry of $X$ in magnitude, i.e.,
        \begin{equation}\label{eq:max|X(j,ell)|}
        (j,\ell)=\arg\max_{(i,k)}\{|X_{(i,k)}|\,:\,1\le i\le n,\,1\le k\le m\},
        \end{equation}
         and at the same time also
        $|X_{(j,\ell)}|>\tau$, the maximum magnitude tolerance, e.g., as in \eqref{eq:maxXY}.
        Denote by $\bx=X_{(:,\ell)}$ and $\bh=E_{(:,\ell)}$.
%   of $X$ is too big, i.e., $\|\bx\|_2>\tau$,
%        and let $\bx_{(j)}=e_j^{\T}\bx$ be the largest entry of $\bx$ in magnitude.
%        Immediately $|\bx_{(j)}|\ge\tau/\sqrt n$.
        Let $S^{\T}$ be
        the permutation matrix such that
        $$
        \begin{bmatrix}
             \hm E & 0  \\
             -X    & I_n
        \end{bmatrix}S^{\T}=\begin{bmatrix}
                            E-\bh\be_{\ell}^{\T} & \bh\be_j^{\T} \\
                           -X+(\bx+\be_j)\be_{\ell}^{\T} & I_n-(\bx+\be_j)\be_j^{\T}
                         \end{bmatrix},
        $$
        i.e., swapping the $\ell$th and $(m+j)$th columns of the matrix on the far-left.
        By \Cref{lm:simple-inv}, we get
        $$
        \Big[I_n-(\bx+\be_j)\be_j^{\T}\Big]^{-1}
          =I_n+\frac {(\bx+\be_j)\be_j^{\T}}{1-\be_j^{\T}(\bx+\be_j)}
          =I_n-\frac {\bx+\be_j}{\be_j^{\T}\bx}\,\be_j^{\T}.
        $$
        Note $\be_j^{\T}\bx=X_{(j,\ell)}$, the largest entry of $X$ in magnitude. Let
        $$
        P_1=\begin{bmatrix}
             I_m & \\
              & \Big[I_n-(\bx+\be_j)\be_j^{\T}\Big]^{-1}
            \end{bmatrix}, \quad
        P_2=\begin{bmatrix}
              I_m & -\bh\be_j^{\T} \\
              0 & I_n
            \end{bmatrix}.
        $$
        We have
        $$
        P_1\left[(\scrA_iQ_1^{\T})S^{\T},\scrB_iQ_2^{\T}\right]
         =\begin{bmatrix}
             E-\bh\be_{\ell}^{\T} & \bh\be_j^{\T}  & I_m & -Y \\
             -\wtd X    & I_n & 0 & \hm\wtd F
          \end{bmatrix},
        $$
        where
        \begin{subequations}\label{eq:bigX-update}
        \begin{align}
        \wtd X &=\Big[I_n-\frac {(\bx+\be_j)\be_j^{\T}}{\be_j^{\T}\bx}\Big]\big[X-(\bx+\be_j)\be_{\ell}^{\T}\big]
                  \nonumber\\
             &=X+\frac{(\bx+\be_j)\be_{\ell}^{\T}}{\be_j^{\T}\bx}
               -\frac {(\bx+\be_j)}{\be_j^{\T}\bx}\be_j^{\T}X \nonumber\\
             &=X+\frac {\bx+\be_j}{\be_j^{\T}\bx}\,\big[\be_{\ell}^{\T}-\be_j^{\T}X\big],
             \label{eq:bigX-update4X}\\
        \wtd F &=\Big[I_n-\frac {(\bx+\be_j)\be_j^{\T}}{\be_j^{\T}\bx}\Big]F \nonumber\\
                &=F-\frac {\bx+\be_j}{\be_j^{\T}\bx}\,\big(\be_j^{\T}F\big), \label{eq:bigX-update4F}
        \end{align}
        and
        $$
        P_2P_1\left[(\scrA_iQ_1^{\T})S^{\T},\scrB_iQ_2^{\T}\right]
         =\begin{bmatrix}
             \hm\wtd E & 0  & I_m & -\wtd Y \\
             -\wtd X    & I_n & 0 & \hm\wtd F
          \end{bmatrix},
        $$
        which is in the desired form,
        where
        \begin{align}
        \wtd E&=E-\bh(\be_{\ell}^{\T}-\be_j^{\T}\wtd X) %\nonumber\\
              =E+\frac \bh{\be_j^{\T}\bx}\,\big[\be_{\ell}^{\T}-e_j^{\T}X\big], \label{eq:bigX-update4E}\\
        \wtd Y&=Y+\bh(\be_j^{\T}\wtd F) %\nonumber\\
              =Y-\frac \bh{\be_j^{\T}\bx}\,\big(\be_j^{\T}F\big). \label{eq:bigX-update4Y}
        \end{align}
        \end{subequations}
        Along with $Q_1\leftarrow SQ_1$, formulas in \eqref{eq:bigX-update} shows how to update
        the $Q$-standard form when one of the entries of $X$ is deemed too big.

  \item[Action (ii)]  Suppose $Y_{(j,\ell)}$ is the largest entry of $Y$ in magnitude, and at the same time also
        $|Y_{(j,\ell)}|>\tau$, the maximum magnitude tolerance, e.g., as in \eqref{eq:maxXY}.
        Denote by $\by=Y_{(:,\ell)}$ and  $\bh=F_{(:,\ell)}$.
        Let $S^{\T}$ be
        the permutation matrix such that
        $$
        \begin{bmatrix}
             I_m & -Y \\
             0 & \hm F
        \end{bmatrix}S^{\T}=\begin{bmatrix}
                            I_m-(\by+\be_j)\be_j^{\T} & -Y+(\by+\be_j)\be_{\ell}^{\T} \\
                            \bh\be_j^{\T} & F-\bh\be_{\ell}^{\T}
                         \end{bmatrix},
        $$
        i.e., swapping the $j$th and $(m+\ell)$th columns of the matrix on the far-left.
        By \Cref{lm:simple-inv}, we get
        $$
        \Big[I_m-(\by+\be_j)\be_j^{\T}\Big]^{-1}
          =I_m+\frac {(\by+\be_j)\be_j^{\T}}{1-\be_j^{\T}(\by+\be_j)}
          =I_m-\frac {(\by+\be_j)\be_j^{\T}}{\be_j^{\T}\by}.
        $$
        Let
        $$
        P_1=\begin{bmatrix}
              \Big[I_m-(\by+\be_j)\be_j^{\T}\Big]^{-1} & \\
             & I_n
            \end{bmatrix}, \quad
        P_2=\begin{bmatrix}
              I_m & 0\\
              -\bh\be_j^{\T} & I_n
            \end{bmatrix}.
        $$
        We have
        $$
        P_1\left[\scrA_iQ_1^{\T},(\scrB_iQ_2^{\T})S^{\T}\right]
         =\begin{bmatrix}
             \wtd E & 0  & I_m & -\wtd Y \\
             - X    & I_n & \bh\be_j^{\T} & F-\bh\be_{\ell}^{\T}
          \end{bmatrix},
        $$
        where
        \begin{subequations}\label{eq:bigY-update}
        \begin{align}
        \wtd E &=\Big[I_m-\frac {(\by+\be_j)\be_j^{\T}}{\be_j^{\T}\by}\Big]E \nonumber \\
                &=E-\frac {\by+\be_j}{\be_j^{\T}\by}\,\big(\be_j^{\T}E\big), \label{eq:bigY-update4E}\\
        \wtd Y &=\Big[I_n-\frac {(\by+\be_j)\be_j^{\T}}{\be_j^{\T}\by}\Big]
                  \big[Y-(\by+\be_j)\be_{\ell}^{\T}\big] \nonumber\\
             &=Y+\frac{(\by+\be_j)\be_{\ell}^{\T}}{\be_j^{\T}\by}
                 -\frac {(\by+\be_j)}{\be_j^{\T}\by}\be_j^{\T}Y \nonumber\\
             &=Y+\frac {\by+\be_j}{\be_j^{\T}\by}\,\big[\be_{\ell}^{\T}-\be_j^{\T}Y\big],
                \label{eq:bigY-update4Y}
        \end{align}
        and
        $$
        P_2P_1\left[\scrA_iQ_1^{\T},(\scrB_iQ_2^{\T})S^{\T}\right]
         =\begin{bmatrix}
             \hm\wtd E & 0  & I_m & -\wtd Y \\
             -\wtd X    & I_n & 0 & \hm\wtd F
          \end{bmatrix},
        $$
        which is in the desired form,
        where
        \begin{align}
        \wtd X&=X+\bh(\be_j^{\T}\wtd E) %\nonumber\\
              =X-\frac {\bh}{\be_j^{\T}\by}\,\big(\be_j^{\T}E\big), \label{eq:bigY-update4X} \\
        \wtd F&=F-\bh(\be_{\ell}^{\T}-\be_j^{\T}\wtd Y) %\nonumber\\
              =F+\frac {\bh}{\be_j^{\T}\by}\,\big[\be_{\ell}^{\T}-e_j^{\T}Y\big]. \label{eq:bigY-update4F}
        \end{align}
        \end{subequations}
        Along with $Q_2\leftarrow SQ_2$, formulas in \eqref{eq:bigY-update} shows how to update
        the $Q$-standard form when one of the entries of $Y$ is deemed too big.
\end{description}
In the next theorem, we will bound the entries of resulting $\wtd X, \wtd Y, \wtd E$, and $\wtd F$ by
the two actions above.

\begin{theorem}\label{thm:update-i}
After {\rm\bf Action (i)} above, we have
\begin{subequations}\label{eq:bigX-update:bd}
\begin{align}
|\wtd X_{(i,k)}|
   &\le\begin{cases}
          \tau^{-1}, &\quad\mbox{for $i=j,\,k=\ell$}, \\
          1, &\quad\mbox{for $i\ne j,\,k=\ell$, or $i=j,\,k\ne\ell$}, \\
          2|X_{(i,k)}|, &\quad\mbox{for $i\ne j,\,k\ne\ell$},
          \end{cases} \label{eq:bigX-update:bd-1}\\
|\wtd Y_{(i,k)}|&\le |Y_{(i,k)}|+|E_{(i,\ell)}|\,|F_{(j,k)}|/\tau. \label{eq:bigX-update:bd-2}
\end{align}
\end{subequations}
\end{theorem}

\begin{proof}
Recall \eqref{eq:max|X(j,ell)|} in {\rm\bf Action (i)}.
It is noted from \eqref{eq:bigX-update4X} that
%the $\ell$th column of $\wtd X$ is
\begin{align*}
\wtd X_{(:,\ell)}=\wtd X\be_{\ell}
  &=\bx+\frac {\bx+\be_j}{\be_j^{\T}\bx}\,\big[1-\be_j^{\T}\bx\big]
  =\frac {\bx+\be_j}{\be_j^{\T}\bx}-\be_j, \\
\wtd X_{(i,\ell)}=\be_i^{\T}\wtd X\be_{\ell}
  &=\begin{cases}
    1/(\be_j^{\T}\bx), &\quad\mbox{for $i=j$}, \\
    (\be_i^{\T}\bx)/(\be_j^{\T}\bx),&\quad\mbox{for $i\ne j$},
   \end{cases}
\end{align*}
yielding $|\wtd X_{(i,\ell)}|\le\tau^{-1}$ for $i=j$ and $1$ for $i\ne j$.
For the $k$th column for $k\ne\ell$, we have
\begin{align*}
\wtd X_{(:,k)}=\wtd X\be_k
     &=X\be_k+\frac {\bx+\be_j}{\be_j^{\T}\bx}\,\big[-\be_j^{\T}X\be_k\big], \\
\wtd X_{(i,k)}=\be_i^{\T}\wtd X\be_k
     &=\begin{cases}
       -(\be_j^{\T}X\be_k)/(\be_j^{\T}\bx), &\quad\mbox{for $i=j$}, \\
       \big[1-(\be_i^{\T}\bx)/(\be_j^{\T}\bx)\big](\be_i^{\T}X\be_k),&\quad\mbox{for $i\ne j$},
       \end{cases}
\end{align*}
yielding, for $k\ne\ell$, $|\wtd X_{(i,k)}|\le 1$ for $i=j$ and $2|X_{(i,k)}|$ for $i\ne j$.
Now we examine the entries of $\wtd Y$. It follows from \eqref{eq:bigX-update4Y} that
$$
|\wtd Y_{(i,k)}|\le |Y_{(i,k)}|+|E_{(i,\ell)}|\,|F_{(j,k)}|/\tau,
$$
as expected.
\end{proof}

We notice from \eqref{eq:bigX-update:bd-1} that the entire row and column that contain
the largest entry $X_{(j,\ell)}$ in magnitude are significantly reduced, but
the  $(i,k)$th entry $\wtd X$ for $i\ne j$ or $k\ne\ell$ may increase its magnitude
up to twice as much. Although this could create a cause of concern, we argue that this extreme situation
is rare because it only happens
when
1) $X_{(i,\ell)}$ has the opposite sign as $X_{(j,\ell)}$ and
both have almost equal magnitude, and
2) $|X_{(i,k)}|$ is about $\tau$ or bigger.
We notice from \eqref{eq:bigX-update:bd-2} that
for the  $(i,k)$th entry $\wtd Y$, any increase in magnitude will be marginally because
$\tau$ is usually set fairly large, e.g., as in \eqref{eq:maxXY}, and
$|E_{(i,\ell)}|\,|F_{(j,k)}|$ goes to $0$ as the doubling iteration progresses.

Similarly for {\rm\bf Action (ii)}, we have the following estimates, whose proof is omitted because of
similarity.

\begin{theorem}\label{thm:update-ii}
After {\rm\bf Action (ii)} above, we have
\begin{subequations}\label{eq:bigY-update:bd}
\begin{align}
|\wtd X_{(i,k)}|&\le |X_{(i,k)}|+|F_{(i,\ell)}|\,|E_{(j,k)}|/\tau, \label{eq:bigY-update:bd-1}\\
|\wtd Y_{(i,k)}|
   &\le\begin{cases}
          \tau^{-1}, &\quad\mbox{for $i=j,\,k=\ell$}, \\
          1, &\quad\mbox{for $i\ne j,\,k=\ell$, or $i=j,\,k\ne\ell$}, \\
          2|Y_{(i,k)}|, &\quad\mbox{for $i\ne j,\,k\ne\ell$}.
          \end{cases}  \label{eq:bigY-update:bd-2}
\end{align}
\end{subequations}
\end{theorem}

\subsection{Q-Doubling Algorithm}
Finally, we present our complete doubling algorithm after all things considered. We will call it
the {\em Q-Doubling Algorithm\/} (QDA). It starts with $\scrA'-\lambda\scrB'\in\bbC^{(m+n)\times (m+n)}$ that has
$m$ eigenvalues in $\bbD_-$ and $n$ eigenvalues in $\bbD_+$. This condition
may be weakened to include the critical case  in \cref{sec:DA-conv-crit}. For simplicity, we will just mention
the regular case.

\begin{algorithm}
\caption{QDA: Q-Doubling Algorithm}
\label{alg:QDA}
\begin{algorithmic}[1]
    \hrule\vspace{1ex}
    \REQUIRE $\scrA'-\lambda\scrB'\in\bbC^{(m+n)\times (m+n)}$ that has
             $m$ eigenvalues in $\bbD_-$ and $n$ eigenvalues in $\bbD_+$;
    \ENSURE  $Q_1$, $X$, $Q_2$, and $Y$ where
             $Q_1^{\T}\begin{bmatrix}
                     I_m \\
                     X
                   \end{bmatrix}$ and
             $Q_2^{\T}\begin{bmatrix}
                     Y \\
                     I_n
                   \end{bmatrix}$ are basis matrices associated with the $m$ eigenvalues of
             $\scrA'-\lambda\scrB'$ in $\bbD_-$ and $n$ eigenvalues in $\bbD_+$, respectively.
    \hrule\vspace{1ex}
    \STATE Initialization from $\scrA'-\lambda\scrB'$ to $\scrA_0-\lambda\scrB_0$ in \eqref{eq:SFQ:i=0} using
           any one of three ideas in subsection~\ref{ssec:init};
    \FOR{$i=0, 1, \ldots,$ until convergence}
         \STATE partition $Q_1Q_2^{\T}$  as in \eqref{eq:Q1Q2t};
         \STATE $W_i=Q_{22}-X_iQ_{12}-(X_iQ_{11}-Q_{21})Y_i$ \\
                (or, alternatively, $\widetilde{W}_i=Q_{11}^{\T}-Y_{i}Q_{12}^{\T}+(Q_{21}^{\T}-Y_iQ_{22}^{\T})X_i$ if $n>m$);
         \STATE compute $E_{i+1},\, F_{i+1},\, X_{i+1},\, Y_{i+1}$ according to \eqref{eq:EFXY-iter-SFQ} \\ (or, alternatively, \eqref{eq:EFXY-iter-SFQ-new} if $n>m$);
         \STATE if some of the entries of $X_{i+1}$ and/or $Y_{i+1}$ are too large,
                update $Q_1$, $Q_2$, and $E_{i+1},\, F_{i+1},\, X_{i+1},\, Y_{i+1}$
                according to subsection~\ref{ssec:adaptiveQ1Q2};
    \ENDFOR
    \RETURN Last $Q_1$, $X_i$, $Q_2$, and $Y_i$ at convergence.
\end{algorithmic}
\end{algorithm}

A commonly used stopping criterion (for any iterative method) is
\begin{equation}\label{eq:stop}
\|X_i-X_{i-1}\|_{\F}\le\rtol\times \|X_i\|_{\F},
\end{equation}
where {\tt rtol} is a given relative tolerance, but there is a
dangerous pitfall: false convergence in the case
when there is a stretch of painfully slowly moving iterations that consequently
satisfy \eqref{eq:stop}  but $X_i$ is still far from the target. In a situation like this,
we can check the normalized residual for $Q_1^{\T}\begin{bmatrix}
                     I_m \\
                     X_i
                   \end{bmatrix}$ as an approximately eigenbasis matrix of $\scrA-\lambda\scrB$
from which $\scrA'-\lambda\scrB'$ is obtained through transformations.

Alternatively, we can use Kahan's criterion
\begin{equation}\label{eq:stop-kahan}
\frac {\|X_i-X_{i-1}\|_{\F}^2}{\|X_{i-1}-X_{i-2}\|_{\F}-\|X_i-X_{i-1}\|_{\F}}\le\rtol\times \|X_i\|_{\F}.
\end{equation}
This stopping criterion also suffers the pitfall we discussed moments ago, and thus should be used
together with the safeguard.
Kahan's stopping criterion \eqref{eq:stop-kahan} usually terminate an iteration more promptly than
\eqref{eq:stop} for the same {\tt rtol} if the iteration converges at least superlinearly.
For more discussion on Kahan's criterion, the reader is referred to \cite[section~3.9]{hull:2018}.

\section{Application to Eigenspace Computation}\label{sec:EigComp}
One possible application of the $Q$-doubling algorithm (QDA) is to compute the invariant subspaces of
matrix pencil
$$
\scrA-\lambda\scrB\in\bbR^{(m+n)\times (m+n)},
$$
assuming that $\scrA-\lambda\scrB$ has $m$ eigenvalues in $\bbC_-$ and $n$ eigenvalues in $\bbC_+$ (or
$m$ eigenvalues in  $\bbD_-$ and $n$ eigenvalues in  $\bbD_+$).
Denote by $\eig(\scrA,\scrB)$ the spectrum of $\scrA-\lambda\scrB$.
%The permutation matrices $Q_1$ and $Q_2$ in SDASFQ will be dynamically updated during the iteration as described in \cref{sec:QDA}.

The basis matrices for the invariant subspaces of
$\scrA-\lambda\scrB$ associated with the $m$ eigenvalues in $\bbC_-$ and the $n$ eigenvalues in $\bbC_+$
(or $m$ eigenvalues in $\bbD_-$ and $n$ eigenvalues in  $\bbD_+$) will be represented
as
\begin{equation}\label{eq:A-InvSub}
Q_1^{\T}\times\kbordermatrix{ &\sss m \\
      \sss m & I \\
      \sss n & X}, \quad
Q_2^{\T}\times\kbordermatrix{ &\sss n \\
      \sss m & Y \\
      \sss n & I},
\end{equation}
respectively, where unknown permutation matrices $Q_1, Q_2\in\bbC^{(m+n)\times (m+n)}$, and unknown matrices
$X\in\bbC^{n\times m}$, $Y\in\bbC^{m\times n}$ are to be computed by QDA that updates
$Q_1$ and $Q_2$ adaptively.

In the case when $\scrA-\lambda\scrB$ has  $m$ eigenvalues in $\bbC_-$ and $n$ eigenvalues in $\bbC_+$,
we start by selecting a parameter $\gamma<0$. The rate of asymptotical
convergence for the doubling iteration is given by
$$
\varrho(\gamma)=\max\left\{\varrho_-(\lambda):=\max_{\lambda\in\eig(\scrA,\scrB)\cap\bbC_-}\left|\frac {\gamma-\lambda}{\gamma+\lambda}\right|,
    \varrho_+(\lambda):=\max_{\lambda\in\eig(\scrA,\scrB)\cap\bbC_+}\left|\frac {\gamma+\lambda}{\gamma-\lambda}\right|\right\}<1
$$
in the sense that the doubling iteration converges at the same speed as $[\varrho(\gamma)]^{2^i}$ goes to $0$ where $i$ is the iterative index \cite{hull:2018}. Hence,
ideally, the optimal $\gamma$ that will lead to the best rate of asymptotical
convergence is
$$
\arg\min_{\gamma<0}\varrho(\gamma).
$$
In practice, a rough estimate of this optimal $\gamma$ works just as well. In the case when some idea on the distribution
of $\eig(\scrA,\scrB)$ is known, a good $\gamma$ can be estimated \cite{hull:2017}.
%In terms of earlier notation, we let $\scrA=A$ and $\scrB=I_{m+n}$.
Next, we define matrix pencil $\scrA'-\lambda\scrB'\in\bbR^{(m+n)\times (m+n)}$ by
\begin{equation}\label{eq:trans:left2unit}
\begin{bmatrix}
  \scrA' \\
  \scrB'
\end{bmatrix}=\begin{bmatrix}
                 I_{m+n} & -\gamma I_{m+n} \\
                 I_{m+n} & \gamma I_{m+n}
              \end{bmatrix}\begin{bmatrix}
                             \scrA \\
                             \scrB
                           \end{bmatrix}.
\end{equation}
The eigenvalues of matrix pencil $\scrA'-\lambda\scrB'$ are
$
\frac {\lambda-\gamma}{\lambda+\gamma}
$
for $\lambda\in\eig(\scrA,\scrB)$. It can be seen that
$\scrA'-\lambda\scrB'$ has $m$ eigenvalues in $\bar\bbD_{[\varrho_-(\gamma)]-}$
coming from the $m$ eigenvalues  of $\scrA-\lambda\scrB$ in $\bbC_-$, and
$n$ eigenvalues in $\bar\bbD_{[1/\varrho_+(\gamma)]+}$
coming from the $n$ eigenvalues  of $A$ in $\bbC_+$.
Recalling the basis matrices in \eqref{eq:A-InvSub}, we will have
$$
\scrA' Q_1^{\T}\begin{bmatrix}
             I \\
             X
           \end{bmatrix}=\scrB' Q_1^{\T}\begin{bmatrix}
             I \\
             X
           \end{bmatrix}\scrM, \quad
\scrA' Q_2^{\T}\begin{bmatrix}
             Y \\
             I
           \end{bmatrix}\scrN=\scrB' Q_2^{\T}\begin{bmatrix}
            Y \\
            I
           \end{bmatrix},
$$
where $\scrM\in\bbR^{m\times m}$ and $\scrN\in\bbR^{n\times n}$ such that
$\eig(\scrM)\subset\bar\bbD_{[\varrho_-(\gamma)]-}$,
$
\eig(\scrN)\subset\bar\bbD_{[1/\varrho_+(\gamma)]+}.
$

In the case when already $\scrA-\lambda\scrB$ has $m$ eigenvalues in $\bbD_-$ and $n$ eigenvalues in  $\bbD_+$, we simply let $\scrA'-\lambda\scrB'\equiv\scrA-\lambda\scrB$.
Finally we use QDA (\Cref{alg:QDA}) to compute $Q_1$, $Q_2$, $X$, and $Y$ in \eqref{eq:A-InvSub}.

\section{Numerical Experiments}\label{sec:egs}
%In \cite{hull:2018}, inspired by publications including their own works before 2018,
%the authors coherently presented a unifying framework and theory
%for the doubling algorithms. Specifically, there are the two standard forms
%called SF1 and SF2 and correspondingly the two doubling algorithms
%called SDASF1 and SDASF2, including notable applications in optimal control and nano research, among others. In this paper we advance the unification to the next level with the $Q$-standard form SFQ in \eqref{eq:preservation-SFQ} and correspondingly the $Q$-doubling algorithm (QDA)
%outlined in \Cref{alg:QDA}. SFQ naturally bridges the gap between SF1 and SF2, previously
%not known, and the same can be said about QDA to SDASF1 and SDASF2.
%
In this section, we will design and perform two experiments to demonstrate the robustness of QDA in comparison with the existing doubling algorithms, SDASF1 and SDASF2 \cite{hull:2018}.
Initialization of $\scrA_0-\lambda\scrB_0$ is done according to {\bf Idea 3} in
subsection~\ref{ssec:init}.

\subsection{Experiment 1}\label{ssec:BSE}
We experiment QDA (\Cref{alg:QDA}) on three discretized Bethe-Salpeter
equations (BSE) for\footnote {The authors are grateful to Z.-C. Guo \cite{gucl:2019}
    and  M. Shao \cite{shdy:2016} for providing these three examples.}
naphthalene (C${}_{10}$H${}_8$), gallium arsenide (GaAs), and boron nitride
(BN) that were previously used in the experiments in \cite{gucl:2019} by essentially SDASF1 but otherwise reformulated to take advantage of the special structure in BSE:
$$
H=\begin{bmatrix}
                            A & B \\
                            -\bar B & -\bar A
                          \end{bmatrix},
$$
where $A\in\bbC^{n\times n}$ is Hermitian and
$B\in\bbC^{n\times n}$ is symmetric, i.e., $B^{\T}=B$. Eigenvalue problems sharing this matrix structure
also include the linear response eigenvalue problems \cite{bali:2012a,bali:2013,bali:2014,ball:2016,robl:2012}.
It is known that $H$ has $n$ eigenvalues in $\bbC_-$ and $n$ eigenvalues in $\bbC_+$, if there is no
eigenvalues on the imaginary axis, because $H$ is a Hamiltonian matrix \cite[section~2.9]{hull:2018}.
As reported in \cite{gucl:2019}, SDASF1 is able to compute the basis matrices for the eigenspace of $H$
for all three examples. What we will demonstrate below is the superior robustness of QDA to SDASF1 in terms of
much smoother convergence behavior, much smaller magnitude in computed $X$, and much smaller normalized residual. Conveniently, we may define normalized residual as
\begin{equation}\label{eq:NRes1}
\NRes_1:=\frac {\|HZ-ZM\|_{\F}}{\|X\|_{\F}(\|H\|_2+\|M\|_2)},
\end{equation}
where $M=(Z^{\HH}Z)^{-1}Z^{\HH}HZ$, and
\begin{equation}\label{eq:basisM-BSE}
\mbox{either}\,\, Z=\begin{bmatrix}
                      I_n \\
                      X
                    \end{bmatrix}\,\,\mbox{for SDASF1},\,\,
\mbox{or}\,\, Z=Q_1^{\T}\begin{bmatrix}
                      I_n \\
                      X
                    \end{bmatrix}\,\,\mbox{for QDA}.
\end{equation}
For numerical convenience, we estimate $\|H\|_2$ by $\sqrt{\|H\|_1\|H\|_{\infty}}$  and similarly for
$\|M\|_2$. But $\NRes_1$ may not be suitable, especially when $\|X\|_{\F}$ is too big, in which case
$Z$ represents a poorly conditioned basis matrix, and a more proper normalized residual is
\begin{equation}\label{eq:NRes2}
\NRes_2=\frac {\|HU-U(U^{\HH}HU)\|_{\F}}{\sqrt n(\|H\|_2+\|U^{\HH}HU\|_2)}
\end{equation}
where $U$ is from orthnormalizing the columns of $Z$ in \eqref{eq:basisM-BSE}, e.g., by the thin QR decomposition: $Z=UR$
where $U\in\bbC^{2n\times n}$ and $U^{\HH}U=I_n$.

\iffalse
\begin{table}
\renewcommand{\arraystretch}{1.2}
\centering\scriptsize
\caption{Performance statistics for BSE}\label{tbl:BSE}
\begin{tabular}{|c|c|c|c|c|c|c|c|c|c|}
  \hline
\multirow{ 2}{*}{BSE}  & \multirow{ 2}{*}{$n$}  & \multicolumn{2}{c|}{$\|X\|_{\F}$}
     & \multicolumn{2}{c|}{CPU} & \multicolumn{2}{c|}{NRes} & \multicolumn{2}{c|}{\# of itn}\\ \cline{3-10}
  &  & SDA & QDA & SDA & QDA & SDA & QDA & SDA & QDA \\ \hline\hline
C${}_{10}$H${}_8$ & 32  & $1.6\cdot 10^5$ & $5.3\cdot 10^{-1}$
                  & $2.7\cdot 10^{-3}$    &  $2.0\cdot 10^{-2}$
                  & $2\cdot 10^{-13}$ & $8\cdot 10^{-17}$
                  & 7 & 7\\ \hline
GaAs              & 128 & $4.2\cdot 10^4$ & $2.6\cdot 10^{-1}$
                  & $5.3\cdot 10^{-2}$    & $4.0\cdot 10^{-1}$
                  & $2\cdot 10^{-13}$ & $6\cdot 10^{-17}$
                  & 9 & 9 \\ \hline
BN                & 2304 & $1.7\cdot 10^9$ & $3.1\cdot 10^{-1}$
                  & $3.9\cdot 10^1$    & $1.7\cdot 10^3$
                  & $2\cdot 10^{-10}$ & $1\cdot 10^{-16}$
                  & 7 & 7 \\ \hline
\end{tabular}
\end{table}
\fi

\begin{table}[t]
\renewcommand{\arraystretch}{1.3}
\centering\scriptsize
\caption{Performance statistics for BSE}\label{tbl:BSE}
\begin{tabular}{|c|c|c|c|c|}
  \hline
\multicolumn{2}{|c|}{} & C${}_{10}$H${}_8$  & GaAs & BN \\ \hline
\multicolumn{2}{|c|}{$n$} & 32 & 128 & 2304 \\ \hline
\multirow{ 2}{*}{$\|X\|_{\F}$}
    & SDASF1 & $1.6\cdot 10^5$    & $4.2\cdot 10^4$    & $1.7\cdot 10^9$ \\ \cline{2-5}
    & QDA & $5.3\cdot 10^{-1}$ & $2.6\cdot 10^{-1}$ & $3.1\cdot 10^{-1}$ \\ \hline
\multirow{ 2}{*}{CPU}
    & SDASF1 & $9.7\cdot 10^{-3}$ & $2.3\cdot 10^{-2}$ & $4.6\cdot 10^1$ \\ \cline{2-5}
    & QDA & $4.7\cdot 10^{-2}$ & $1.9\cdot 10^{-1}$ & $1.8\cdot 10^3$ \\ \hline
\multirow{ 2}{*}{$\NRes_1$}
    & SDASF1 & $1.9\cdot 10^{-13}$ & $1.8\cdot 10^{-13}$ & $1.7\cdot 10^{-10}$ \\ \cline{2-5}
    & QDA & $7.8\cdot 10^{-17}$ & $6.3\cdot 10^{-17}$ & $1.0\cdot 10^{-16}$ \\ \hline
\multirow{ 2}{*}{$\NRes_2$}
    & SDASF1 & $6.9\cdot 10^{-10}$ & $4.2\cdot 10^{-11}$ & $6.0\cdot 10^{-4}$ \\ \cline{2-5}
    & QDA & $1.3\cdot 10^{-16}$ & $1.7\cdot 10^{-16}$ & $2.9\cdot 10^{-16}$ \\ \hline
\multirow{ 2}{*}{\#it'n}
    & SDASF1 & 7 & 9 & 7 \\ \cline{2-5}
    & QDA & 7 & 9 & 7 \\ \hline
\end{tabular}
\end{table}

\begin{figure}[t]
{\centering
\begin{tabular}{lcc}
& SDASF1 \cite{gucl:2019,hull:2018} & QDA \\
\rotatebox{90}{\hspace*{1.8cm}C${}_{10}$H${}_8$} &
\resizebox*{0.40\textwidth}{0.20\textheight}{\includegraphics{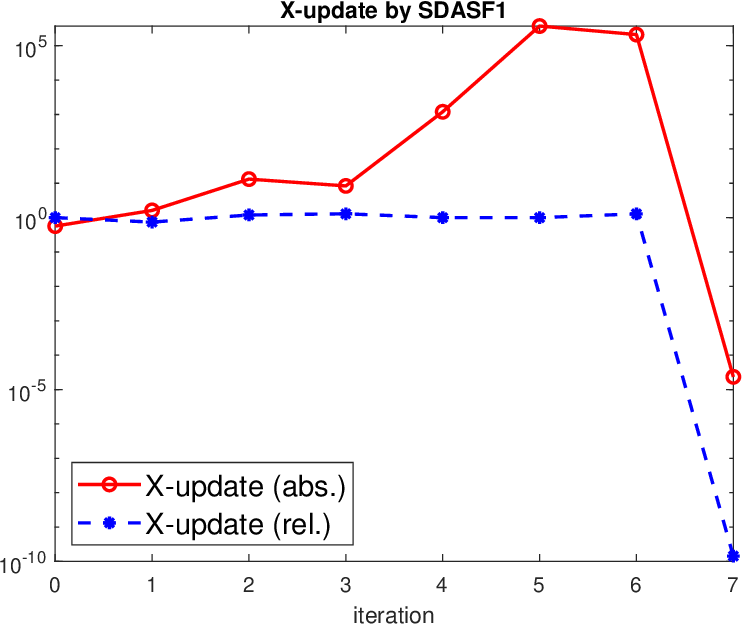}}
  & \resizebox*{0.40\textwidth}{0.20\textheight}{\includegraphics{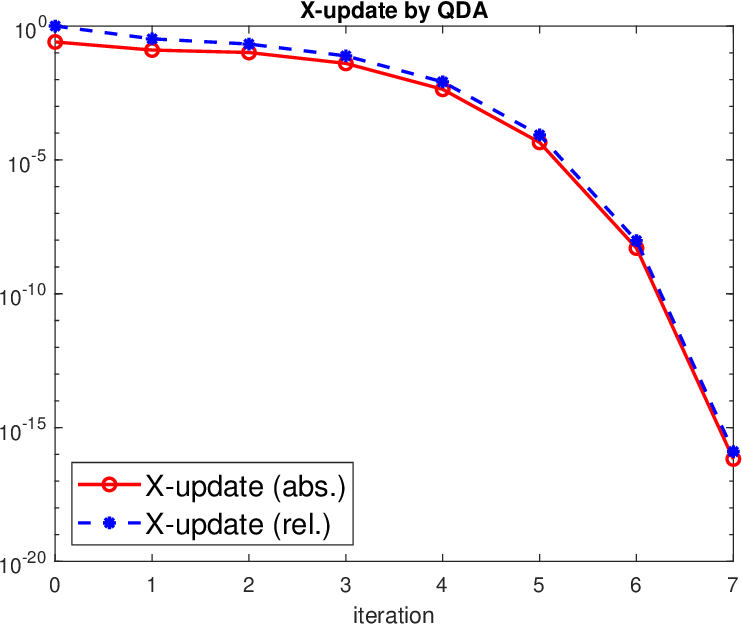}} \\
\rotatebox{90}{\hspace*{1.8cm}GaAs} &
\resizebox*{0.40\textwidth}{0.20\textheight}{\includegraphics{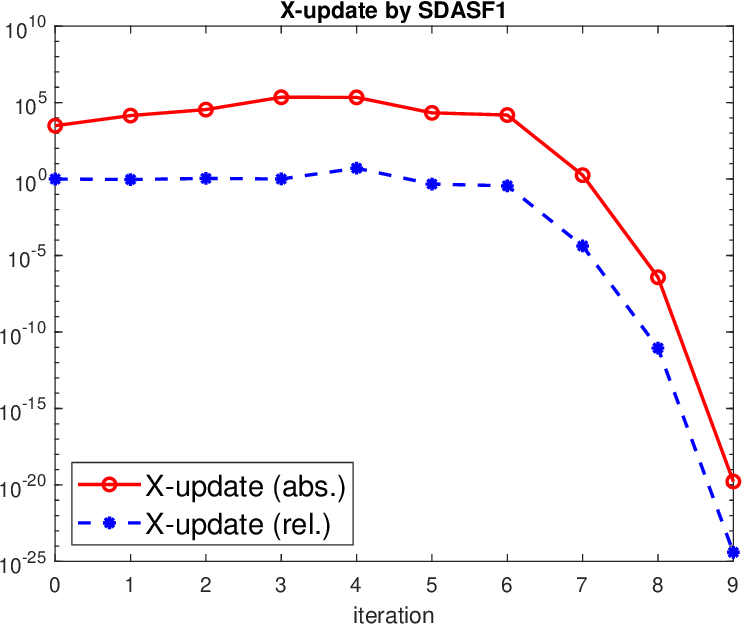}}
  & \resizebox*{0.40\textwidth}{0.20\textheight}{\includegraphics{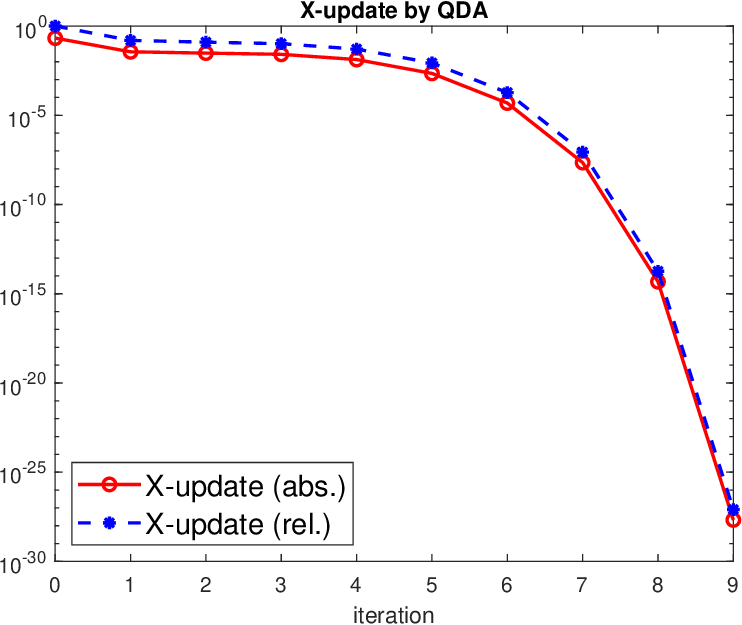}} \\
\rotatebox{90}{\hspace*{1.8cm}BN} &
\resizebox*{0.40\textwidth}{0.20\textheight}{\includegraphics{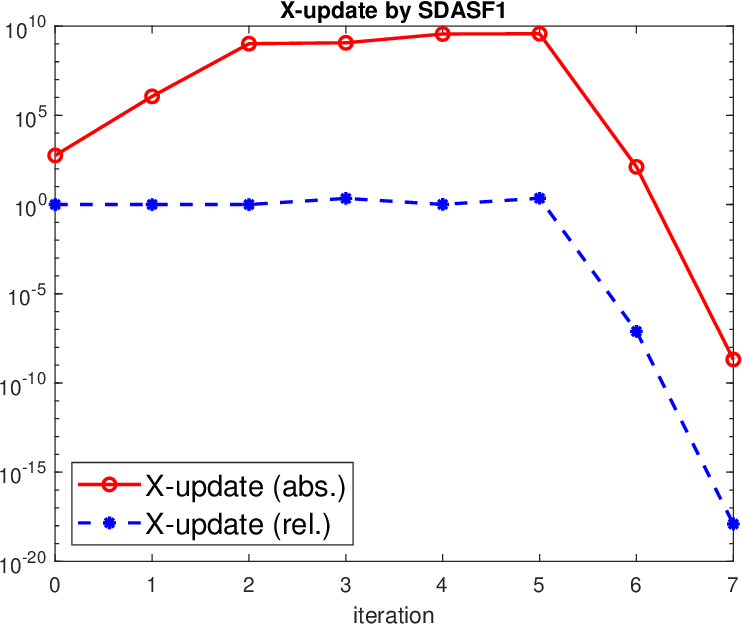}}
  & \resizebox*{0.40\textwidth}{0.20\textheight}{\includegraphics{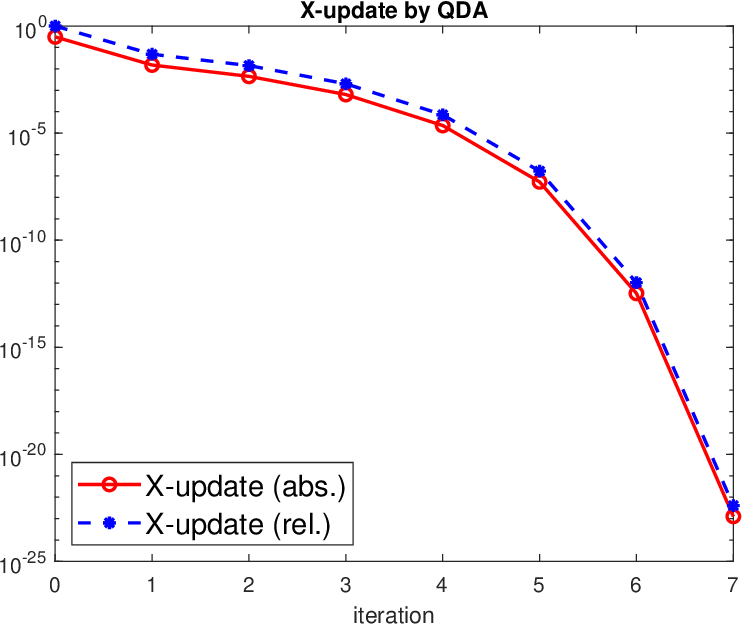}} \\
\end{tabular}\par
}
\vspace{-0.15 cm}
\caption{\small The histories of $X$-updates during the doubling iteration by SDASF1 and QDA on  BSE.
  }
\label{fig:BSE}
\end{figure}

Before applying SDASF1 and QDA, we will need to pick $\gamma<0$ and perform
transformation \eqref{eq:trans:left2unit} from $\scrA-\lambda\scrB:=H -\lambda I_{2n}$ to $\scrA'-\lambda\scrB'$.
For simplicity, we will simply use $\gamma=-1$ for illustration purpose.
\Cref{tbl:BSE} displays performance statistics on the three BSE
by SDASF1  \cite{gucl:2019,hull:2018}
and by QDA developed in this paper.
We summarize our observation as follows:
1) Both take the same numbers of doubling iterations;
2) %QDA is much involved than SDASF1 and, when successful,
SDASF1 runs much faster than QDA, as expected due to the fact that QDA is much involved and its current implementation of Gaussian elimination is far from being optimized;
3) QDA is much more robust, as evidenced by much smaller normalized residuals
   and much smaller norms $\|X\|_{\F}$ at convergence.
Particularly, we would like to point out that, for BN, $\NRes_1=1.7\cdot 10^{-10}$
while $\NRes_2=6.0\cdot 10^{-4}$, indicating that $Z$ computed by SDASF1 poorly represents
the eigenspace of interest. In all examples, both $\NRes_1$ and $\NRes_2$ by QDA are about $10^{-16}$
indicating that $Z$ computed by QDA accurately represents
the eigenspace.

\Cref{fig:BSE} plots the histories of $X$-updates,
where the $i$th absolute and relative $X$-updates are defined, respectively, as
$$
\|\Delta_i\|_{\F}, \quad
\frac {\|\Delta_i\|_{\F}}{\|X_{i+1}\|_{\F}}
$$
for $X_{i+1}=X_i+\Delta_i$ during the doubling iterations. Again the superior robustness of QDA is on the display. We observed:
1) there are  quite large disparities between the absolute and relative $X$-updates
for SDASF1 because large $\|X\|_{\F}$ in the end necessarily entails large updates for fast convergence;
2) the convergence of SDASF1 is very rough and in fact, the relative $X$-update does not show any convergence
until the last 2 to 3 iterations with the extreme going to C${}_{10}$H${}_8$ for which
the convergence of SDASF1 occurs at the last doubling iteration!
3) for all three BSE, QDA exhibits superior smooth convergence behaviors.

\subsection{Experiment 2}
We first construct some random matrices with $m$ eigenvalues in $\bbC_-$ and  $n$ eigenvalues in $\bbC_+$
such that any basis matrix $Z$ as in \eqref{eq:BaseM} associated with the $m$ eigenvalues in $\bbC_-$
with a nearly singular $Z_1$. Assuming $m\le n$, in MATLAB-like pseudo-code, we do
%\begin{equation}\label{eq:SCF-form:NEPv:intro}
%\framebox{
%\parbox{10.0cm}{
%\begin{align*}
%Q&={\tt randn}(N)+\iota{\tt randn}(N); \quad Q_{(1:m,1:m)}=\eta Q_{(1:m,1:m)}; \\
%T&={\tt triu}({\tt randn}(N)+\iota{\tt randn}(N),1) \\
% &\quad+{\tt diag}([(2{\tt rand}(m,1)-\alpha),(2{\tt rand}(n,1)+\alpha)])\\
% &\quad+\iota\diag({\tt randn}(N,1));
%\end{align*}
%}}
%\end{equation}
\begin{subequations}\label{eq:QTAB}
\begin{align}
U&={\tt randn}(N)+{\tt 1i*}{\tt randn}(N),  \label{eq:Q} \\
T&={\tt triu}({\tt randn}(N)+{\tt 1i*}{\tt randn}(N),1) \nonumber\\
 &\quad+{\tt diag}([(2{\tt rand}(m,1)-\alpha),(2{\tt rand}(n,1)+\alpha)])
    +{\tt 1i*}\diag({\tt randn}(N,1)),  \label{eq:T}
\end{align}
where $N=m+n$, $\alpha>2$ is a parameter,
and then let, with varying $\eta$,
\begin{equation}\label{eq:AB}
U_{(1:m,1:m)}=\eta U_{(1:m,1:m)}, \quad \scrA=UTU^{-1}, \quad \scrB=I_{m+n},
\end{equation}
\end{subequations}
where $\eta$ is another adjustable parameter. By design and in theory, $\scrA-\lambda\scrB$ has $m$ eigenvalues in $\bbC_-$ and the $n$ eigenvalues in $\bbC_+$, which is critical for doubling algorithms to converge, but it has been observed that rounding errors in forming
$\scrA-\lambda\scrB$ in \eqref{eq:AB} can destroy this theoretic fact unless parameter $\alpha$
in \eqref{eq:T} is made sufficiently large.

\begin{table}
\renewcommand{\arraystretch}{1.3}
\centering\scriptsize
\caption{Performance statistics for random \eqref{eq:QTAB} }\label{tbl:rand}
\begin{tabular}{|c|c|c|c|c|c|}
  \hline
\multicolumn{2}{|c|}{$\eta$} & $10^{-4}$ & $10^{-5}$ & $10^{-6}$ & $10^{-7}$  \\ \hline
\multirow{ 2}{*}{$\|X\|_{\F}$}
    & SDASF1 & $2.2\cdot 10^6$ & $2.9\cdot 10^7$ & -- & -- \\ \cline{2-6}
    & QDA & $7.8\cdot 10^1$ & $3.2\cdot 10^1$ & $3.2\cdot 10^1$ & $3.3\cdot 10^1$ \\ \hline
\multirow{ 2}{*}{CPU}
    & SDASF1 & $1.1\cdot 10^{-1}$ & $1.0\cdot 10^{-1}$ & -- & -- \\ \cline{2-6}
    & QDA & $1.2\cdot 10^{0}$ & $2.3\cdot 10^{0}$ & $2.3\cdot 10^0$ & $2.3\cdot 10^0$ \\ \hline
\multirow{ 2}{*}{$\NRes_1$}
    & SDASF1 & $5.3\cdot 10^{-9}$ & $1.5\cdot 10^{-8}$ & -- & -- \\ \cline{2-6}
    & QDA & $5.2\cdot 10^{-11}$ & $8.0\cdot 10^{-11}$ & $2.4\cdot 10^{-10}$ & $1.0\cdot 10^{-9}$ \\ \hline
\multirow{ 2}{*}{$\NRes_2$}
    & SDASF1 & $2.8\cdot 10^{-8}$ & $1.0\cdot 10^{-6}$ & -- & -- \\ \cline{2-6}
    & QDA & $5.6\cdot 10^{-11}$ & $8.0\cdot 10^{-11}$ & $2.5\cdot 10^{-10}$ & $8.9\cdot 10^{-10}$ \\ \hline
\multirow{ 2}{*}{\#it'n}
    & SDASF1 & 9 & 9 & -- & -- \\ \cline{2-6}
    & QDA & 9 & 8 & 8 & 8 \\ \hline
\end{tabular}
\end{table}

Next pick $\gamma<0$ and perform
transformation \eqref{eq:trans:left2unit} from $\scrA-\lambda\scrB$ to $\scrA'-\lambda\scrB'$.
It can be verified that $\scrA' Z=\scrB' Z M'$ where $M'=(M-\gamma I_m)^{-1}(M+\gamma I_m)$ whose eigenvalues lie in $\bbD_-$.
From this point on, we could face two choices to compute the eigenspace $\cR(Z)$ where $Z=U_{(:,1:m)}$:
1) the existing doubling algorithm SDASF1 in \cite[chapter~3]{hull:2018} which transforms
$\scrA'-\lambda\scrB'$ into the standard form SF1 and then calls upon SDASF1;
2) the new QDA.
As $\eta$ goes to $0$,
$Z_1:=Z_{(1:m,:)}$ goes to $0$ and hence $Z_2Z_1^{-1}$ goes to infinity where $Z_2=Z_{(m+1:m+n,:)}$.
It is expected that the existing doubling algorithm SDASF1 will encounter increasing difficulty in computing $Z_2Z_1^{-1}$, as $\eta$ becomes very small.

%We comment that $T$ in \eqref{eq:T} has $m$ eigenvalues in $\bbC_-$ and  $n$ eigenvalues in $\bbC_+$ so long as $\alpha>2$ by construction and so does
%$\scrA-\lambda\scrB$ in \eqref{eq:AB} in  theory because
%$\scrA$ is supposedly similar to $T$. But in actual execution, the rounding errors
%in computing $UTU^{-1}$ can perturb the eigenvalues of $T$ badly
%to make some of the eigenvalues of $T$ cross over the imaginary axis if $\alpha$ is
%not large enough.

While we have experimented with numerous random $\scrA-\lambda\scrB$ as generated according to
\eqref{eq:QTAB}, in what follows, we will report a typical example, in which we first generate
and save $U$ and $T$ by \eqref{eq:Q} and \eqref{eq:T} with $m=200$, $n=250$, and $\alpha=8$, and then vary $\eta$ to
produce $\scrA-\lambda\scrB$ by \eqref{eq:AB}.
\Cref{tbl:rand} displays the performance statistics as $\eta$ varies from $10^{-4}$ to $10^{-7}$,
where $\NRes_1$ and $\NRes_2$ are defined in the same way as in
\eqref{eq:NRes1} and \eqref{eq:NRes2} (with $H$ there replaced by $\scrA$ here, noting that $\scrB=I$).
It is observed that SDASF1  runs  faster than QDA when it works but
produces less accurate results. In fact, SDASF1 explodes, i.e., generating NaNs,
for $\eta=10^{-7}$, and produces erroneous results for $\eta=10^{-6}$.

\section{Conclusions}\label{sec:conclu}
The existing doubling algorithms \cite{hull:2018} for solving various nonlinear matrix equations
arising from real-world engineering applications are in fact to compute an eigenspace of
some matrix pencil in terms of a basis matrix of the eigenspace. It is known that the
basis matrix of an $m$-dimensional subspace in $\bbC^{m+n}$ is not unique.
Let
$
Z=\begin{bmatrix}
    Z_1 \\
    Z_2
  \end{bmatrix}
$ be a basis matrix of the subspace,
where $Z_1$ is an $m\times m$ matrix. In the case when $Z_1$ is nonsingular,
$ZZ_1^{-1}=\begin{bmatrix}
    I\\
    Z_2Z_1^{-1}
  \end{bmatrix}=:\begin{bmatrix}
                   I \\
                   X
                 \end{bmatrix}
$
is also a basis matrix. It is noted that the nonsingularity of $Z_1$ depends on the subspace, not a choice of some basis matrix which is not unique. Similarly if the submatrix of $Z$, consisting of its last
$m$ rows,
is nonsingular, then it has another basis matrix representation in the form of $\begin{bmatrix}
                                                     Y \\
                                                     I
                                                   \end{bmatrix}$.

Dependent on the two particular forms for the basis matrices, the existing doubling algorithms
are classified into two types \cite{hull:2018}. In this paper, a unifying doubling algorithm is proposed
to bridge the gap between the two existing  types of doubling algorithms by working with
the versatile basis matrix representation
$
Q^{\T}\begin{bmatrix}
                   I \\
                   X
                 \end{bmatrix}
$ where $Q$ is some permutation matrix.
It is noted that our presentation works for orthogonal $Q$
without any changes and can be made to work for unitary $Q$ too upon replacing matrix transpose by
complex conjugate transpose. In doing that, we also propose a new standard form, called the $Q$-standard form (SFQ), which unifies the existing first and second standard form SF1 and SF2 in \cite{hull:2018}
associated with the two existing doubling algorithms as well. We name our new algorithm, the $Q$-doubling algorithm (QDA) which includes dynamically updating the $Q$-matrix in the
basis matrix representation. Convergence analysis of QDA is also performed.
We also present two sets of experiments, one for the discretized Bethe-Salpeter
equations (BSE) and the other for random matrices. These experiments demonstrate that QDA is superior in
robustness and accuracy to the existing doubling algorithms.
Unfortunately, there is a tradeoff, i.e., QDA runs  slower than the existing doubling algorithms,
which is expected.
%which will be a topic for further research.

\appendix

\section{Convergence Analysis: Critical Case}\label{sec:DA-conv-crit}
%\marginpar{\tiny need updating, updated}
Previously in section~\ref{sec:DA-conv-regu}, we performed a convergence analysis of
SDASFQ (Algorithm~\ref{alg:DA-SFQ})
for the regular case, i.e.,
$\rho(\scrM)\cdot\rho(\scrN)<1$. In this appendix, we shall
perform a convergence analysis for the critical case, i.e.,
\begin{subequations}\label{eq:DA-cvg:crit-dfn}
\begin{equation}\label{eq:DA-cvg:crit-dfn-1}
\rho(\scrM)\cdot\rho(\scrN)=1.
\end{equation}
We largely employ the techniques used in \cite[section~3.8]{hull:2018}.

Equivalently, we may express \eqref{eq:DA-cvg:crit-dfn-1} as
%Throughout the rest of this section, we assume \eqref{eq:DA-cvg:crit-dfn-1} holds and let
\begin{equation}\label{eq:DA-cvg:crit-dfn-2}
0<\varrho:=\rho(\scrM)=1/\rho(\scrN).
\end{equation}
\end{subequations}
This will be assumed throughout this section.
Introduce notation $J_p(\omega)$ for the $p\times p$ Jordan block
corresponding to a scalar $\omega$:
\begin{equation}\label{eq:mtx_J_omega_m}
J_{p}(\omega)= \begin{bmatrix}
             \omega & 1 & 0 & \cdots & 0 \\
             0 & \omega & 1 & \ddots & \vdots \\
             \vdots & \ddots & \ddots & \ddots & 0 \\
             \vdots & & \ddots & \ddots & 1 \\
             0 & \cdots & \cdots & 0 & \omega
               \end{bmatrix}\in\bbC^{p \times p}.
\end{equation}
It is upper triangular. It is well-known that any power of it remains upper triangular. For our analysis later,
we will need the $2^i$th power of $J_p(\omega)$ for positive integer $i$.

\begin{lemma}[{\cite[pp. 557]{govl:1996}}] \label{thm:JF_power}
Let $J_{p}(\omega)$ be given by \eqref{eq:mtx_J_omega_m} and write
\begin{equation}\label{eq:mtx_J_omega_m_2k}
\big[J_{p}(\omega)\big]^{2^i} =\begin{bmatrix}
        \gamma_{1,i} & \gamma_{2,i} & \cdots & \gamma_{p,i} \\
        0 & \gamma_{1,i} & \ddots & \vdots \\
        \vdots & \ddots & \ddots & \gamma_{2,i} \\
        0 & \cdots & 0 & \gamma_{1,i}
\end{bmatrix}.
\end{equation}
Then $\gamma_{1,i}=\omega^{2^i}$ and
$$
\gamma_{j,i} = \frac{2^i ( 2^i - 1) \cdots (2^i-j+2)}{(j-1)!} \omega^{2^i-j+1}\quad
\mbox{for $j = 2, \ldots, p$}.
$$
\end{lemma}

In view of Lemma~\ref{thm:JF_power}, we set (for $p=2k$)
\begin{equation}\label{eq:mtx_Gamma_k}
\Gamma_{i,k}:=\big\{\big[J_{2k}(\omega)\big]^{2^i}\big\}_{(1:k,k+1:2k)}
   = \left[ \begin{array}{ccccc}
            \gamma_{k+1,i} & \gamma_{k+2,i} & \cdots & \gamma_{2k-1,i} & \gamma_{2k,i} \\
            \gamma_{k,i} & \ddots & \ddots &   & \gamma_{2k-1,i} \\
            \vdots & \ddots & \ddots & \ddots & \vdots \\
            \gamma_{3,i} &   & \ddots  & \ddots & \gamma_{k+2,i} \\
            \gamma_{2,i} & \gamma_{3,i} & \cdots  & \gamma_{k,i} &  \gamma_{k+1,i}
        \end{array} \right].
\end{equation}
In particular, $\Gamma_{0,k}=\be_k\be_1^{\T}$.
Here and in Lemma~\ref{thm:JF_power} and Lemma~\ref{lem:nonsingular_Gamma} below, we suppress the dependence of $\Gamma_{i,k}$ and $\gamma_{i,j}$ on $\omega$
for notation clarity. The next lemma is the backbone of our convergence analysis.

\begin{lemma}[\cite{huli:2009}] \label{lem:nonsingular_Gamma}
    The Toeplitz matrix $\Gamma_{i,k}$ in \eqref{eq:mtx_Gamma_k} is
    nonsingular  for sufficiently large $i$ and satisfies
    $$
        \big\| [J_{k}(\omega)]^{2^i} \Gamma_{i,k}^{-1} \big\| = \big\| \Gamma_{i,k}^{-1} [J_{k}(\omega)]^{2^i}\big\| = O(2^{-i}),
        \, \big\| [J_{k}(\omega)]^{2^i}\Gamma_{i,k}^{-1} [J_{k}(\omega)]^{2^i}\big\| = O(2^{-i}).
    $$
\end{lemma}

Before we step into convergence analysis for SDASFQ, we set up common assumptions on
$\scrA_0-\lambda\scrB_0$.
%We will state all of them in terms of $m$ and $n$ as in SF1 with an understanding that
%they apply to SF2, too, after letting $m=n$.
Recall \eqref{eq:DA-cvg:crit-dfn}. We group the eigenvalues of $\scrA_0-\lambda\scrB_0$ into three subsets according to
being in $\bbD_{\varrho-}$, on $\bbT_{\varrho}$, or outside $\bar\bbD_{\varrho-}$
(i.e., in $\bbD_{\varrho+}$). Our main assumptions on the Weierstrass canonical form of
$\scrA_0-\lambda\scrB_0$ are stated in \eqref{eq:DA-cvg:crit:WCJ-SFQ} below, where
partial multiplicities of an eigenvalue are the orders of all the Jordan blocks associated with the eigenvalue.
\begin{equation}\label{eq:DA-cvg:crit:WCJ-SFQ}
\framebox{
\parbox{11cm}{
(1) $\eig(\scrA_0,\scrB_0)\cap\bbT_{\varrho}\ne\emptyset$; (2) the partial
multiplicities for each eigenvalue on $\bbT_{\varrho}$ are even; (3)
the number $2n_0$ of eigenvalues, counting algebraic multiplicities, on $\bbT_{\varrho}$
is no larger than $2\times\min\{m,n\}$, i.e., $n_0\le\min\{m,n\}$; and (4) the numbers of eigenvalues
in $\bbD_{\varrho-}$, on $\bbT_{\varrho}$, and in $\bbD_{\varrho+}$ are $m-n_0$, $2n_0$, and $n-n_0$,
respectively.
}
}
\end{equation}
%Previous analysis in \cite{chcg:2009} requires $\varrho=1$ (and it has
%symmetry and definiteness assumptions, too). Thus it cannot be applied
%to \ADDA\ \cite{wawl:2012} for \MARE\  in the critical case. The analysis in
%Guo and Lu~\cite{gulu:2016} is, however, specifically performed for \ADDA\ in the critical case.

The first assumption in \eqref{eq:DA-cvg:crit:WCJ-SFQ} is a consequence of \eqref{eq:DA-cvg:crit-dfn-2}, but
the second assumption  is rather involved and impossible to verify numerically
because the Weierstrass canonical form is not continuous in the sense that it can change dramatically under
arbitrarily tiny perturbations to the entries of $\scrA_0$ and $\scrB_0$. So it is only of
theoretical value. It is needed to allow us to divide the $2n_0$ eigenvalues into
two equal halves: one half to be combined with those in $\bbD_{\varrho-}$ and the other half
with those in $\bbD_{\varrho+}$, in the way to be detailed in the proofs later, so that we can have the statement that
$\scrA_0-\lambda\scrB_0$ has $m$ eigenvalues in $\bar\bbD_{\varrho-}$ and the other $n$ eigenvalues
in $\bar\bbD_{\varrho+}$.  For later use, we introduce

\begin{definition}\label{dfn:A0B0-P}
The matrix pencil $\scrA_0-\lambda\scrB_0$ is said to
have the {\em Property\/} $\PrP{\bbT_{\varrho}}$
if it has at least one eigenvalue on $\bbT_{\varrho}$ and
the partial multiplicities for every eigenvalue on $\bbT_{\varrho}$ are even. By default, we write
{\em Property $\PrP{\bbT}$} when $\varrho=1$.
\end{definition}

In view of these assumptions in \eqref{eq:DA-cvg:crit:WCJ-SFQ}, we let $J_{2m_j}(\omega_j)$ for $j=1,2,\ldots,r$ be all of the Jordan blocks associated
with eigenvalues on $\bbT_{\varrho}$ and write
\begin{equation}\label{eq:Jordan-2mj}
    J_{2m_j}(\omega_j) = \left[ \begin{array}{cc}
         J_{m_j}(\omega_j) & \Gamma_{0, m_j} \\ 0_{m_j} &
         J_{m_j}(\omega_j)
    \end{array} \right],
\end{equation}
where $\Gamma_{0,m_j} = \be_{m_j} \be_1^{\T}$. Then $n_0 = \sum_{j=1}^r m_j$.
Let
\begin{equation}\label{eq:m'-n':SFQ}
m' = m - n_0, \quad
n' = n - n_0.
\end{equation}
Then $m'\ge 0$, $n'\ge 0$ by \eqref{eq:DA-cvg:crit:WCJ-SFQ}, and $\scrA_0-\lambda\scrB_0$ has $m'$ eigenvalues
in $\bbD_{\varrho-}$ and $n'$ eigenvalues in $\bbD_{\varrho+}$.
Note $\omega_j\in\bbT_{\varrho}$ for $1\le j\le r$ are not necessarily distinct. With all these preparations, we
find that the Weierstrass canonical form \cite[p.28]{gant:1959}  of $\scrA_0-\lambda\scrB_0$ can be
written as
\begin{subequations} \label{eq:crit-WCJ-SFQ}
\begin{equation}\label{eq:crit-WCJ-SFQa}
J_{\scrA_0}-\lambda J_{\scrB_0}
  =\kbordermatrix{ & \sss m' & \sss n_0 & \omit  & \sss n' & \sss n_0  \cr
                \sss m' & \scrM_{\stb} & 0 & \omit\vrule & 0 & 0 \cr
                \sss n_0 & 0 & J_1 &\omit\vrule & 0 & \Gamma_0 \cr \cline{2-6}
                \sss n' & 0 & 0 & \omit\vrule & I & 0 \cr
                \sss n_0 & 0 & 0 & \omit\vrule & 0 & J_1}
  -\lambda\kbordermatrix{ & \sss m' & \sss n_0 & \omit  & \sss n' & \sss n_0  \cr
                \sss m' & I & 0 & \omit\vrule & 0 & 0 \cr
                \sss n_0 & 0 & I &\omit\vrule & 0 & 0 \cr \cline{2-6}
                \sss n' & 0 & 0 & \omit\vrule & \scrN_{\stb} & 0 \cr
                \sss n_0 & 0 & 0 & \omit\vrule & 0 & I},
\end{equation}
where $\scrM_{\stb}$ and $\scrN_{\stb}$ satisfy $\rho(\scrM_{\stb} ) < \varrho$ and $\rho(\scrN_{\stb}) < \varrho^{-1}$, and
\begin{equation}\label{eq:crit-WCJ-SFQc}
J_ 1 = \oplus_{j=1}^r J_{m_j}(\omega_j), \quad
\Gamma_0 = \oplus_{j=1}^r \Gamma_{0, m_j},
\end{equation}
and
there are $(m+n)\times (m+n)$ nonsingular matrices $\scrP$ and $\scrU$ such that
\begin{equation}\label{eq:crit-WCJ-SFQb}
\scrP \scrA_0 \scrU = J_{\scrA_0}, \quad
\scrP \scrB_0 \scrU = J_{\scrB_0}.
\end{equation}
\end{subequations}
On the other hand, if we interchange the roles of $\scrA_0$ and $\scrB_0$,
we find that the Weierstrass canonical form of $\scrB_0-\lambda\scrA_0$ can be
written as
\begin{subequations} \label{eq:crit-WCJ-SFQ'}
\begin{equation}\label{eq:crit-WCJ-SFQa'}
\what J_{\scrB_0}-\lambda \what J_{\scrA_0}
  =\kbordermatrix{ & \sss m' & \sss n_0 & \omit  & \sss n' & \sss n_0  \cr
                \sss m' & I & 0 & \omit\vrule & 0 & 0 \cr
                \sss n_0 & 0 & \what J_1^{\T} &\omit\vrule & 0 & 0 \cr \cline{2-6}
                \sss n' & 0 & 0 & \omit\vrule & \scrN_{\stb} & 0 \cr
                \sss n_0 & 0 & \Gamma_0^{\T} & \omit\vrule & 0 & \what J_1^{\T}}
  -\lambda\kbordermatrix{ & \sss m' & \sss n_0 & \omit  & \sss n' & \sss n_0  \cr
                \sss m' & \scrM_{\stb} & 0 & \omit\vrule & 0 & 0 \cr
                \sss n_0 & 0 & I &\omit\vrule & 0 & 0 \cr \cline{2-6}
                \sss n' & 0 & 0 & \omit\vrule & I & 0 \cr
                \sss n_0 & 0 & 0 & \omit\vrule & 0 & I},
\end{equation}
where $\scrM_{\stb}$, $\scrN_{\stb}$, and $\Gamma_0$ are the same as in \eqref{eq:crit-WCJ-SFQ}, and
\begin{equation}\label{eq:crit-WCJ-SFQc'}
\what J_1= \oplus_{j=1}^r J_{m_j}(1/\omega_j),
\end{equation}
and
there are nonsingular matrices $\scrQ$ and $\scrV$ such that
\begin{equation}\label{eq:crit-WCJ-SFQb'}
\mathscr{Q} \scrA_0 \scrV = \what{J}_{\scrA_0}, \quad
\mathscr{Q} \scrB_0 \scrV = \what{J}_{\scrB_0}.
\end{equation}
\end{subequations}

\begin{lemma}\label{lm:DA-cvg-crit-SF1:rels}
Let $\{ \scrA_i-\lambda \scrB_i\}_{i=1}^{\infty}$ be the sequence generated by {\rm SDASFQ (Algorithm~\ref{alg:DA-SFQ})}
with the matrices $\scrA_i$ and $\scrB_i$ as in \eqref{eq:preservation-SFQ}. Then for $i\ge 0$
\begin{equation}\label{eq:DA-cvg-crit-SF1:rels}
\scrA_i \scrU J_{\scrB_0}^{2^i} = \scrB_i \scrU J_{\scrA_0}^{2^i}, \quad
\scrA_i \scrV\what J_{\scrB_0}^{2^i} = \scrB_i \scrV\what J_{\scrA_0}^{2^i}.
\end{equation}
\end{lemma}

\begin{proof}
Since $J_{\scrA_0}$ and $J_{\scrB_0}$ in \eqref{eq:crit-WCJ-SFQa} commute, i.e.,
$J_{\scrA_0}J_{\scrB_0}=J_{\scrB_0}J_{\scrA_0}$, we have by \eqref{eq:crit-WCJ-SFQb}
$$
\scrP\scrA_0 \scrU J_{\scrB_0}=J_{\scrA_0}J_{\scrB_0}
  =J_{\scrB_0}J_{\scrA_0}=\scrP\scrB_0 \scrU J_{\scrA_0}
$$
which, after pre-multiplying both sides by $\scrP^{-1}$,
gives the first equation in \eqref{eq:DA-cvg-crit-SF1:rels} for $i=0$.
Inductively using Theorem~\ref{thm:DBTrans}, we have it for all $i$.
The second equation in \eqref{eq:DA-cvg-crit-SF1:rels} can be proved in the same way.
\end{proof}

Partition $\scrU$ and $\scrV$ in \eqref{eq:crit-WCJ-SFQ} and \eqref{eq:crit-WCJ-SFQ'} as
\begin{equation}\label{eq:scrUscrV-crit}
      \scrU = \kbordermatrix{ & \sss m & \sss n \\
          \sss m & U_{11} & U_{12} \\
          \sss n & U_{21} & U_{22} }, \quad
      \scrV = \kbordermatrix{ & \sss m & \sss n \\
          \sss m &  V_{11} & V_{12} \\
          \sss n &  V_{21} & V_{22} }.
\end{equation}

%\subsubsection{Case of SF1}
%Consider $\scrA_0-\lambda\scrB_0$ that is given by SFQ.
%The main results are stated in Theorem~\ref{thm:DA-cvg-SF1:crit} below
%whose proof is very complicated and can be skipped on the first reading.

\begin{lemma}\label{lm:DA-cvg-crit-SF1:sols}
Let
\begin{equation}\label{eq:Z-tildeZ} \kbordermatrix{ & \sss m \\
		\sss m & Z_{11} \\
		\sss n & Z_{21} }
:=Q_1\begin{bmatrix}
	U_{11}\\
	U_{21}
\end{bmatrix},\ \kbordermatrix{ & \sss n \\
\sss m & \widetilde Z_{12} \\
\sss n & \widetilde Z_{22} }=:Q_2\begin{bmatrix}
	V_{12}\\
	V_{22}
\end{bmatrix}.	
\end{equation}
If $Z_{11}$ and $Q_{11}^{\T}-Y_0Q_{12}^{\T}+(Q_{21}^{\T}-Y_0Q_{22}^{\T})Z_{21}Z_{11}^{-1}$ are nonsingular, then $\Phi:=Z_{21}Z_{11}^{-1}$ solves \eqref{eq:NME-dual:SFQp}. If
$\widetilde Z_{22}$ and $Q_{22}-X_0Q_{12}+(Q_{21}-X_0Q_{11})\widetilde Z_{12}\widetilde{Z}_{22}^{-1}$ are nonsingular, then $\Psi:=\widetilde Z_{12}\widetilde{Z}_{22}^{-1}$ solves \eqref{eq:NME-dual:SFQd}. Moreover
\begin{subequations}\label{eq:crit-wstab-SF1:sols}
\begin{equation}\label{eq:crit-wstab-SF1:sols-a}
\eig([Q_{11}^{\T}-Y_0Q_{12}^{\T}+(Q_{21}^{\T}-Y_0Q_{22}^{\T})\Phi]^{-1}E_0)\subset\bar\bbD_{\varrho-}
\end{equation}
and \begin{equation}\label{eq:crit-wstab-SF1:sols-b}
	\eig([Q_{22}-X_0Q_{12}+(Q_{21}-X_0Q_{11})\Psi]^{-1}F_0)\subset\bar\bbD_{\varrho^{-1}-}.
\end{equation}
\end{subequations}
\end{lemma}

\begin{proof}
It can be verified, by  \eqref{eq:DA-cvg-crit-SF1:rels} with $i=0$, that
\begin{equation}\label{DA-cvg-crit-SF1:sols-pf}
\scrA_0\begin{bmatrix}
         U_{11} \\
         U_{21}
       \end{bmatrix}=\scrB_0\begin{bmatrix}
                               U_{11} \\
                               U_{21}
                            \end{bmatrix}
                    \begin{bmatrix}
                      \scrM_{\stb} & \\
                       & J_1
                    \end{bmatrix}, \quad
\scrA_0\begin{bmatrix}
         V_{12} \\
         V_{22}
       \end{bmatrix}
                    \begin{bmatrix}
                      \scrN_{\stb} & \\
                       & \hat J_1^{\T}
                    \end{bmatrix}=\scrB_0\begin{bmatrix}
                                V_{12} \\
                                V_{22}
                            \end{bmatrix}.
\end{equation}
Plugging \eqref{eq:SFQ:i=0} and \eqref{eq:Z-tildeZ} into \eqref{DA-cvg-crit-SF1:sols-pf} and then
eliminating $\scrM_{\stb}\oplus J_1$ and $\scrN_{\stb}\oplus \hat J_1^{\T}$, we
find $\Phi=Z_{21}Z_{11}^{-1}$ and $\Psi=\widetilde Z_{12}\widetilde{Z}_{22}^{-1}$ solve \eqref{eq:NME-dual:SFQp}
and \eqref{eq:NME-dual:SFQd}, respectively. Moreover,
$$
[Q_{11}^{\T}-Y_0Q_{12}^{\T}+(Q_{21}^{\T}-Y_0Q_{22}^{\T})\Phi]^{-1}E_0=Z_{11}(\scrM_{\stb}\oplus J_1)Z_{11}^{-1}$$
and $$[Q_{22}-X_0Q_{12}+(Q_{21}-X_0Q_{11})\Psi]^{-1}F_0=\widetilde{Z}_{22}(\scrN_{\stb}\oplus \hat J_1^{\T})\widetilde{Z}_{22}^{-1}
$$
leading to \eqref{eq:crit-wstab-SF1:sols}.
\end{proof}

\begin{definition}\label{dfn:wstab-sol:SFQ}
Let $\Phi$ be a solution  of \eqref{eq:NME-dual:SFQp}
and $\Psi$ be a solution  of \eqref{eq:NME-dual:SFQd}. We
call $(\Phi,\Psi)$ a {\em generalized weakly stabilizing solution pair\/} of the primal \eqref{eq:NME-dual:SFQp}
and its dual \eqref{eq:NME-dual:SFQd}
if \eqref{eq:crit-wstab-SF1:sols} holds. In the case of $\varrho=1$,
we simply call it a {\em weakly stabilizing solution pair}.
\end{definition}

This notion of weakly stabilizing solution is motivated by a similar one in the optimal control theory.
Our main goal in what follows is to establish sufficient conditions for the existence of
a generalized weakly stabilizing solution pair to  the primal \eqref{eq:NME-dual:SFQp} and its dual
\eqref{eq:NME-dual:SFQd} and to prove that the $X$- and $Y$-sequences produced by SDASFQ (Algorithm~\ref{alg:DA-SFQ})
converge linearly to the weakly stabilizing solutions, respectively, with a linear rate $1/2$.

\begin{theorem}\label{thm:DA-cvg-SF1:crit}
Let $\scrA_0-\lambda\scrB_0$ be in {\rm SFQ} as in \eqref{eq:SFQ:i=0} and assume \eqref{eq:DA-cvg:crit:WCJ-SFQ}.
Let $\Phi$ and $\Psi$ be the ones in {\em Lemma~\ref{lm:DA-cvg-crit-SF1:sols}}, and
$\rho_{\max} = \max \{ \rho(\scrM_{\stb})/\varrho, \varrho\,\rho(\scrN_{\stb}) \}<1$.
If the sequence $\{ (E_i, X_i, Y_i, F_i)\}_{i=1}^{\infty}$
generated by {\em SDASFQ (Algorithm~\ref{alg:DA-SFQ})}  is well-defined, i.e.,
without any breakdown, then
\begin{enumerate}
\item[{\rm (a)}] $\max \{ \varrho^{-2^i}\| E_i \|, \varrho^{2^i}\| F_i \| \} \leq O(2^{-i}) + O(\rho_{\max}^{2^i})$ as $i \to
            \infty$,
\item[{\rm (b)}] $\| \Phi - X_i \| \leq  O(2^{-i}) + O(\rho_{\max}^{2^i}) $ as $i \to \infty$,
\item[{\rm (c)}]
                $\| \Psi - Y_i \| \leq O(2^{-i}) + O(\rho_{\max}^{2^i})$ as $i \to \infty$, and
\item[{\rm (d)}] $\wtd W_{i}$ and $W_{i}$ approach singular matrices $\wtd W=Q_{11}^{\T}-\Psi Q_{12}^{\T}+(Q_{21}^{\T}-\Psi Q_{22}^{\T})\Phi$ and $W=Q_{22}-\Phi Q_{12}+(Q_{21}-\Phi Q_{11})\Psi$, respectively.
    Furthermore,
    $$
    \wtd W(Q_1\bu_j)_{(1:m)}=0,\quad W(Q_2\bu_j)_{(m+1:m+n)}=0,
    $$
    where $\bu_j\in\bbC^{m+n}$ are the eigenvectors of
    $\scrA_0-\lambda\scrB_0$ associated with $\omega_j$, i.e., $(\scrA_0-\omega_j\scrB_0)\bu_j=0$
    for $1\le j\le r$.
%\\
%Furthermore,
%                 $$
%                 \wtd W(Q_1\bu_j)_{(1:m)}=0,\quad W(Q_2\bv_j)_{(m+1:m+n)}=0,
%                 $$
%                 where $\bu_j\in\bbC^{m+n}$ are the eigenvectors of
%                 $\scrA_0-\lambda\scrB_0$ associated with $\omega_j$, i.e., $(\scrA_0-\omega_j\scrB_0)\bu_j=0$
%                 for $1\le j\le r$, and  $\bv_j\in\bbC^{m+n}$ are the eigenvectors of
%                 $\scrB_0-\lambda\scrA_0$ associated with $\omega_j$, i.e., $(\scrB_0-\omega_j\scrA_0)\bv_j=0$
%                 for $1\le j\le r$.
\end{enumerate}
\end{theorem}

%\marginpar{\tiny \Red{$W(Q_2\bu_j)$}?}

\begin{proof}
Without loss of generality, we may assume $\varrho=1$; otherwise we may consider
$$
\scrA(\varrho)-\lambda\scrB(\varrho)
  :=\begin{bmatrix}
     \frac 1{\varrho} E_0 & 0 \\
      -X_0 & I
    \end{bmatrix}Q_1-\lambda\begin{bmatrix}
                           I & -Y_0 \\
                           0 & \varrho F_0
                         \end{bmatrix}Q_2,
$$
instead, because
\begin{enumerate}[(i)]
  \item SDASFQ (Algorithm~\ref{alg:DA-SFQ}) applied to $\scrA_0-\lambda\scrB_0$ generates
        the same approximations $X_i$ and $Y_i$ as it applied to $\scrA(\varrho)-\lambda\scrB(\varrho)$ but
        with $E_i$ and $F_i$ scaled to $\varrho^{-2^i} E_i$ and $\varrho^{2^i} F_i$, respectively, and
  \item if $\scrA_0$ and $\scrB_0$ satisfies \eqref{eq:DA-cvg:crit:WCJ-SFQ}, then
        $\scrA(\varrho)-\lambda\scrB(\varrho)$ satisfies \eqref{eq:DA-cvg:crit:WCJ-SFQ} with $\varrho$ set to $1$.
\end{enumerate}
Item (i) can be verified directly by going through the doubling iteration kernel \eqref{eq:EFXY-iter-SFQ}.
For item (ii), we notice that
\begin{alignat*}{3}
\scrA_0Z&=\scrB_0ZM&&\quad\Leftrightarrow\quad&\scrA(\varrho)Z&=\scrB(\varrho)Z(\frac 1{\varrho} M), \\
\scrA_0ZN&=\scrB_0Z&&\quad\Leftrightarrow\quad&\scrA(\varrho)Z(\varrho N)&=\scrB(\varrho)Z,
\end{alignat*}
which implies that $\scrA_0-\lambda\scrB_0$ and $\scrA(\varrho)-\lambda\scrB(\varrho)$ have the same Weierstrass canonical
structure, except with all eigenvalues $\lambda$ scaled to $\lambda/\varrho$.

Suppose from now on $\varrho=1$. Thus in \eqref{eq:crit-WCJ-SFQa'}, $\what J_1^{\T}=J_1^{\HH}$.

Plug in $J_{\scrA_0}$, $J_{\scrB_0}$, $\what J_{\scrA_0}$, and $\what J_{\scrB_0}$ into
\eqref{eq:DA-cvg-crit-SF1:rels} to obtain
\begin{subequations}\label{eq:crit-cvg:SFQ-i}
\begin{align}
\scrA_i \scrU \begin{bmatrix}
                 I_m & 0_{m\times n} \\
                 0_{n\times m} & \scrN_{\stb}^{2^i} \oplus I_{n_0}
              \end{bmatrix}
   &= \scrB_i\scrU \begin{bmatrix}
                      \scrM_{\stb}^{2^i} \oplus J_1^{2^i} & 0_{m'\times n'} \oplus \Gamma_i \\
                       0_{n\times m} & I_{n'} \oplus J_1^{2^i}
                   \end{bmatrix}, \label{eq:crit-cvg:SFQ-ia} \\
\scrA_i \scrV \begin{bmatrix}
                     I_{m'}\oplus (J_1^{\HH})^{2^i} & 0_{m\times n} \\
                     0_{n'\times m'}\oplus \Gamma_i^{\T} & \scrN_{\stb}^{2^i} \oplus (J_1^{\HH})^{2^i}
             \end{bmatrix}
   &= \scrB_i
     \scrV \begin{bmatrix}
              \scrM_{\stb}^{2^i} \oplus I_{n_0} & 0_{m\times n}   \\
              0_{n\times m} & I_{n}
           \end{bmatrix}, \label{eq:crit-cvg:SFQ-ib}
\end{align}
\end{subequations}
where $\Gamma_i = \Gamma_{i,m_1} \oplus  \cdots \oplus \Gamma_{i,m_r}$ with $\Gamma_{i,m_j}$
is as
defined  in \eqref{eq:mtx_Gamma_k}.

Let\begin{equation}\label{eq:scrZ}
	\scrZ=Q_1\scrU = \kbordermatrix{ & \sss m & \sss n \\
		\sss m &Z_{11} & Z_{12} \\
		\sss n & Z_{21} & Z_{22} }
\end{equation}
and
\begin{equation}\label{eq:scrG}
	\scrG=Q_2Q_1^{\T}\scrZ =:\begin{bmatrix}
		G_{11} & G_{12}\\
		G_{21} & G_{22}
	\end{bmatrix}=\begin{bmatrix}
	Q_{11}^{\T}Z_{11}+Q_{21}^{\T}Z_{21}& Q_{11}^{\T}Z_{12}+Q_{21}^{\T}Z_{22}\\
	Q_{12}^{\T}Z_{11}+Q_{22}^{\T}Z_{21}& Q_{12}^{\T}Z_{12}+Q_{22}^{\T}Z_{22}
	\end{bmatrix}.
\end{equation}
Substitute $\scrA_i$ and $\scrB_i$ in \eqref{eq:preservation-SFQ}, $\scrZ$ in \eqref{eq:scrZ} and $\scrG$ in \eqref{eq:scrG} into \eqref{eq:crit-cvg:SFQ-ia} to get
\begin{subequations} \label{eq:crit-cvg:SFQ-pf-I}
\begin{align}
E_i Z_{11} &=(G_{11}-Y_iG_{21})( \scrM_{\stb}^{2^i} \oplus J_1^{2^i}), \label{eq:crit-cvg:SFQ-pf-1} \\
E_i Z_{12} (\scrN_{\stb}^{2^i} \oplus I_{n_0})
      &= (G_{11}-Y_iG_{21}) ( 0_{m'\times n'} \oplus \Gamma_i) \nonumber \\*
      &\quad + (G_{12}-Y_iG_{22})(I_{n'} \oplus J_1^{2^i}), \label{eq:crit-cvg:SFQ-pf-2} \\
-X_i Z_{11} + Z_{21}
      &= F_iG_{21} (\scrM_{\stb}^{2^i} \oplus J_1^{2^i}), \label{eq:crit-cvg:SFQ-pf-3} \\*
(-X_i Z_{12} + Z_{22} )( \scrN_{\stb}^{2^i} \oplus I_{n_0})
      &= F_iG_{21}( 0_{m'\times n'}\oplus \Gamma_i) + F_iG_{22}(I_{n'} \oplus J_1^{2^i}). \label{eq:crit-cvg:SFQ-pf-4}
\end{align}
\end{subequations}
Post-multiply \eqref{eq:crit-cvg:SFQ-pf-1} by $Z_{11}^{-1}$ and
\eqref{eq:crit-cvg:SFQ-pf-2} by $(0_{n'\times m'}\oplus \Gamma_i^{-1}J_1^{2^i})Z_{11}^{-1}$ to get,
respectively,
\begin{align}
E_i  &= ( G_{11} - Y_i G_{21} ) ( \scrM_{\stb}^{2^i} \oplus J_1^{2^i}) Z_{11}^{-1}, \tag{\ref{eq:crit-cvg:SFQ-pf-1}'} \\
E_i Z_{12} (0_{n'\times m'}\oplus \Gamma_i^{-1}J_1^{2^i}) Z_{11}^{-1}
      &= ( G_{11} - Y_i G_{21}) ( 0_{m'\times m'} \oplus J_1^{2^i})Z_{11}^{-1} \nonumber \\
      &\quad+ (G_{12} - Y_i G_{22})(0_{n'\times m'} \oplus J_1^{2^i} \Gamma_i^{-1} J_1^{2^i}) Z_{11}^{-1}.
           \tag{\ref{eq:crit-cvg:SFQ-pf-2}'}
\end{align}
Subtracting (\ref{eq:crit-cvg:SFQ-pf-2}') from (\ref{eq:crit-cvg:SFQ-pf-1}'),
we obtain
\begin{multline}\label{eq:crit-cvg:SFQ-pf-5}
E_i \left[ I - Z_{12} (0_{n'\times m'}\oplus \Gamma_i^{-1}J_1^{2^i}) Z_{11}^{-1}\right]  \\
      = ( G_{11} - Y_i G_{21}) ( \scrM_{\stb}^{2^i} \oplus 0_{n_0\times n_0}) Z_{11}^{-1}
         - (G_{12} - Y_i G_{22}) (0_{n'\times m'} \oplus J_1^{2^i} \Gamma_i^{-1} J_1^{2^i}) Z_{11}^{-1}.
\end{multline}
Post-multiply \eqref{eq:crit-cvg:SFQ-pf-3} by $Z_{11}^{-1}$ and
\eqref{eq:crit-cvg:SFQ-pf-4} by $(0_{n'\times m'} \oplus \Gamma_i^{-1} J_1^{2^i})Z_{11}^{-1}$ to get,
respectively,
\begin{gather}
-X_i  + \Phi
      = F_i G_{21} (\scrM_{\stb}^{2^i} \oplus J_1^{2^i})Z_{11}^{-1}, \tag{\ref{eq:crit-cvg:SFQ-pf-3}'} \\
(-X_i Z_{12} + Z_{22} )(0_{n'\times m'} \oplus \Gamma_i^{-1} J_1^{2^i})Z_{11}^{-1}
      = F_i G_{21} ( 0_{m'\times m'}\oplus J_1^{2^i})Z_{11}^{-1} \nonumber \\
\hphantom{
    (-X_i Z_{12} + Z_{22} )(0_{n'\times m'} \oplus \Gamma_i^{-1} J_1^{2^i})Z_{11}^{-1}=
 }
         + F_i G_{22} (0_{n'\times m'} \oplus J_1^{2^i}\Gamma_i^{-1} J_1^{2^i})Z_{11}^{-1}.
      \tag{\ref{eq:crit-cvg:SFQ-pf-4}'}
\end{gather}
Subtracting (\ref{eq:crit-cvg:SFQ-pf-4}') from (\ref{eq:crit-cvg:SFQ-pf-3}'), we obtain
\begin{multline}\label{eq:crit-cvg:SFQ-pf-X}
-X_i \left[ I - Z_{12} (0_{n'\times m'} \oplus \Gamma_i^{-1} J_1^{2^i}) Z_{11}^{-1} \right] + \Phi
      =Z_{22} (0_{n'\times m'} \oplus \Gamma_i^{-1} J_1^{2^i}) Z_{11}^{-1} \\
        - F_i G_{22} (0_{n'\times m'} \oplus J_1^{2^i} \Gamma_i^{-1} J_1^{2^i}) Z_{11}^{-1}
        + F_i G_{21} (\scrM_{\stb}^{2^i} \oplus 0_{n_0\times n_0} ) Z_{11}^{-1}.
\end{multline}
By Lemma~\ref{lem:nonsingular_Gamma}, \eqref{eq:crit-cvg:SFQ-pf-X} can be written as
\begin{equation}\label{eq:crit-cvg:SFQ-pf-X'}
X_i[ I + O_{m\times m}(2^{-i})] = \Phi + O_{n\times m}(2^{-i}) + F_i\big[ O_{n\times m}(2^{-i}) + O_{n\times m}(\rho(\scrM_{\stb})^{2^i})\big],
\end{equation}
where each $O_{n\times m}(\epsilon)$
denotes some $n\times m$ matrix with entries in the order of $\epsilon$.

Let\begin{equation}\label{eq:widetildescrZ}
	\widetilde\scrZ=Q_2\scrV = \kbordermatrix{ & \sss m & \sss n \\
		\sss m &\widetilde Z_{11} &\widetilde Z_{12} \\
		\sss n &\widetilde Z_{21} &\widetilde Z_{22} }
\end{equation}
and
\begin{equation}\label{eq:widetildescrG}
	\widetilde \scrG=Q_1Q_2^{\T}\widetilde \scrZ =:\begin{bmatrix}
	\widetilde	G_{11} &\widetilde G_{12}\\
	\widetilde	G_{21} &\widetilde G_{22}
	\end{bmatrix}=\begin{bmatrix}
		Q_{11}\widetilde Z_{11}+Q_{12}\widetilde Z_{21}& Q_{11}\widetilde Z_{12}+Q_{12} \widetilde Z_{22}\\
		Q_{21}\widetilde Z_{11}+Q_{22}\widetilde Z_{21}& Q_{21}\widetilde Z_{12}+Q_{22}\widetilde Z_{22}
	\end{bmatrix}.
\end{equation}
Next we substitute $\scrA_i$ and $\scrB_i$ in \eqref{eq:preservation-SFQ}, $\widetilde\scrZ$ in \eqref{eq:widetildescrZ} and $\widetilde\scrG$ in \eqref{eq:widetildescrG}
into \eqref{eq:crit-cvg:SFQ-ib} to get
\begin{subequations} \label{eq:crit-cvg:SFQ-pf-II}
\begin{align}
F_i \widetilde Z_{22} &= (-X_i\widetilde G_{12} +\widetilde G_{22}) \big[\scrN_{\stb}^{2^i} \oplus (J_1^{\HH})^{2^i}\big], \label{eq:crit-cvg:SFQ-pf-6} \\
F_i \widetilde Z_{21} (\scrM_{\stb}^{2^i} \oplus I_{n_0})
       &= (-X_i\widetilde G_{12} + \widetilde G_{22}) (0_{n'\times m'} \oplus \Gamma_i^{\T}) \nonumber \\
       &\quad   +(-X_i \widetilde G_{11} +\widetilde G_{21}) \big[I_{m'} \oplus (J_1^{\HH})^{2^i}\big], \label{eq:crit-cvg:SFQ-pf-7} \\
\widetilde Z_{12} - Y_i\widetilde Z_{22}
       &= E_i\widetilde G_{12} \big[\scrN_{\stb}^{2^i} \oplus (J_1^{\HH})^{2^i}\big] \nonumber \\
       &= E_i \widetilde G_{12} (\scrN_{\stb}^{2^i} \oplus 0_{n_0\times n_0}) +  E_i\widetilde G_{12} (0_{n'\times n'} \oplus (J_1^{\HH})^{2^i}),
        \label{eq:crit-cvg:SFQ-pf-8} \\
(\widetilde Z_{11} - Y_i \widetilde Z_{21})(\scrM_{\stb}^{2^i} \oplus I_{n_0})
       &= E_i \widetilde G_{11} \big[ I_{m'} \oplus (J_1^{\HH})^{2^i}\big] + E_i\widetilde G_{12}(0_{n'\times m'} \oplus \Gamma_i^{\T}).
        \label{eq:crit-cvg:SFQ-pf-9}
\end{align}
\end{subequations}
Post-multiply \eqref{eq:crit-cvg:SFQ-pf-6} by $\widetilde Z_{22}^{-1}$ and
\eqref{eq:crit-cvg:SFQ-pf-7} by $\big[0_{m'\times n'}\oplus \Gamma_i^{-\T}(J_1^{2^i})^{\HH}\big]\widetilde Z_{22}^{-1}$ to get,
respectively,
%\begin{equation}\tag{\ref{eq:crit-cvg:SFQ-pf-6}'}
%F_i  = (-X_i\widetilde G_{12} +\widetilde G_{22})
%       \big[\scrN_{\stb}^{2^i} \oplus (J_1^{\HH})^{2^i}\big]\widetilde Z_{22}^{-1},
%\end{equation}
\begin{gather}
F_i  = (-X_i\widetilde G_{12} +\widetilde G_{22})
       \big[\scrN_{\stb}^{2^i} \oplus (J_1^{\HH})^{2^i}\big]\widetilde Z_{22}^{-1},
          \tag{\ref{eq:crit-cvg:SFQ-pf-6}'}\\
F_i \widetilde Z_{21} \big[0_{m'\times n'}\oplus \Gamma_i^{-\T}(J_1^{2^i})^{\HH}\big]\widetilde Z_{22}^{-1}
       = (-X_i\widetilde G_{12} +\widetilde G_{22}) \big[0_{n'\times n'} \oplus (J_1^{2^i})^{\HH}\big]\widetilde Z_{22}^{-1} \nonumber\\
\hspace*{4cm}          +(-X_i\widetilde G_{11} +\widetilde G_{21}) \big[0_{m'\times n'} \oplus (J_1^{\HH})^{2^i}\Gamma_i^{-\T}(J_1^{2^i})^{\HH}\big]\widetilde Z_{22}^{-1}.
           \tag{\ref{eq:crit-cvg:SFQ-pf-7}'}
\end{gather}
Subtracting (\ref{eq:crit-cvg:SFQ-pf-7}') from (\ref{eq:crit-cvg:SFQ-pf-6}'),
we obtain
\begin{multline}\label{eq:crit-cvg:SFQ-pf-10}
F_i \left\{ I -\widetilde Z_{21} \big[0_{m'\times n'}\oplus \Gamma_i^{-\T}(J_1^{2^i})^{\HH}\big]\widetilde Z_{22}^{-1}\right\}
      = (-X_i\widetilde G_{12} + \widetilde G_{22}) ( \scrN_{\stb}^{2^i} \oplus 0_{n_0\times n_0})\widetilde Z_{22}^{-1} \\
         - (-X_i\widetilde G_{11} +\widetilde G_{21}) \big[0_{m'\times n'} \oplus (J_1^{\HH})^{2^i}\Gamma_i^{-\T}(J_1^{2^i})^{\HH}\big]\widetilde Z_{22}^{-1}.
\end{multline}
Post-multiply \eqref{eq:crit-cvg:SFQ-pf-8} by $\widetilde Z_{22}^{-1}$ and
\eqref{eq:crit-cvg:SFQ-pf-9} by $\big[0_{m'\times n'} \oplus \Gamma_i^{-\T} (J_1^{\HH})^{2^i}\big]\widetilde Z_{22}^{-1}$
to get,
respectively,
\begin{gather}
\Psi - Y_i= E_i \widetilde G_{12} (\scrN_{\stb}^{2^i} \oplus 0_{n_0\times n_0})\widetilde Z_{22}^{-1}
          +  E_i\widetilde G_{12} \big[0_{n'\times n'} \oplus (J_1^{\HH})^{2^i}\big]\widetilde Z_{22}^{-1},
        \tag{\ref{eq:crit-cvg:SFQ-pf-8}'} \\
(\widetilde Z_{11} - Y_i\widetilde Z_{21})\big[0_{m'\times n'} \oplus \Gamma_i^{-\T} (J_1^{\HH})^{2^i}\big]\widetilde Z_{22}^{-1}
       = E_i \widetilde G_{11} \big[ 0_{m'\times n'} \oplus (J_1^{\HH})^{2^i}\Gamma_i^{-\T} (J_1^{\HH})^{2^i}\big]\widetilde Z_{22}^{-1} \nonumber \\
  \hphantom{
       (V_{11} - Y_i V_{21})\big[0_{m'\times n'} \oplus \Gamma_i^{-\T} (J_1^{\HH})^{2^i}\big] V_{22}^{-1}=
           }
   + E_i \widetilde G_{12} \big[0_{n'\times n'} \oplus (J_1^{\HH})^{2^i}\big]\widetilde Z_{22}^{-1}.
               \tag{\ref{eq:crit-cvg:SFQ-pf-9}'}
\end{gather}
Subtracting (\ref{eq:crit-cvg:SFQ-pf-9}') from (\ref{eq:crit-cvg:SFQ-pf-8}'), we obtain
\begin{multline}\label{eq:crit-cvg:SFQ-pf-Y}
\Psi - Y_i \big\{ I - \widetilde Z_{21} \big[0_{m'\times n'} \oplus \Gamma_i^{-\T} (J_1^{\HH})^{2^i}\big] \widetilde Z_{22}^{-1} \big\}
     =\widetilde Z_{11} \big[0_{m'\times n'} \oplus \Gamma_i^{-\T}(J_1^{\HH})^{2^i}\big]\widetilde Z_{22}^{-1} \\
         - E_i \widetilde G_{11} \big[ 0_{m'\times n'} \oplus (J_1^{\HH})^{2^i}\Gamma_i^{-\T} (J_1^{\HH})^{2^i}\big]\widetilde Z_{22}^{-1}
         + E_i\widetilde G_{12} (\scrN_{\stb}^{2^i} \oplus 0_{n_0\times n_0})\widetilde Z_{22}^{-1}.
\end{multline}
By Lemma~\ref{lem:nonsingular_Gamma}, \eqref{eq:crit-cvg:SFQ-pf-Y} can be written as
\begin{equation}\label{eq:crit-cvg:SFQ-pf-Y'}
       Y_i[ I + O_{n\times n}(2^{-i})] = \Psi + O_{m\times n}(2^{-i}) + E_i\big[ O_{m\times n}(2^{-i}) + O_{m\times n}(\rho(\scrN_{\stb})^{2^i})\big].
\end{equation}

Substituting $Y_i$ in \eqref{eq:crit-cvg:SFQ-pf-Y'} into \eqref{eq:crit-cvg:SFQ-pf-5}
and by Lemma~\ref{lem:nonsingular_Gamma}, we have
\begin{equation}\label{eq:crit-cvg:SFQ-pf-E}
      \| E_i \| \leq O(2^{-i}) + O(\rho_{\max}^{2^i}) \to 0, \mbox{ as } i \to \infty.
\end{equation}
Substituting $X_i$ in \eqref{eq:crit-cvg:SFQ-pf-X'} into \eqref{eq:crit-cvg:SFQ-pf-10}
and by Lemma~\ref{lem:nonsingular_Gamma}, we have
\begin{equation}\label{eq:crit-cvg:SFQ-pf-F}
      \| F_i \| \leq O(2^{-i}) + O(\rho_{\max}^{2^i}) \to 0, \mbox{ as } i \to \infty.
\end{equation}
Item (a) is a consequence of \eqref{eq:crit-cvg:SFQ-pf-E} and \eqref{eq:crit-cvg:SFQ-pf-F}.

From \eqref{eq:crit-cvg:SFQ-pf-F}
and \eqref{eq:crit-cvg:SFQ-pf-X'} it follows that
%\marginpar{\tiny delete the red part? }
$$
\|\Phi - X_i \| \leq O(2^{-i}) + O(\rho_{\max}^{2^i}) \to 0.
$$
This proves item (b).
From \eqref{eq:crit-cvg:SFQ-pf-E} and \eqref{eq:crit-cvg:SFQ-pf-Y'}  it follows that
$$
\| \Psi - Y_i \| \leq O(2^{-i}) + O(\rho_{\max}^{2^i}) \to 0.
$$
This proves item~(c).

For item (d), substitute the formulas of $G_{11}$ and $G_{21}$ in \eqref{eq:scrG} to \eqref{eq:crit-cvg:SFQ-pf-1}, we obtain
\begin{align}
	E_iZ_{11}&=\big[Q^{\T}_{11}Z_{11}+Q_{21}^{\T}Z_{21}-Y_i(Q_{12}^{\T}Z_{11}+Q_{22}^{\T}Z_{21})\big] (\scrM_{\stb}^{2^i} \oplus J_1^{2^i})\nonumber\\
	&=(Q^{\T}_{11}-Y_iQ_{12}^{\T})Z_{11}(\scrM_{\stb}^{2^i} \oplus J_1^{2^i})+(Q_{21}^{\T}-Y_iQ_{22}^{\T})\Phi Z_{11}(\scrM_{\stb}^{2^i} \oplus J_1^{2^i}).\label{eq:EiZ11}
\end{align}
Rewrite \eqref{eq:crit-cvg:SFQ-pf-3} as
\begin{equation}\label{eq:crit-cvg:SFQ-pf-3new}
	\Phi Z_{11}=X_iZ_{22}+F_i(Q_{22}^{\T}\Phi+Q_{12}^{\T}Z_{11}(\scrM_{\stb}^{2^i} \oplus J_1^{2^i}),
\end{equation}
and  then post-multiplying \eqref{eq:crit-cvg:SFQ-pf-3new} by $\scrM_{\stb}^{2^i} \oplus J_1^{2^i}$ gives
\begin{equation}\label{eq:middle-two}
	\Phi Z_{11}(\scrM_{\stb}^{2^i} \oplus J_1^{2^i})=X_iZ_{22}(\scrM_{\stb}^{2^i} \oplus J_1^{2^i})+F_i(Q_{22}^{\T}\Phi+Q_{12}^{\T}Z_{11}\big(\scrM_{\stb}^{2^i} \oplus J_1^{2^i}\big)^2.
\end{equation}
Plug \eqref{eq:middle-two} into \eqref{eq:EiZ11} to get,
\begin{equation}\label{eq:EiZ11-new}
E_iZ_{11}=\widetilde{W}_iZ_{11}	(\scrM_{\stb}^{2^i} \oplus J_1^{2^i})+(Q_{21}^{\T}-Y_iQ_{22}^{\T})F_i(Q_{12}^{\T}+Q_{22}^{\T}\Phi)Z_{11}(\scrM_{\stb}^{2^i} \oplus J_1^{2^i})^2.
\end{equation}
Post-multiplying \eqref{eq:EiZ11-new} by
$(0_{m'\times m'}\oplus J_1^{-2^i})\begin{bmatrix}
	0_{m'\times 1}  \\
	\be_1
\end{bmatrix}$, we get
\begin{equation}\label{eq:tildeWZ11}
	\widetilde{W}_iZ_{11}\begin{bmatrix}
		0_{m'\times 1}  \\
		\be_1
	\end{bmatrix}=-E_iZ_{11}\begin{bmatrix}
		0_{m'\times 1} \\ \omega_1^{-2^i} \be_1
	\end{bmatrix}+(Q_{21}^{\T}-Y_iQ_{22}^{\T})F_i(Q_{12}^{\T}+Q_{22}^{\T}\Phi)Z_{11}\begin{bmatrix}
		0_{m'\times 1} \\ \omega_1^{2^i} \be_1
	\end{bmatrix},
\end{equation}
which converges to $0$ as $i \to \infty$ by using item (a) and item (c). Therefore, $\widetilde{W}_i$ converges to the singular matrix
$\wtd W=Q_{11}^{\T}-\Psi Q_{12}^{\T}+(Q_{21}^{\T}-\Psi Q_{22}^{\T})\Phi$  with
\begin{equation}
\widetilde	WZ_{11}\begin{bmatrix}
		0_{m'\times 1}  \\
		\be_1
	\end{bmatrix}=0.
\end{equation}
Notice that from \eqref{eq:crit-WCJ-SFQa} we find that $\be_1$
in \eqref{eq:tildeWZ11} can be replaced by
$\be_{m_1+\cdots+m_j+1}$ for $1\le j\le r-1$.

Similarly, we work with \eqref{eq:crit-cvg:SFQ-pf-6} and \eqref{eq:crit-cvg:SFQ-pf-8} to prove
$W_i\widetilde{Z}_{22}\be_{n'+1}  \to 0$ as $i \to \infty$. Substituting the formulas of $\widetilde{G}_{12}$ and $\widetilde{G}_{22}$ in \eqref{eq:widetildescrG} into \eqref{eq:crit-cvg:SFQ-pf-6}, we obtain
\begin{align}
	F_i\widetilde{Z}_{22}&=\big[Q_{22}-X_iQ_{12}+(Q_{21}-X_iQ_{11})\Psi\big]\widetilde{Z}_{22}\big[\scrN_{\stb}^{2^i} \oplus (J_1^{\HH})^{2^i}\big]\nonumber\\
	&=(Q_{22}-X_iQ_{12})\widetilde{Z}_{22}\big[\scrN_{\stb}^{2^i} \oplus (J_1^{\HH})^{2^i}\big]+(Q_{21}-X_iQ_{11})\Psi\widetilde{Z}_{22}\big[\scrN_{\stb}^{2^i} \oplus (J_1^{\HH})^{2^i}\big].\label{eq:FiZ22}
\end{align}
Rewrite \eqref{eq:crit-cvg:SFQ-pf-8} as
\begin{equation}\label{eq:crit-cvg:SFQ-pf-8new}	\Psi\widetilde{Z}_{22}=Y_i\widetilde{Z}_{22}+E_i(Q_{11}\Psi+Q_{12})\widetilde{Z}_{22}\big[\scrN_{\stb}^{2^i} \oplus (J_1^{\HH})^{2^i}\big],
\end{equation}
and  then post-multiplying \eqref{eq:crit-cvg:SFQ-pf-8new} by $\scrN_{\stb}^{2^i} \oplus (J_1^{\HH})^{2^i}$ gives
\begin{equation}\label{eq:middle-one}
	\Psi\widetilde{Z}_{22}\big[\scrN_{\stb}^{2^i} \oplus (J_1^{\HH})^{2^i}\big]=Y_i\widetilde{Z}_{22}\big[\scrN_{\stb}^{2^i} \oplus (J_1^{\HH})^{2^i}\big]+E_i(Q_{11}\Psi+Q_{12})\widetilde{Z}_{22}\big[\scrN_{\stb}^{2^i} \oplus (J_1^{\HH})^{2^i}\big]^2.
\end{equation}
Plug \eqref{eq:middle-one} into \eqref{eq:FiZ22} to get,
\begin{equation}
	F_i\widetilde{Z}_{22}
	=W_i\big[\scrN_{\stb}^{2^i} \oplus (J_1^{\HH})^{2^i}\big]+(Q_{21}-X_iQ_{11})E_i(Q_{11}\Psi+Q_{12})\widetilde{Z}_{22}\big[\scrN_{\stb}^{2^i} \oplus (J_1^{\HH})^{2^i}\big]^2.\label{eq:FiZ22-new}
\end{equation}
Post-multiplying \eqref{eq:FiZ22-new} by
$(0_{n'\times n'}\oplus (J_1^{\HH})^{-2^i})\begin{bmatrix}
	0_{n'\times 1}  \\
	\be_{m_1}
\end{bmatrix}$, we get
\begin{equation}\label{eq:WtildeZ22}
W_i\widetilde{Z}_{22}\begin{bmatrix}
	0_{n'\times 1}  \\
	\be_{m_1}
\end{bmatrix}=-F_i\widetilde{Z}_{22}\begin{bmatrix}
		0_{n'\times 1}  \\
		(\overline{\omega}_1)^{-2^i}\be_{m_1}
	\end{bmatrix}+(Q_{21}-X_iQ_{11})E_i(Q_{11}\Psi+Q_{12})\widetilde{Z}_{22}\begin{bmatrix}
	0_{n'\times 1}  \\
	\overline{\omega}_1^{2^i}\be_{m_1}
	\end{bmatrix},
\end{equation}
which converges to $0$ as $i \to \infty$ by using item (a) and item (b). Therefore, $W_i$ converges to the singular matrix
$W=Q_{22}-\Phi Q_{12}+(Q_{21}-\Phi Q_{11})\Psi$
with
\begin{equation}
	W\widetilde{Z}_{22}\begin{bmatrix}
		0_{n'\times 1}  \\
		\be_{m_1}
	\end{bmatrix}=0.
\end{equation}
Notice that from \eqref{eq:crit-WCJ-SFQa'} we find that $\be_{m_1}$
in \eqref{eq:WtildeZ22} can be replaced by $\be_{m_1+m_2+\cdots+m_j}$ for $2\le j\le r$.
\end{proof}

\begin{remark}\label{rk:DA-cvg-SF1:crit}
{\rm
It is important to realize that the generalized weakly stabilizing solution pair
$(\Phi,\Psi)$ in the theorem is specifically constructed in Lemma~\ref{lm:DA-cvg-crit-SF1:sols}
and that it is the one that gets computed by SDASFQ (Algorithm~\ref{alg:DA-SFQ}).
It is possible that there are other generalized weakly stabilizing solution pairs in the sense of
Definition~\ref{dfn:wstab-sol:SFQ} but these pairs are not gotten computed by SDASFQ.
%This remark applies to Theorem~\ref{thm:DA-cvg-SF2:crit} below.
}
\end{remark}

\begin{remark}\label{rk:DA-cvg-SF1:crit'}
{\rm
A linear convergence rate $1/2$ is often very good for practical purposes, but it is still slower than quadratic convergence.
For particular nonlinear matrix equations, it is possible to still retain quadratic convergence with special techniques.
One of the examples is \MARE\  in the critical case, for which efforts have been made
\cite{guim:2007,huhl:2016,iapo:2013,wawl:2013} to achieve quadratic convergence.
}
\end{remark}

%\clearpage
{\small
%\bibliographystyle{plain}
%\bibliography{\TeXHOME/BIB/strings,\TeXHOME/BIB/mxptrefs,\TeXHOME/BIB/li}
%\bibliography{BIB/strings,BIB/mxptrefs,BIB/li}

\begin{thebibliography}{10}

\bibitem{badg:1997}
Zhaojun Bai, James Demmel, and Ming Gu.
\newblock An inverse free parallel spectral divide and conquer algorithm for
  nonsymmetric eigenproblems.
\newblock {\em Numer. Math.}, 76:279--308, 1997.

\bibitem{bali:2012a}
Zhaojun Bai and Ren-Cang Li.
\newblock Minimization principle for linear response eigenvalue problem, {I}:
  Theory.
\newblock {\em SIAM J. Matrix Anal. Appl.}, 33(4):1075--1100, 2012.

\bibitem{bali:2013}
Zhaojun Bai and Ren-Cang Li.
\newblock Minimization principles for the linear response eigenvalue problem
  {II}: Computation.
\newblock {\em SIAM J. Matrix Anal. Appl.}, 34(2):392--416, 2013.

\bibitem{bali:2014}
Zhaojun Bai and Ren-Cang Li.
\newblock Minimization principles and computation for the generalized linear
  response eigenvalue problem.
\newblock {\em BIT Numer. Math.}, 54(1):31--54, 2014.

\bibitem{ball:2016}
Zhaojun Bai, Ren-Cang Li, and Wen-Wei Lin.
\newblock Linear response eigenvalue problem solved by extended locally optimal
  preconditioned conjugate gradient methods.
\newblock {\em SCIENCE CHINA Math.}, 59(8):1443--1460, 2016.

\bibitem{benn:1997}
P.~Benner.
\newblock {\em Contributions to the Numerical Solutions of Algebraic {R}iccati
  Equations and Related Eigenvalue Problems}.
\newblock PhD thesis, Fakult{\"{a}}t f{\"{u}}r Mathematik, TU Chemnitz-Zwickau,
  1997.

\bibitem{biim:2012}
Dario~A. Bini, Bruno Iannazzo, and Beatrice Meini.
\newblock {\em Numerical Solution of Algebraic Riccati Equations}.
\newblock SIAM, Philadelphia, 2012.

\bibitem{bimp:2010}
Dario~A. Bini, Beatrice Meini, and Federico Poloni.
\newblock Transforming algebraic {R}iccati equations into unilateral quadratic
  matrix equations.
\newblock {\em Numer. Math.}, 116:553--578, 2010.

\bibitem{bugo:1988}
A.~Ya. Bulgakov and S.~K. Godunov.
\newblock Circular dichotomy of the spectrum of a matrix.
\newblock {\em Siberian Math. J.}, 29:734--744, 1988.

\bibitem{chfl:2005}
E.~K.-W. Chu, H.-Y. Fan, and W.-W. Lin.
\newblock A structure-preserving doubling algorithm for continuous-time
  algebraic {R}iccati equations.
\newblock {\em Linear Algebra Appl.}, 396:55--80, 2005.

\bibitem{chfl:2004}
E.~K.~W. Chu, H.-Y. Fan, W.~W. Lin, and C.~S. Wang.
\newblock Structure-preserving algorithms for periodic discrete-time algebraic
  {R}iccati equations.
\newblock {\em Int. J. Control}, 77(8):767--788, 2004.

\bibitem{dhma:2007}
F.~R. {de Hoog} and R.~M.~M. Mattheij.
\newblock Subset selection for matrices.
\newblock {\em Linear Algebra Appl.}, 422:349--359, 2007.

\bibitem{gant:1959}
F.~R. Gantmacher.
\newblock {\em The Theory of Matrices, Vols. I, II}.
\newblock Chelsea Publishing Company, New York, 1959.

\bibitem{godu:1986}
S.~K. Godunov.
\newblock Problem of the dichotomy of the spectrum of a matrix.
\newblock {\em Siberian Math. J.}, 27:649--660, 1986.

\bibitem{govl:1996}
G.~H. Golub and C.~F. {Van Loan}.
\newblock {\em Matrix Computations}.
\newblock Johns Hopkins University Press, Baltimore, Maryland, 3rd edition,
  1996.

\bibitem{gotz:1997}
S.~A. Goreinov, E.~E. Tyrtyshnikov, and N.~L. Zamarashkin.
\newblock A theory of pseudoskeleton approximations.
\newblock {\em Linear Algebra Appl.}, 261(1-3):1--21, 1997.

\bibitem{guim:2007}
Chun-Hua Guo, Bruno Iannazzo, and Beatrice Meini.
\newblock On the doubling algorithm for a (shifted) nonsymmetric algebraic
  {R}iccati equation.
\newblock {\em SIAM J. Matrix Anal. Appl.}, 29(4):1083--1100, 2007.

\bibitem{gulx:2006}
X.~Guo, W.~Lin, and S.~Xu.
\newblock A structure-preserving doubling algorithm for nonsymmetric algebraic
  {R}iccati equation.
\newblock {\em Numer. Math.}, 103:393--412, 2006.

\bibitem{gucl:2019}
Zhen-Chen Guo, Eric King-Wah Chu, and Wen-Wei Lin.
\newblock Doubling algorithm for the discretized {Bethe-Salpeter} eigenvalue
  problem.
\newblock {\em Math. Comp.}, 88(319):2325–2350, 2019.

\bibitem{hopa:1992}
Y.~P. Hong and C.-T. Pan.
\newblock Rank-revealing {QR} factorizations and the singular value
  decomposition.
\newblock {\em Math. Comp.}, 58(197):213--232, 1992.

\bibitem{huhl:2016}
T.-M. Huang, W.-Q. Huang, R.-C. Li, and W.-W. Lin.
\newblock A new two-phase structure-preserving doubling algorithm for
  critically singular {$M$}-matrix algebraic {Riccati} equations.
\newblock {\em Numer. Linear Algebra Appl.}, 23:291--313, 2016.

\bibitem{huli:2009}
T.-M. Huang and W.-W. Lin.
\newblock Structured doubling algorithms for weakly stabilizing {H}ermitian
  solutions of algebraic {R}iccati equations.
\newblock {\em Linear Algebra Appl.}, 430:1452--1478, 2009.

\bibitem{hull:2018}
Tsung-Ming Huang, Ren-Cang Li, and Wen-Wei Lin.
\newblock {\em Structure-Preserving Doubling Algorithms For Nonlinear Matrix
  Equations}, volume~14 of {\em Fundamentals of Algorithms}.
\newblock SIAM, Philadelphia, September 2018.

\bibitem{hull:2017}
Tsung-Ming Huang, Ren-Cang Li, Wen-Wei Lin, and Linzhang Lu.
\newblock Optimal parameters for doubling algorithms.
\newblock {\em J. Math. Study}, 50(4):339--357, 2017.

\bibitem{iapo:2013}
Bruno Iannazzo and Federico Poloni.
\newblock A subspace shift technique for nonsymmetric algebraic {R}iccati
  equations associated with an ${M}$-matrix.
\newblock {\em Numer. Linear Algebra Appl.}, 20(3):440--452, 2013.

\bibitem{laro:1995}
Peter Lancaster and Leiba Rodman.
\newblock {\em Algebraic Riccati Equations}.
\newblock Oxford University Press, New York, USA, 1995.

\bibitem{laub:1979}
A.~Laub.
\newblock A {Schur} method for solving algebraic {Riccati} equations.
\newblock {\em IEEE Trans. Automat. Control}, 24(6):913--921, 1979.

\bibitem{laub:1991}
A.~J. Laub.
\newblock Invariant subspace methods for the numerical solution of {R}iccati
  equations.
\newblock In S.~Bittanti, A.~J. Laub, and J.~C. Willems, editors, {\em The
  {R}iccati Equation}, pages 163--196. Springer-Verlag, Berlin, 1991.

\bibitem{lixu:2006}
W.-W. Lin and S.-F. Xu.
\newblock Convergence analysis of structure-preserving doubling algorithms for
  {R}iccati-type matrix equations.
\newblock {\em SIAM J. Matrix Anal. Appl.}, 28(1):26--39, 2006.

\bibitem{lixu:2025}
Changli Liu and Jungong Xue.
\newblock Efficient computation of {i}ener-{H}opf factorization of
  {M}arkov-modulated {B}rownian motion.
\newblock {\em Math. Comp.}, 94(355):2367--2408, 2025.

\bibitem{maly:1989}
A.~N. Malyshev.
\newblock Computing invariant subspaces of a regular linear pencil of matrices.
\newblock {\em Siberian Math. J.}, 30:559--567, 1989.

\bibitem{mehr:1991}
V.~Mehrmann.
\newblock {\em The Autonomous Linear Quadratic Control Problem, {T}heory and
  {N}umerical {S}olution}.
\newblock Springer, 1991.

\bibitem{ngpo:2015}
G.~T. Nguyen and F.~Poloni.
\newblock Componentwise accurate fluid queue computations using doubling
  algorithms.
\newblock {\em Numer. Math.}, 130(4):763--792, 2015.

\bibitem{pals:1980}
T.~Pappas, A.~J. Laub, and N.~R. Sandell.
\newblock On the numerical solution of the discrete-time algebraic {R}iccati
  equation.
\newblock {\em IEEE Trans. Automat. Control}, 25:631--641, 1980.

\bibitem{polo:2010}
F.~Poloni.
\newblock {\em Algorithms for Quadratic Matrix and Vector Equations}.
\newblock PhD thesis, Scuola Normale Superiore, Pisa, Italy, 2010.

\bibitem{robl:2012}
D.~Rocca, Z.~Bai, R.-C. Li, and G.~Galli.
\newblock A block variational procedure for the iterative diagonalization of
  non-{Hermitian} random-phase approximation matrices.
\newblock {\em J. Chem. Phys.}, 136:034111, 2012.

\bibitem{shdy:2016}
M.~Shao, F.~H. {da Jornada}, C.~Yang, J.~Deslippe, and S.~G. Louie.
\newblock Structure preserving ing parallel algorithms for solving the
  {Bethe-Salpeter} eigenvalue problem.
\newblock {\em Linear Algebra Appl.}, 488:148--167, 2016.

\bibitem{suqu:2004}
Xiaobai Sun and Enrique~S. Quintana-Ort{\' i}.
\newblock Spectral division methods for block generalized {Schur}
  decompositions.
\newblock {\em Math. Comp.}, 73:1827--1847, 2004.

\bibitem{wawl:2012}
Wei-Guo Wang, Wei-Chao Wang, and Ren-Cang Li.
\newblock Alternating-directional doubling algorithm for {$M$}-matrix algebraic
  {R}iccati equations.
\newblock {\em SIAM J. Matrix Anal. Appl.}, 33(1):170--194, 2012.

\bibitem{wawl:2013}
Wei-Guo Wang, Wei-Chao Wang, and Ren-Cang Li.
\newblock Deflating irreducible singular {$M$}-matrix algebraic {R}iccati
  equations.
\newblock {\em Numer. Alg., Contr. Optim.}, 3:491--518, 2013.

\bibitem{xuzl:2015}
Hongteng Xu, Hongyuan Zha, Ren-Cang Li, and Mark~A. Davenport.
\newblock Active manifold learning via gershgorin circle guided sample
  selection.
\newblock In {\em AAAI Conference on Artificial Intelligence}. 2015.
\newblock Available at {\tt
  www.aaai.org/ocs/index.php/AAAI/AAAI15/paper/view/9479}.

\bibitem{xuli:2017}
Jungong Xue and Ren-Cang Li.
\newblock Highly accurate doubling algorithms for ${M}$-matrix algebraic
  {Riccati} equations.
\newblock {\em Numer. Math.}, 135(3):733--767, 2017.

\end{thebibliography}
\def\noopsort#1{}\def\l{\char32l}\def\v#1{{\accent20 #1}}
  \let\^^_=\v\def\hbk{hardback}\def\pbk{paperback}

}

\end{document}